\documentclass[11pt,a4paper]{article}
 \pdfoutput=1
\usepackage{amsmath, amssymb, latexsym, multicol, bm, placeins}

\usepackage{tikz}
\usetikzlibrary{calc,intersections,through,backgrounds}
\usepackage{tkz-euclide}
\usetikzlibrary{decorations.pathreplacing,calligraphy}
\usetikzlibrary{shapes.geometric}

\usepackage{ulem}
\usepackage{color}

\usepackage{algpseudocode,algorithm,algorithmicx}

\usepackage{enumerate}
\usepackage{geometry}
\usepackage{epstopdf}
\usepackage[percent]{overpic}
\usepackage{subcaption}

\usepackage{amsmath,amssymb,amsthm}

\usepackage{authblk}

\makeatletter\@addtoreset{equation}{section}\makeatother

\definecolor{ao}{rgb}{0.0, 0.5, 0.0}

\newcommand{\jump}[1]{\lbrack\!\lbrack\,#1\,\rbrack\!\rbrack}

\newtheorem{theorem}{Theorem}[section]
\newtheorem{lemma}[theorem]{Lemma}
\newtheorem{proposition}[theorem]{Proposition}
\newtheorem{remark}[theorem]{Remark}
\newtheorem{corollary}[theorem]{Corollary}

\newtheorem{assumption}[theorem]{Assumption}
\newtheorem{definition}[theorem]{Definition}

\newcommand{\mesh}{\mathcal{T}}

\newcommand{\meshs}{{\mesh_*}}

\newcommand{\edge}{\mathcal{E}}
\newcommand{\edgeE}{{\mathcal{E}_E}}

\newcommand{\edges}{{\edge_*}}

\newcommand{\nodes}{\mathcal{N}}
\newcommand{\nodesE}{{\mathcal{N}_E}}

\newcommand{\Wmeshu}{\mathbb{W}_\mesh^1} 
\newcommand{\V}{\mathbb{V}}
\newcommand{\VV}{{\widetilde{\V}}}
\newcommand{\Vmesh}{\V_\mesh}
\newcommand{\VE}{{\V_E}}

\newcommand{\VdE}{{\V_{\partial E}}}
\newcommand{\Vmeshz}{\V_\mesh^0}

\newcommand{\Vmeshs}{\V_\meshs}

\newcommand{\cB}{{\mathcal{B}}}
\newcommand{\B}{a}
\newcommand{\Bmesh}{\mathcal{B}_\mesh}

\newcommand{\amesh}{a_\mesh}
\newcommand{\mmesh}{m_\mesh}
\newcommand{\aE}{a_E}
\newcommand{\mE}{m_E}
\newcommand{\Smesh}{S_\mesh}
\newcommand{\Smeshs}{S_\meshs}
\newcommand{\sE}{s_E}
\newcommand{\SE}{S_E}

\newcommand{\Pimesh}{{\Pi^\nabla_\mesh}}
\newcommand{\PiE}{{\Pi^\nabla_E}}
\newcommand{\PiEi}{{\Pi^\nabla_{E_i}}}

\newcommand{\Imesh}{\mathcal{I}_\mesh}
\newcommand{\IE}{\mathcal{I}_E}
\newcommand{\Imeshz}{\mathcal{I}_\mesh^0 }

\newcommand{\Pimeshz}{{\Pi^0_\mesh}}
\newcommand{\PiEz}{{\Pi^0_E}}

\newcommand{\umesh}{u_\mesh}
\newcommand{\umeshs}{u_\meshs}

\newcommand{\vmesh}{v_\mesh}

\newcommand{\etamesh}{\eta_\mesh}
\newcommand{\etameshs}{\eta_\meshs}
\newcommand{\rmesh}{r_\mesh}
\newcommand{\jmesh}{j_\mesh}

\newcommand{\dataoscmesh}{{\Psi}_\mesh}

\newcommand{\vvvert}{|\!|\!|}

\newcommand{\data}{\mathcal{D}}
\newcommand{\appdata}{\widehat{\mathcal{D}}}

\newenvironment{algotab}
{\par\begin{samepage}%
\begin{tabbing}\ttfamily%
 \hspace*{5mm}\=\hspace{3ex}\=\hspace{3ex}\=\hspace{3ex}\=\hspace{3ex}%
\=\hspace{3ex}\=\hspace{3ex}\=\hspace{3ex}\=\hspace{3ex}\kill}%
{\end{tabbing}\end{samepage}}

\title{Adaptive VEM: Stabilization-Free A Posteriori Error Analysis and Contraction Property}

\author[1]{L. Beir\~ao da Veiga \thanks{lourenco.beirao@unimib.it}}

\author[2]{C. Canuto \thanks{claudio.canuto@polito.it}}

\author[3]{R. H. Nochetto \thanks{rhn@math.umd.edu}}

\author[4]{G. Vacca \thanks{giuseppe.vacca@uniba.it}}

\author[5]{M. Verani \thanks{marco.verani@polimi.it}}

\affil[1]{Dipartimento di Matematica e Applicazioni,
Universit\`a degli Studi di Milano Bicocca,
Via Roberto Cozzi 55 - 20125 Milano, Italy}

\affil[2]{Dipartimento di Scienze Matematiche G.L. Lagrange,
 Politecnico di Torino, 
 Corso Duca degli Abruzzi 24 - 10129 Torino, Italy}
  
\affil[3]{Department of Mathematics and Institute
for Physical Science and Technology, 
University of Maryland, 
College Park - 20742, MD, USA}

\affil[4]{Dipartimento di Matematica, 
Universit\`a degli Studi di Bari, 
Via Edoardo Orabona 4  - 70125 Bari, Italy}

\affil[5]{MOX-Laboratory for Modeling and Scientific Computing,  Dipartimento di Matematica, Politecnico di Milano, 
Piazza Leonardo da Vinci 32 - 20133 Milano, Italy}

\title{Adaptive VEM: Stabilization-Free A Posteriori Error Analysis and Contraction Property}

\begin{document}
\maketitle

\begin{abstract}
In the present paper we initiate the challenging task of building a mathematically sound theory for Adaptive Virtual Element Methods (AVEMs). Among the realm of polygonal meshes, we restrict our analysis to triangular meshes with hanging nodes in 2d -- the simplest meshes with a systematic refinement procedure that preserves shape regularity and optimal complexity. A major challenge in the a posteriori error analysis of AVEMs is the presence of the stabilization term, which is of the same order as the residual-type error estimator but prevents the equivalence of the latter with the energy error.
Under the assumption that any chain of recursively created hanging nodes has uniformly bounded length, we show that the stabilization term can be made arbitrarily small relative to the error estimator provided the stabilization parameter of the scheme is sufficiently large. This quantitative estimate leads to stabilization-free upper and lower a posteriori bounds for the energy error. This novel and crucial property of VEMs hinges on the largest subspace of continuous piecewise linear functions and the delicate interplay between its coarser scales and the finer ones of the VEM space. An important consequence for piecewise constant data is a contraction property between consecutive loops of AVEMs, which we also prove. Our results apply to $H^1$-conforming (lowest order) VEMs of any kind, including the classical and enhanced VEMs.
\end{abstract}

%

\section{Introduction}\label{sec:intro}

A posteriori error estimates have become over the last four decades an indispensable tool
for realistic and intricate computations in both science and engineering. They are computable
quantities in terms of the discrete solution and data that control the approximation error,
typically in the energy norm $|\cdot|_{1,\Omega}$, from both above and below. Such estimators can be split into local
contributions and exploited to drive adaptive algorithms that equidistribute the approximation
error and so the computational effort. This has made simulation of complex phenomena accessible
with modest computational resources.

Practice and theory of a posteriori error analysis and ensuing adaptive algorithms is a
relatively mature research area for linear elliptic partial differential equations (PDEs), especially
with the finite element method (FEM). They give rise to the so-called adaptive FEMs (or
AFEMs for short). We refer to the survey papers \cite{NSV:09,NochettoVeeser:12} for an
account of the state-of-the-art on the following two fundamental and complementary aspects of this endeavor:
\begin{enumerate}[$\bullet$]
\item
  {\it Derivation of a posteriori error estimates:} Residual estimators are the first and
  simplest estimators; see Babu\v ska and Miller \cite{BabuskaMiller:87} and \cite{BabuskaRheinboldt:1978}.
  They exhibit upper
  and lower bounds (up to data oscillation) with stability constants of moderate size that depend on
  interpolation constants and thus on the geometry of the underlying meshes. Other estimators
  have been developed over the years with the goal of getting more precise or even constant
  free estimates; examples are local problems on elements \cite{BankWeiser:1985} and stars
  \cite{MorinNochettoSiebert:2003}, gradient
  recovery estimators \cite{ZZ:1987,Rodriguez:1994}, and flux equilibration estimators
  \cite{BraessSchoeberl:2008,BPS:2009}. It turns out that
  they are all equivalent to the energy error. In addition, low order approximation
  \cite{BabuskaMiller:87,BabuskaRheinboldt:1978,BankWeiser:1985}
  has evolved into high-order methods such as the $hp$-FEM \cite{MelenkWohlmuth:2001}.

\item
  {\it Proof of convergence and optimatity of AFEMs:} The study of adaptive loops of the form
  \begin{equation}\label{eq:SEMR}
    \texttt{SOLVE}
    \quad\longrightarrow\quad
    \texttt{ESTIMATE}
    \quad\longrightarrow\quad
    \texttt{MARK}
    \quad\longrightarrow\quad
    \texttt{REFINE}
  \end{equation}
  is an oustanding problem in numerical analysis of PDEs. The issue at stake is that discrete
  solutions at different level of resolution, typically on nested meshes, must be compared.
  This, in conjunction with the upper bound and D\"orfler marking, yields a contraction
  property for every step of \eqref{eq:SEMR}. Optimality entails further understanding of how
  the a posteriori estimator changes with the discrete solution and mesh refinement, as well
  as whether it can be localized to the refined region and yet provide control of the error between
  discrete solutions. This, combined with marking minimal sets and complexity estimates
  for mesh refinement strategies, leads to optimality of AFEM in the sense that the energy
  error decreases with optimal rate (up to a multiplicative constant) in terms of degrees
  of freedom. Theory for fixed polynomial degree \cite{NSV:09,NochettoVeeser:12} extends
  somewhat to variable order \cite{CNSV:2016,CNSV:2017}.
\end{enumerate}

\medskip

{\bf Virtual element methods (VEMs).}  They are a relatively new discretization paradigm which allows for general
polytopal meshes, any polynomial degree, and yet conforming $H^1$-approximations for second order
problems \cite{volley,autostop}. This geometric flexibility is very useful in some applications (a few examples being \cite{Berrone-appl, Paulino-appl, Artioli-appl, Wriggers-appl, leaflet}),
but comes at a price for the design and practical use of adaptive VEMs (or AVEMs for short).

Two natural, but yet open, questions arise:
\begin{enumerate}[$\bullet$]
\item
  {\it Procedure:} Is it possible to systematically refine general polytopes and preserve shape regularity?
  Beir\~ao da Veiga and Manzini \cite{daVeigaManzini:2015} proposed a first residual based error estimator and introduced a simple refinement rule for any  convex polygon. The rule is to connect the barycenter of the polygon with mid-points of edges, where the word ``edge'' needs to be intended disregarding the existence of hanging nodes generated during the refinement procedure. It is not difficult to check that such procedure guarantees to generate a sequence of shape regular meshes. More sophisticated practical refinement procedures, which guarantee shape regularity, have been recently proposed (see, e.g., \cite{Berrone:2022} and \cite{Antonietti-Manuzzi:2022}. Note that shape regularity is critical to have robust interpolation estimates regardless of the resolution level.
 
\item
  {\it Complexity:} Is it possible to prove that the number of
  elements generated by {\tt REFINE} is proportional to the number of elements marked collectively
  for refinement by {\tt MARK}? On the one hand, the answer is affirmative if the refinement is completely local.
  This in turn comes at the expense of unlimited growth of nodes per element, which may be hard to handle
  computationally and does not add enhanced accuracy. On the other hand, restricting the number of
  hanging nodes per edge makes the question very delicate, and generally false
  at every step of \eqref{eq:SEMR}. This is altogether crucial
  to show that iteration of \eqref{eq:SEMR} leads to an error decay comparable with the best
  approximation in terms of degrees of freedom.
\end{enumerate}

The development of a posteriori error estimates for VEMs mimics that of FEMs.
The estimator of residual type most relevant to us is that proposed by Cangiani et al. \cite{Cangiani-etal:2017}.
The upper and lower bounds derived in \cite{Cangiani-etal:2017} involve stabilization terms, but
are valid for arbitrary polygonal elements,
any polynomial degree, and general (coercive) second order operators with variable coefficients. 
Estimators for the $hp$-version of VEMs are developed in \cite{BeiraoManziniMascotto:2019},
for anisotropic VEMs in \cite{Antonietti-etal:2020}, and for mixed VEMs in
\cite{CangianiMunar:2019,MunarSequeiraFilander:2020}. Gradient recovery estimators are derived in
\cite{ChiBeiraoPaulino:2019} whereas those based on equilibrated fluxes are studied in
\cite{DassiGedickeMascotto:2021a,DassiGedickeMascotto:2021b}. 

A key constituent of VEMs to deal with general polytopes is stabilization
(although it is possible to design VEMs that do not require any stabilization, see \cite{Berrone-Borio-Marcon:2021}, the analysis is still at its infancy). 
Even though the role of stabilization is clear and precise in the
a priori error analysis of VEMs to make the discrete bilinear form coercive,
it remains elusive in the a posteriori counterpart. The main
contribution of this paper is to show that such a role is not vital.

\medskip
{\bf Setting.}
Our approach to adaptivity for VEMs is twofold. In this paper, we consider residual estimators,
derive {\it stabilization-free} a posteriori upper and lower bounds, and prove a contraction property
for AVEMs for piecewise constant data. The removal of the
stabilization term is essential to study \eqref{eq:SEMR} for variable data and thereby design a two-step
AVEM and prove its convergence and
optimal complexity, which will be accomplished in \cite{BCNVV:22}. To achieve these goals, we put ourselves
in the simplest possible but relevant setting consisting of the following four simplifying assumptions.
\begin{enumerate}[$\bullet$]
\item
  {\it Meshes:} We consider partitions $\mesh$ of a  polygonal domain $\Omega$ for $d=2$ made of elements
  $E$, which are triangles with hanging nodes and refined via the newest vertex bisection (NVB). In contrast to
  FEMs, the hanging nodes carry degrees of freedom in the VEM philosophy. The NVB dictates a unique infinite binary
  tree with roots in the initial mesh, in which every triangle $E$ is uniquely determined and traceable back
  to the roots. This geometric rigidity is crucial to optimal complexity (see \cite{BDD:04,NochettoVeeser:12}
  for $d=2$ and \cite{Stevenson:08,NSV:09} for $d>2$),
  and plays an essential role in the study of AFEMs \cite{NSV:09,NochettoVeeser:12} as well as in the
  sequel paper \cite{BCNVV:22} on AVEMs. Quadrilateral partitions with hanging nodes are practical in the VEM
  context and amenable to analysis, but general polygonal elements are currently out of reach.

\item
  {\it Polynomial degree:} We restrict our analysis to piecewise linear elements on the skeleton $\edge$
  of $\mesh$. This is not just for convenience to simplify the presentation. It enters in the notion of
  global index (see Definition \ref{def:node-index}) and the scaled Poincar\'e inequality (see
  Proposition \ref{prop:scaledPoincare}). They lead to the stabilization-free a posteriori error estimators
  discussed below. Extensions to higher polynomial degrees appear feasible and are currently underway.

\item
  {\it Global index:} This is a natural number $\lambda(\bm{x})$ that characterizes the level of a
  hanging node $\bm{x}$ generated by successive NVB of an element $E\in\mesh$. We make the key assumption that,
  for all hanging nodes $\bm{x}$ of all meshes $\mesh$, there exists a universal constant $\Lambda>0$ such that
  \begin{equation}\label{eq:global-index-intro}
    \lambda(\bm{x}) \le \Lambda.
  \end{equation}
  This novel notion has profound geometric consequences.
  First, any chain of recursively created hanging nodes has uniformly bounded length, second a side of a triangle $E$ can contain at most $2^\Lambda-1$
  hanging nodes, and third any edge $e$ of $E$ has a size comparable with that of $E$, namely $h_e \simeq h_E$.
  These properties are instrumental to prove the scaled Poincar\'e inequality, but do not prevent deep
  refinement in the interior of $E$.

\item
  {\it Data:} We consider $\Omega$ to be polygonal and the symmetric elliptic PDE
  \begin{equation}\label{eq:pde-intro}
   - \nabla\cdot \left( A \nabla u \right) + c u = f \quad  \text{ in } \Omega,
  \end{equation}
  with piecewise constant data $\data = (A,c,f)$ and vanishing Dirichlet boundary condition.
  This choice simplifies the presentation and avoids approximation of $\Omega$
  and data oscillation terms. We extend the efficiency and reliability
  estimates to variable data $\data$ in Section \ref{sec:extensions}, and postpone the convergence (and complexity) analysis for general data to the forthcoming article \cite{BCNVV:22}. 
  However, piecewise constant data play a
  fundamental role in the design of AVEMs in \cite{BCNVV:22} because we approximate adaptively $\data$
  to a desired level of accuracy before we reduce the PDE error to a comparable level. Therefore,
  the analysis of \cite{BCNVV:22} hinges on having $\data$ piecewise constant when dealing with a
  posteriori error estimators for \eqref{eq:pde-intro}.
\end{enumerate}

Our approach is a first attempt to develop mathematically sound AVEMs. This simplest setting
serves to highlight similarities and striking differences with respect to AFEMs.

\medskip
{\bf Contributions.}
We now describe our main contributions. Let $\V_\mesh$ be a general $H^1$-conforming (lowest order) VEM space over $\mesh$ (see for instance \cite{volley,projectors}), which entails a suitable continuous extension of piecewise linear functions on the skeleton $\edge$ to $\Omega$. 
Let $a_\mesh$ and $m_\mesh$ be the VEM bilinear forms
corresponding to the second-order and zeroth-order terms in \eqref{eq:pde-intro}, and let ${\cal F}_{\mesh}$ be the linear form corresponding to the forcing term; we refer to
Section \ref{sec:discrete-pb} for details. If $S_\mesh$ denotes the stabilization term, then the discrete
solution $u_\mesh\in\V_\mesh$ satisfies
\begin{equation}\label{eq:discrete-prob}
  a_\mesh(u_\mesh,v) + m_\mesh(u_\mesh,v) + \gamma S_\mesh(u_\mesh,v) = {\cal F}_{\mesh} (v)
  \quad\forall \, v\in\V_\mesh.
\end{equation}
Problem \eqref{eq:discrete-prob} admits a unique solution $u_\mesh$ for all values of the
stabilization parameter $\gamma>0$. Let $\eta_\mesh(u_\mesh,\data)$ be the residual a posteriori
error estimator for piecewise constant data $\data$ studied in Section \ref{sec:aposteriori}.
Our global a posteriori error estimates read as follows:
\begin{equation}\label{eq:aposteriori}
c_{apost} \eta_\mesh^2(u_\mesh,\data) - S_\mesh(u_\mesh,u_\mesh) \le
|u-u_\mesh|_{1,\Omega}^2 \le C_{apost} \big(\eta_\mesh^2(u_\mesh,\data) + S_\mesh(u_\mesh,u_\mesh)  \big),
\end{equation}
for suitable constants $c_{apost}<C_{apost}$; 
see Proposition \ref{prop:aposteriori} and Corollary \ref{global-lower-bound}. We stress that, in
contrast to \cite{Cangiani-etal:2017}, the
stabilization term $S_\mesh$ appears without the constant $\gamma$ in \eqref{eq:aposteriori}. 
One of our main results is Proposition \ref{prop:bound-ST}: there exists a constant $C_B>0$ depending
on $\Lambda$ but independent of $\mesh, u_\mesh$ and $\gamma$ such that
\begin{equation}\label{eq:bound-ST-intro}
\gamma^2 S_\mesh(u_\mesh,u_\mesh) \le C_B \, \eta_\mesh^2(u_\mesh,\data).
\end{equation}
Computations reveal that \eqref{eq:bound-ST-intro} is sharp provided the number of hanging
nodes is large relative to the total, and confirm that the stabilization term $S_\mesh(u_\mesh,u_\mesh)$
is of the same asymptotic order as the estimator $\eta_\mesh^2(u_\mesh,\data)$;
see details in Section \ref{sec:numerics-1}.
The significance of \eqref{eq:bound-ST-intro} is that it gives the quantitative condition
$\gamma^2 > C_B/c_{apost}$ on $\gamma$ for $S_\mesh(u_\mesh,u_\mesh)$ to be absorbed within
$\eta_\mesh^2(u_\mesh,\data)$ and, combined with \eqref{eq:aposteriori}, yields the
{\it stabilization-free a posteriori error estimates}
\begin{equation}\label{eq:stab-free}
  \Big(c_{apost} - \frac{C_B}{\gamma^2}  \Big) \eta_\mesh^2(u_\mesh,\data)
  \le |u- u_\mesh|_{1,\Omega}^2
    \le C_{apost} \Big( 1 + \frac{C_B}{\gamma^2}  \Big) \eta_\mesh^2(u_\mesh,\data).
\end{equation}
In contrast to a priori error estimates,
this new estimate sheds light on the secondary role played by stabilization in a posteriori
error analysis. Moreover, the relation between discrete solutions on
different meshes involves the stabilization terms on each mesh, which complicates the theory of
\eqref{eq:SEMR}. Applying once again \eqref{eq:bound-ST-intro}, we prove a contraction property
for AVEMs of the form \eqref{eq:SEMR} for piecewise constant data,
with $\gamma$ chosen perhaps a bit larger than in \eqref{eq:stab-free}.
Precisely, we show that a suitable combination of energy error and residual contracts at each iteration of \eqref{eq:SEMR}: for some $\alpha \in (0,1)$ and $\beta>0$, there holds
\begin{equation}
\vvvert u - \umeshs \vvvert^2+\beta \etameshs^2(\umeshs,\data) \leq \alpha \left( \vvvert u - \umesh \vvvert^2+\beta \etamesh^2(\umesh,\data) \right)\,,
\end{equation}
where $\meshs$ denotes the refinement of $\mesh$ produced by AVEM.
We use this framework as a building block in the construction and analysis of AVEMs
for variable data in \cite{BCNVV:22}.

We conclude this introduction with a heuristic explanation of the idea behind \eqref{eq:bound-ST-intro}.
It is inspired by the analysis of adaptive discontinuous Galerkin methods (dG) by Karakashian and Pascal
\cite{KarakashianPascal:03,KarakashianPascal:07} and Bonito and Nochetto \cite{BonitoNochetto:10}. It
turns out that to control the penalty term of dG, which is also of the same order as the estimator, a suitable
estimate involving the penalty parameter $\gamma$ similar to \eqref{eq:bound-ST-intro} is derived in
\cite{KarakashianPascal:07} to prove convergence and is further exploited in \cite{BonitoNochetto:10} to
show convergence under minimal regularity and optimality of \eqref{eq:SEMR}. This hinges on the 
subspace $\V_\mesh^0$ of all continuous, piecewise linear functions over $\mesh$. 
It turns out that
the stabilization term $S_\mesh$ vanishes on the subspace $\V_\mesh^0$, namely $S_\mesh(w,v)=0$
for all $v\in\V_\mesh^0,w\in\V_\mesh$. The delicate issue at stake is to relate the coarser scales of $\V_\mesh^0$ with the
finer ones of $\V_\mesh$, which is made possible by the restriction \eqref{eq:global-index-intro} on the global
index $\lambda$. This leads to the following two fundamental and novel estimates for VEMs.

To state these results, hereafter we will make use of the $\lesssim$ and $\simeq$ symbols to denote bounds up to a constant that is independent of $\mesh$ and any other critical parameter; specifications will be given when needed.
The first key estimate is the following scaled Poincar\'e inequality proved in Section \ref{sec:poincare}:
\begin{equation}\label{eq:scaled-Poincare-intro}
  \sum_{E\in\mesh} h_E^{-2} \|v\|_{0,E}^2 \lesssim |v|_{1,\Omega}^2
  \quad\forall \, v\in\V_\mesh
\end{equation}
so that $v$ vanishes at all nodes of $\mesh^0$, the so-called {\it proper nodes}.
The second key estimate relates the interpolation errors in $\V_\mesh^0$ and $\V_\mesh$ due to
the corresponding piecewise linear Lagrange interpolation operators $\Imeshz$ and $\Imesh$:
\begin{equation}\label{eq:Lagrange-interpolation} 
  | v-\Imeshz v|_{1,\Omega}  \, \lesssim \, | v-\Imesh v|_{1,\mesh} \qquad \forall v \in \Vmesh \,;
\end{equation}
note that in general $\Imesh v$ is discontinuous in $\Omega$ and $|\cdot|_{1,\mesh}$ stands for the
broken $H^1$-seminorm.
This estimate is proved in Section \ref{sec:bound-ST} along with \eqref{eq:bound-ST-intro}. The constants
hidden in both \eqref{eq:scaled-Poincare-intro} and \eqref{eq:Lagrange-interpolation} depend on
$\Lambda$ in \eqref{eq:global-index-intro} and blow-up as $\Lambda\to\infty$.
This extends, upon suitably modifying the VEM, to rectangular elements but not
to general polygons.

\medskip
{\bf Outline.}
The paper is organized as follows. In Section \ref{sec:ingr} we introduce the bilinear forms associated
with \eqref{eq:pde-intro} and the VEM discretization
with piecewise linear functions on the skeleton $\edge$. We also
discuss the notion of global index $\lambda$ and the main restriction \eqref{eq:global-index-intro}.
In Section \ref{sec:prep} we present some technical estimates such as \eqref{eq:scaled-Poincare-intro}.
The a posteriori error
analysis is carried out in Section \ref{sec:aposteriori}. Inequality \eqref{eq:scaled-Poincare-intro}
is instrumental to derive \eqref{eq:aposteriori} without the parameter $\gamma$, which
combined with \eqref{eq:bound-ST-intro} yields \eqref{eq:stab-free} immediately.
We postpone the proof of \eqref{eq:scaled-Poincare-intro} to Section \ref{sec:poincare} and those
of \eqref{eq:Lagrange-interpolation} and \eqref{eq:bound-ST-intro} to Section \ref{sec:bound-ST}. 
Inequality \eqref{eq:bound-ST-intro} is essential to study the effect of mesh refinements in the
a posteriori error estimator $\eta_\mesh(u_\mesh,\data)$ and the error $|u- u_\mesh|_{1,\Omega}$,
which altogether culminates with a proof of the contraction property of AVEM in Section \ref{sec:AVEM-pcwconstant}.
In Section \ref{sec:extensions}
we extend some of our estimates to variable coefficients. We conclude in
Section \ref{sec:experiments} with two insightful numerical experiments. The first one verifies
computationally that the dependence on $\gamma$ in \eqref{eq:bound-ST-intro} is generically sharp.
The second test, on a highly singular problem with checkerboard pattern, illustrates the ability of
AVEM to capture the local solution structure and compares the performance of AVEM with that of conforming AFEM.

\section{The problem and its discretization}
\label{sec:ingr}
In a polygonal domain $\Omega \subset \mathbb{R}^2$, consider the second-order Dirichlet boundary-value problem
\begin{equation}\label{eq:pde}
- \nabla\cdot \left(A \nabla u \right) + c u = f \quad  \text{ in } \Omega\,, \qquad u = 0 \quad \text{on } \partial\Omega\,,
\end{equation}
with data $\data=(A,c,f)$, where $A \in (L^\infty(\Omega))^{2\times 2}$ is symmetric and uniformly positive-definite in $\Omega$,
$c \in L^\infty(\Omega)$ is non-negative in $\Omega$, and $f \in L^2(\Omega)$. 
The variational formulation of the problem is
\begin{equation}\label{eq:pde-var}
u \in \V \ : \   \cB(u,v) = (f,v)_\Omega  \qquad \forall v \in \V \,,
\end{equation}
with $\V := H^1_0(\Omega)$ and 
$\cB(u,v):=\B(u,v) + m(u,v)$
where 
$$
\B(u,v):= \int_\Omega ( A \nabla u) \cdot \nabla v \, ,  \quad 
m(u,v) := \int_\Omega c \, u \, v 
$$ 
are the bilinear forms associated with Problem \eqref{eq:pde}.
Let $\vvvert \cdot \vvvert=\sqrt{\mathcal{B}(\cdot,\cdot)}$ be the energy norm, which satisfies  
\begin{equation}\label{norm:equiv}
c_\cB \vert v \vert_{1,\Omega}^2\leq \vvvert v \vvvert^2\leq c^\cB \vert v \vert^2_{1,\Omega}\quad \forall v \in \V\,,
\end{equation}
for suitable constants $0 <c_\cB \le c^\cB$.

\subsection{Triangulations}

In view of the adaptive discretization of the problem, let us fix an initial conforming partition $\mesh_0$ of $\overline{\Omega}$ made of triangular elements. Let us denote by $\mesh$ any refinement of $\mesh_0$ obtained by a finite number of successive newest-vertex element bisections; the triangulation $\mesh$ need not be conforming, since hanging nodes may be generated by the refinement. Let $\nodes$ denote the set of nodes of $\mesh$, i.e., the collection of all vertices of the triangles in $\mesh$; a node $z \in \nodes$ is {\it proper} if it is a vertex of all triangles containing it; otherwise, it is a {\it hanging node}. Thus,  ${\cal N}={\cal P}\cup{\cal H}$ is partitioned into the union of the set ${\cal P}$ of proper nodes and the set ${\cal H}$ of hanging nodes.

Given an element $E \in \mesh$, let $\nodesE$ be the set of nodes sitting on $\partial E$; it contains the three vertices and, possibly, some hanging node.  If the cardinality $|\nodesE|=3$, $E$ is said a {\it proper triangle} of $\mesh$; if $|\nodesE|>3$, then according to the VEM philosophy $E$ is not viewed as a triangle, but as a polygon having $|\nodesE|$ edges, some of which are placed consecutively on the same line. Any such edge $e \subset \partial E$ is called an {\it interface} (with the neighboring element, or with the exterior of the domain); the set of all edges of $E$ is denoted by $\edgeE$. Note that if $e \subset \partial E \cap \partial E'$, then it is an edge for both elements; consequently, it is meaningful to define the {\it skeleton} of the triangulation $\mesh$ by setting $\edge := \bigcup_{E \in \mesh} \edgeE$. Throughout the paper, we will set $h_E = |E|^{1/2}$ for an element and $h_e=|e|$ for an edge.

\medskip
The concept of {\it global index} of a hanging node will be crucial in the sequel. To define it, let us first observe that any hanging node $\bm{x} \in {\cal H}$ has been obtained through a newest-vertex bisection by halving an edge of a triangle in the preceding triangulation; denoting by  $\bm{x}', \bm{x}'' \in  {\cal N}$ the endpoints of such edge, let us set ${\cal B}(\bm{x})=\{\bm{x}', \bm{x}''\}$.

\begin{definition}[global index of a node]\label{def:node-index}
The global index $\lambda$ of a node $\bm{x} \in {\cal N}$  is recursively defined as follows:
\begin{itemize}
\item If $\bm{x} \in {\cal P}$, then set $\lambda(\bm{x}):=0$;
\item If $\bm{x} \in {\cal H}$,  with $\bm{x}', \bm{x}'' \in  {\cal B}(\bm{x})$, then set $\lambda(\bm{x}):=\max(\lambda(\bm{x}'), \lambda(\bm{x}''))+1$.
\end{itemize}
\end{definition}

\noindent
We require that the largest global index in $\mesh$, defined as 
$$
\Lambda_{\mesh} := \max_{\bm{x} \in {\cal N}} \lambda(\bm{x}) \,,
$$
does not blow-up when we take successive refinements of the initial triangulation $\mesh_0$.

\begin{definition}[$\Lambda$-admissible partitions]\label{def:Lambda-partitions}
Given a constant $\Lambda \geq 1$, a non-conforming partition $\mesh$ is said to be $\Lambda$-admissible  if 
$$
\Lambda_{\mesh} \leq \Lambda \,.
$$
\end{definition}

Starting from the initial conforming partition $\mesh_0$ (which is trivially $\Lambda$-admissible), all the subsequent non-conforming partitions generated by the module $\texttt{REFINE}$ introduced later on will remain $\Lambda$-admissible, due to the algorithm $\texttt{MAKE\_ADMISSIBLE}$ described in Section \ref{sec:experiments}.

\begin{remark}\label{rem:bound-global-index}
{\rm
The condition that $\mesh$ is $\Lambda$-admissible has the following implications for each element $E \in \mesh$:

i) If $L \subset \partial E$ is one of the three sides of the triangle $E$, then $L$ may contain at most $2^\Lambda-1$ hanging nodes; consequently, $|{\cal N}_E| \leq 3\cdot 2^\Lambda$.

ii) If $e \subset \partial E$ is any edge , then $h_e \simeq h_E$, where the hidden constants only depend on the shape of the initial triangulation and possibly on $\Lambda$.
}
\end{remark}
\noindent
Fig. \ref{fig:sample_element} displays examples that illustrate the dynamic change of
  $\lambda(\bm{x})$ for a given $\bm{x} \in {\cal N}$.
\begin{figure}[t!]
\begin{center}
\begin{overpic}[scale=0.25]{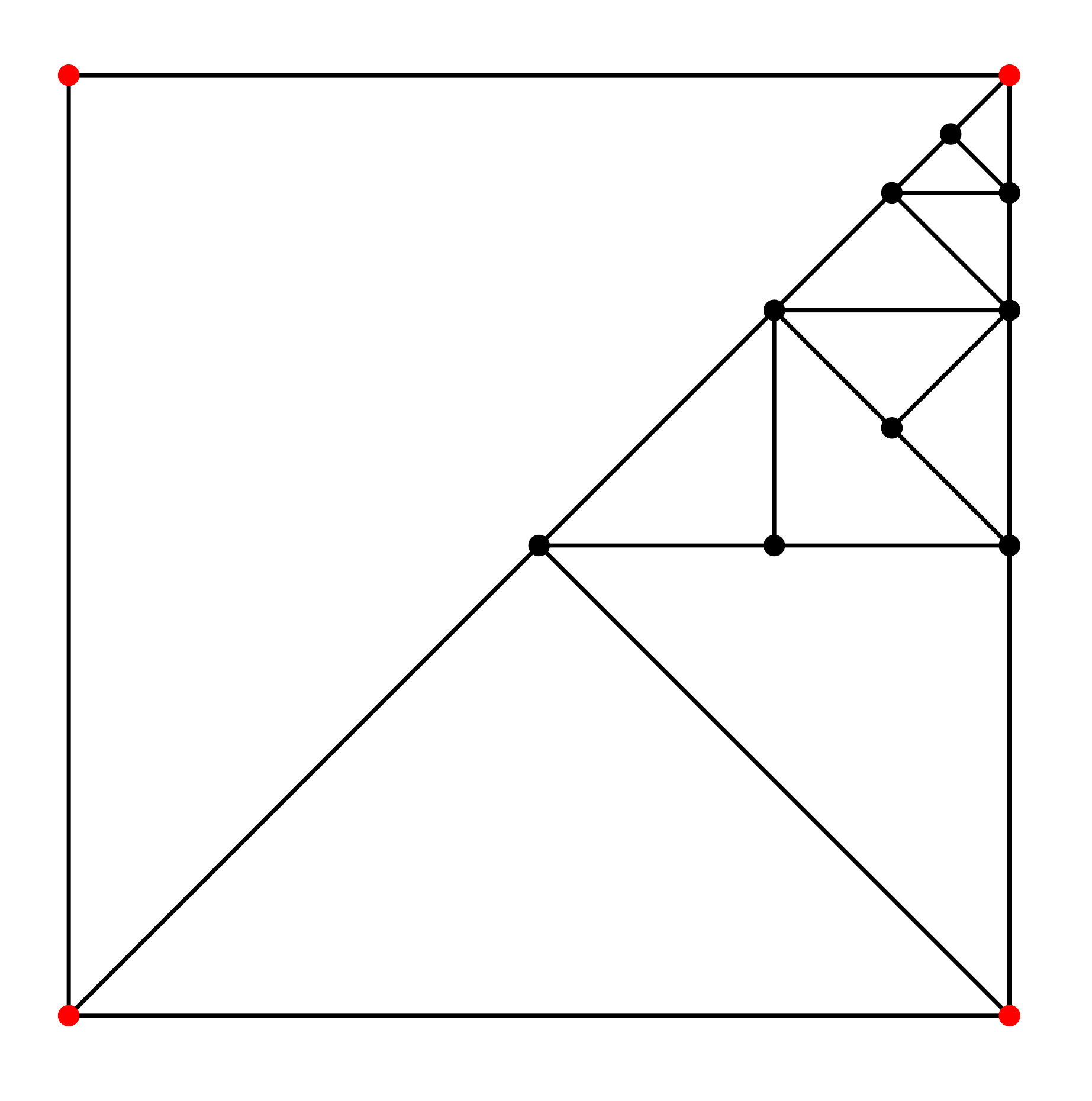}
\put( 2, 2){$0$}
\put( 2,95){$0$}
\put(92, 2){$0$}
\put(92,95){$0$}
\put(45,52){$1$}
\put(93,52){$1$}
\put(65,70){$2$}
\put(69,44){$2$}
\put(93,70){$2$}
\put(75,80){$3$}
\put(79,54){$3$}
\put(93,80){$3$}
\put(82,87){$4$}
\end{overpic}
\quad 
\begin{overpic}[scale=0.25]{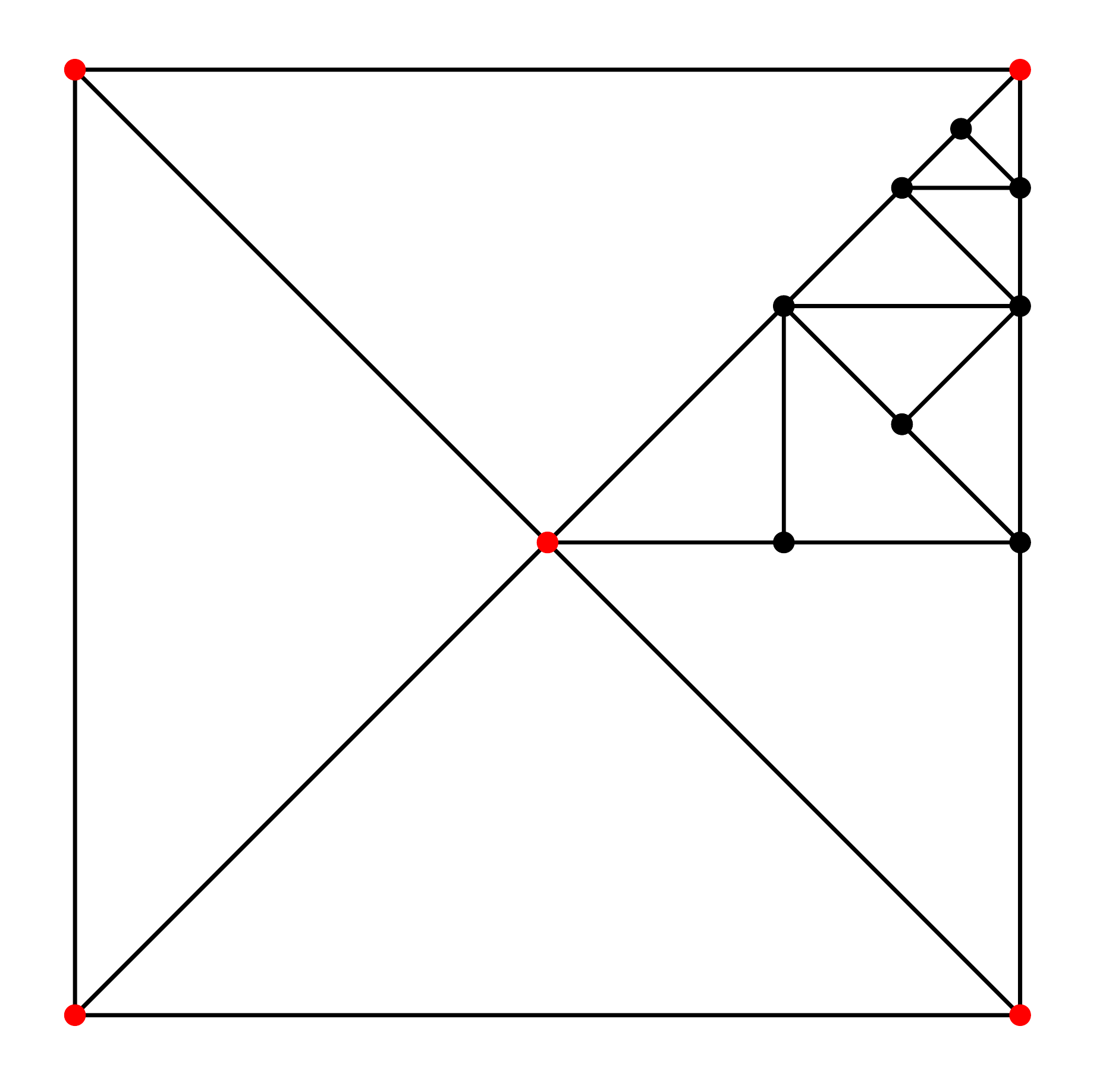}
\put( 2, 2){$0$}
\put( 2,95){$0$}
\put(94, 2){$0$}
\put(94,95){$0$}
\put(49,53){$0$}
\put(95,53){$1$}
\put(65,70){$1$}
\put(69,44){$2$}
\put(95,70){$2$}
\put(75,80){$2$}
\put(79,54){$2$}
\put(95,80){$3$}
\put(82,87){$3$}
\end{overpic}
\quad 
\begin{overpic}[scale=0.25]{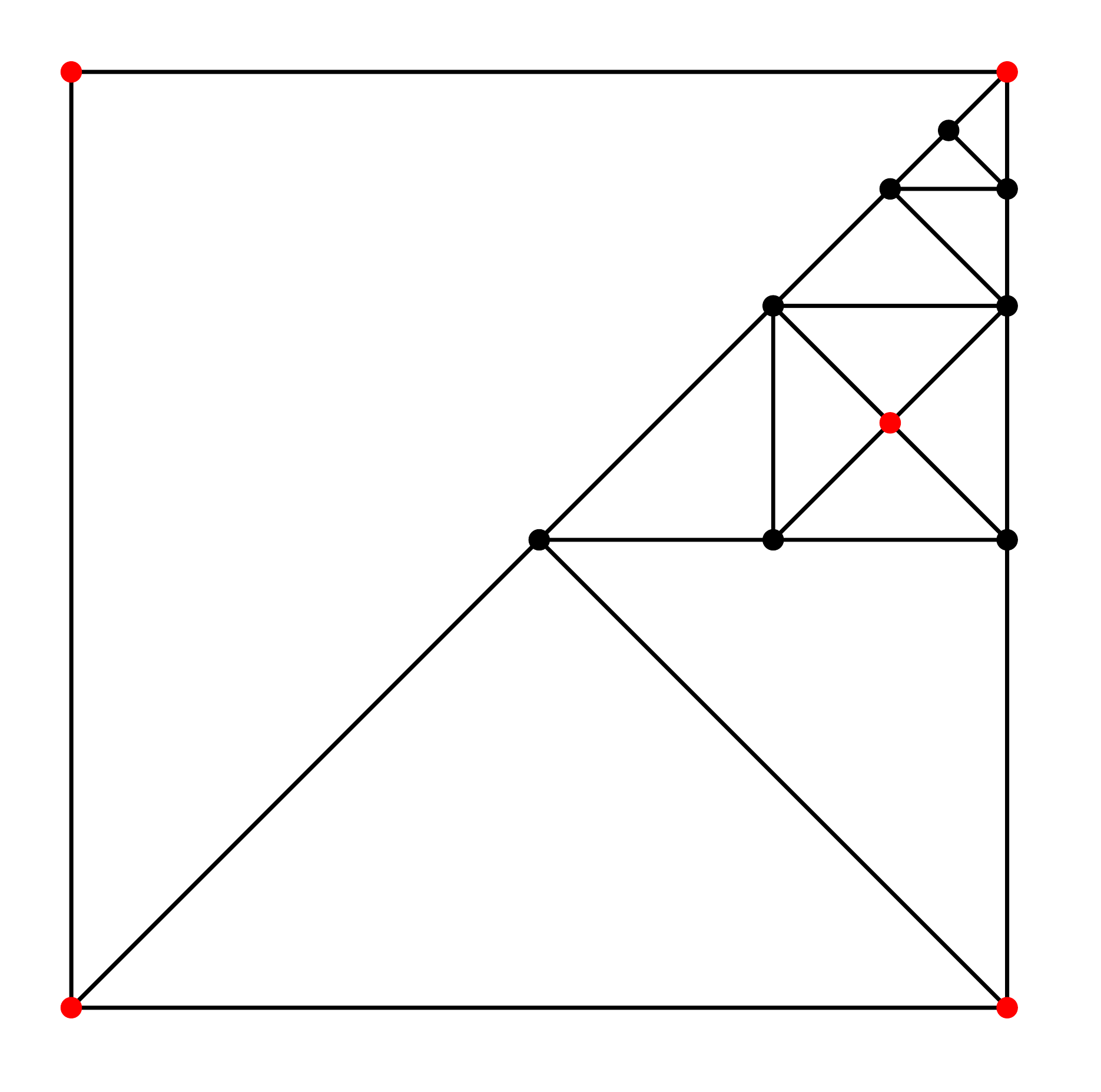}
\put( 2, 2){$0$}
\put( 2,95){$0$}
\put(92, 2){$0$}
\put(92,95){$0$}
\put(45,52){$1$}
\put(93,52){$1$}
\put(65,70){$2$}
\put(69,44){$2$}
\put(93,70){$2$}
\put(75,80){$3$}
\put(79,54){$0$}
\put(93,80){$3$}
\put(82,87){$4$}
\end{overpic}
\end{center}
\caption{Three examples of distributions of proper nodes (red) and hanging nodes (black), with associated global indices $\lambda$. The bisection added in the middle picture converts the centered node into proper, and induces nonlocal changes of global indices on chains associated with it. If $\Lambda=3$, then the leftmost mesh is not admissible and this procedure is instrumental to restore admissibility. The right picture illustrates the creation of a proper node without nonlocal effects on global indices.
}
\label{fig:sample_element}
\end{figure}
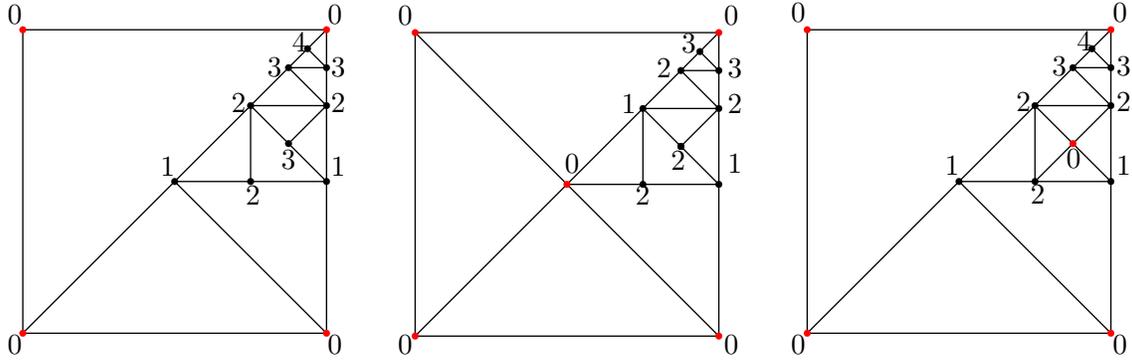

\subsection{VEM spaces and projectors}

In order to define a space of discrete functions in $\Omega$ associated with $\mesh$, for each element $E \in \mesh$ let us first introduce the space of continuous, piecewise affine functions on $\partial E$
\begin{equation}\label{eq:cond-VdE}
\VdE := \{ v \in {\cal C}^0(\partial E) : v_{|e} \in \mathbb{P}_1(e) \ \forall e \in \edgeE \}.
\end{equation}
Then, one needs to introduce a finite dimensional space $\VE \subset {\cal C}^0(E)$ satisfying the three following properties:
\begin{equation}\label{eq:cond-VE}
\text{dim}\, \VE = |\nodesE|\,, \qquad  \mathbb{P}_1(E) \subseteq \VE \,, \qquad  \tau_{\partial E}(\VE) = \VdE \,,
\end{equation}
where $ \tau_{\partial E}$ is the trace operator on the boundary of $E$. Obviously, if $E$ is a proper triangle, then $\VE = \mathbb{P}_1(E)$ is the usual space of affine functions in $E$; otherwise, note that a function in $\VE$ is uniquely identified by its trace on $\partial E$, but its  value in the interior of $E$ must be defined. 

The results of the present paper apply to any generic VEM space satisfying the conditions above and a suitable stability property introduced below. The well known examples are the basic VEM space of \cite{volley}
\begin{equation}\label{vem:basic} 
\VE := \big\{ v \in H^1(E) \ : \ v|_{\partial E} \in \VdE , \ \Delta v =0 \big\} 
\end{equation}
and the more advanced ``enhanced'' space from \cite{projectors,BBMR:2016} 
\begin{equation}\label{vem:choice:2} 
\VE := \big\{ v \in H^1(E) \ : \ v|_{\partial E} \in \VdE , \ \Delta v \in \mathbb{P}_1(E) \, ,
\int_E (v - \PiE v) q_1 = 0 \ \forall q_1 \in \mathbb{P}_1(E)
\big\} \,,
\end{equation}
where the projector $\PiE : H^1(E) \to \mathbb{P}_1(E)$ is defined by the conditions
\begin{equation}\label{eq:def-PinablaE}
(\nabla (v - \PiE v), \nabla w)_E = 0 \quad \forall w \in \mathbb{P}_1(E), \qquad  \int_{\partial E} (v-\PiE v) = 0 \, .
\end{equation}
It is easy to check that the above spaces are well defined and satisfy conditions \eqref{eq:cond-VE}.

Once the local spaces $\VE$ are defined, we introduce the global discrete space
\begin{equation}\label{eq:def-VT}
\Vmesh := \{ v \in \V :  \ v_{|E} \in \VE \ \ \forall E \in \mesh \}\,.
\end{equation} 
Note that functions in $\Vmesh$ are piecewise affine on the skeleton $\edge$, and indeed they are uniquely determined by their values therein and are globally continuous. Introducing the spaces of piecewise polynomial functions on $\mesh$
\begin{equation}\label{eq:def-Wk}
\mathbb{W}_\mesh^k :=\{ w \in L^2(\Omega) : w_{|E} \in \mathbb{P}_k(E) \ \ \forall E \in \mesh \}\,,  \qquad k=0,1\,,
\end{equation}
we also define the subspace of continuous, piecewise affine functions on $\mesh$
\begin{equation}\label{eq:def-VT0}
\Vmeshz := \Vmesh \cap \Wmeshu \,,
\end{equation} 
which will play a key role in the sequel. 

\medskip
The discretization of Problem \eqref{eq:pde} will involve certain projection operators, that we are going to define locally and then globally. To this end, let $\Pimesh : \Vmesh \to \Wmeshu$ be the operator that restricts to $\PiE$ on each $E \in \mesh$. Similarly, let $\IE : \VE \to \mathbb{P}_1(E)$ be the Lagrange interpolation operator at the vertices of $E$, and let $\Imesh : \Vmesh \to \Wmeshu$ be the Lagrange interpolation operator that restricts to $\IE$ on each $E \in \mesh$. Note that $\Pimesh v = v$ and $\Imesh v =v$ for all $v \in \Vmeshz$. Finally, let $\PiEz : L^2(E) \to \mathbb{P}_1(E)$, resp. $\Pimeshz : L^2(\Omega) \to \Wmeshu$, be the local, resp. global, $L^2$-orthogonal projection operator.

Using an integration by parts, it is easy to check that the $\PiE$ operator is directly computable from the boundary values of $v \in \VE$, and the same clearly holds for $\IE$. On the contrary, on a general VEM space the operator $\PiEz$ may be not computable. A notable exception is given by the space \eqref{vem:choice:2}, since by definition of the space it easily follows the following property:
\begin{equation}\label{enh-prop}
\textrm{For the local space choice \eqref{vem:choice:2} the operators } \PiEz \textrm{ and } \PiE \textrm{ coincide.}
\end{equation} 

\subsection{The discrete problem}\label{sec:discrete-pb}
 
Next, we introduce the discrete bilinear forms to be used in a Galerkin discretization of our problem. Here we make a simplifying assumption on the coefficients of the equation, in order to arrive at the core of our contribution without too much technical burden. In Sect. \ref{sec:extensions} we will discuss the general situation.

\begin{assumption}[coefficients and right-hand side of the equation]\label{ass:constant-coeff}
The data  $\data=(A,c,f)$ in \eqref{eq:pde} are constant in each element $E$ of $\mesh$; their values will be denoted by $(A_E,c_E,f_E)$.
\end{assumption}

Under this assumption, define  $\amesh, \mmesh : \Vmesh \times \Vmesh \to \mathbb{R}$ by
\begin{equation}\label{eq;def-aT}
\begin{aligned}
& \amesh(v,w) : = \sum_{E \in \mesh} \int_E  (A_E \nabla \PiE v) \cdot \nabla \PiE w =: \sum_{E \in \mesh} \aE(v,w) \, ,  \\
& \mmesh(v,w) : = \sum_{E \in \mesh} c_E \int_E \PiE v \, \PiE w =: \sum_{E \in \mesh} \mE(v,w) \,.
\end{aligned}
\end{equation}

Next, for any $E \in \mesh$, we introduce the symmetric bilinear form $\sE : \VE \times \VE \to \mathbb{R}$ 
\begin{equation}\label{eq:stab-dofidofi}
\sE(v,w) = \sum_{i=1}^{\nodesE} v({\bf x}_i) w({\bf x}_i) \ ,
\end{equation}
with $\{ {\bf x}_i \}_{i=1}^{\nodesE}$ denoting the vertexes of $E$. Such form will take the role of a stabilization in the numerical method; other choices for the stabilizing form are available in the literature and the results presented here easily extend to such cases.
We assume the following condition, stating that the local virtual spaces $\VE$ constitute a ``stable lifting'' of the element boundary values:
\begin{equation}\label{eq:stab-sE}
c_s | v |_{1,E}^2 \leq \sE(v,v) \leq C_s | v |_{1,E}^2 \qquad \forall v \in \VE  / {\mathbb R} \,,
\end{equation}
for constants $C_s \geq c_s > 0$ independent of $E$. 
For a proof of \eqref{eq:stab-sE} for some typical choices of $\VE$ and $\sE$ we refer to \cite{BLR:2017,brenner-sung:2018}; in particular,
the result holds for the choices \eqref{vem:basic} or \eqref{vem:choice:2}, and \eqref{eq:stab-dofidofi}.  
With the local form $s_E$ at hand, we define the local stabilizing form
\begin{equation}\label{eq:def-SE}
\SE(v,w) := \sE(v-\IE v, w-\IE w) \qquad \forall \, v,w\in \VE \,,
\end{equation}
as well as the global stabilizing form
\begin{equation}\label{eq:def-ST}
\Smesh(v,w) :=  \sum_{E \in \mesh}  \SE(v,w) \qquad \forall \, v,w\in \Vmesh \,.
\end{equation}
Note that from  \eqref{eq:stab-sE} we obtain
\begin{equation}\label{eq:stab-norm}
\Smesh(v,v) \simeq | v - \Imesh v|_{1,\mesh}^2 \qquad \forall v \in \Vmesh\,,
\end{equation}
where $ | \, \cdot \, |_{1,\mesh}$ denotes the broken $H^1$-seminorm over the mesh $\mesh$.

\medskip
Finally, for all $v,w \in \Vmesh$ we define the complete bilinear form 
\begin{equation}\label{eq:def-BT}
\Bmesh : \Vmesh \times \Vmesh \to \mathbb{R}\,, \qquad 
\Bmesh(v,w) := \amesh(v,w) + \gamma \Smesh(v,w) + \mmesh(v,w ) \,, 
\end{equation} 
where $\gamma \geq \gamma_0$ for some fixed $\gamma_0 >0$ is a stabilization constant independent of $\mesh$, which will be chosen later on.

The following properties are an easy consequence of the definitions and bounds outlined above.

\begin{lemma}[properties of bilinear forms]\label{prop:bilin-forms}
 i) For any $v \in \Vmesh$ and any $w \in \Vmeshz$, it holds
\begin{equation}\label{eq:propB1}
\amesh(v,w) = a(v,w)\,, \qquad \Smesh(v,w) = 0\,.  
\end{equation}
ii) The form $\Bmesh$ satisfies 
\begin{equation}\label{eq:propB2}
\beta |v |_{1,\Omega}^2 \leq \Bmesh(v,v),  \qquad |\Bmesh(v,w)| \leq B |v|_{1,\Omega} |w|_{1,\Omega}\,, \qquad \forall v,w \in \Vmesh\,,  
\end{equation}
with continuity and coercivity constants $B \geq \beta >0$ independent of the triangulation $\mesh$.  
The constant $\beta$ is a non-decreasing function of $\gamma$; hence, if we increase the value of $\gamma$ there is no risk of getting a vanishing $\beta$.
\end{lemma}
\begin{proof} Condition {\it i)}  follows easily recalling Assumption \ref{ass:constant-coeff} and noting that $\nabla \PiE v$ corresponds to the $L^2(E)$ projection of $\nabla v$ on the constant vector fields (living on E).
Condition { \it ii)} follows from \eqref{eq:stab-norm} with trivial arguments.
\end{proof}

Regarding the approximation of the loading term, we here consider 
\begin{equation}\label{discr-rhs}
{\cal F}_{\mesh} (v) := \sum_{E \in \mesh} f_E \int_E  \PiE v \qquad \forall v \in H^1_0(\Omega) \,. 
\end{equation}

We now have all the ingredients to set the Galerkin discretization of Problem \eqref{eq:pde}:  
{\it find $\umesh \in \Vmesh$ such that}
\begin{equation}\label{def-Galerkin}
\Bmesh(\umesh,v) = {\cal F}_{\mesh} (v) \qquad \forall v \in \Vmesh \,.
\end{equation}
Lemma \ref{prop:bilin-forms} guarantees existence, uniqueness and stability of the Galerkin solution.
We now establish a useful version of Galerkin orthogonality.

\begin{lemma}[Galerkin quasi-orthogonality]\label{L:Galerkin-orthogonality}
The solutions $u$ of \eqref{eq:pde-var} and $\umesh$ of \eqref{def-Galerkin} satisfy
\begin{equation}\label{eq:Galerkin-orthogonality}
  \cB(u-\umesh,v) = \sum_{E\in\mesh} c_E \int_E \big(\Pimesh \umesh - \umesh  \big) v
  \qquad\forall \, v\in \V_\mesh^0.
\end{equation}
In particular, the choice \eqref{vem:choice:2} of enhanced VEM space further implies
\begin{equation}\label{eq:PGO}
\cB(u-\umesh,v) = 0  \qquad\forall \, v\in \V_\mesh^0.
\end{equation}
\end{lemma}
\begin{proof}
The definitions \eqref{eq:pde-var} and \eqref{def-Galerkin} imply
\begin{equation*}
\cB(u-\umesh,v) = 
\big( (f,v)_\Omega - {\cal F}_\mesh(v) \big)+ 
\big( \Bmesh(u_\mesh,v)-\cB(u_\mesh,v) \big)
\qquad\forall \, v\in\V_\mesh.
\end{equation*}
If $v\in\V_\mesh^0$, then $\Pimesh v = v$ and ${\cal F}_\mesh(v) = (f,v)_\Omega$ according to
\eqref{discr-rhs}. On the other hand, \eqref{eq:propB1} yields
\begin{equation*}
  \Bmesh(\umesh,v)-\cB(\umesh,v) = m_\mesh(\umesh,v) - m(\umesh,v) =
  \sum_{E\in\mesh} c_E \int_E \big( \PiE \umesh - \umesh\big) v
  \qquad\forall \, v\in \V_\mesh^0,
\end{equation*}
which in turn leads to \eqref{eq:Galerkin-orthogonality}. Finally, for the choice \eqref{vem:choice:2} the right-hand
side of \eqref{eq:Galerkin-orthogonality} vanishes because $v\in \mathbb{P}_1(E)$ for all $E\in\mesh$.
This completes the proof.
\end{proof}

We finally remark that Galerkin quasi-orthogonality  easily implies the useful estimate
\begin{equation}\label{aux:PGO}
\vert \mathcal{B}(u-\umesh,v)\vert \lesssim S_\mesh (\umesh,\umesh)^{1/2} \vert v\vert_{1,\Omega}\qquad \forall v\in \Vmesh^0.
\end{equation}

\section{Preparatory results}\label{sec:prep}

In preparation for the subsequent a posteriori error analysis, we collect here some useful results involving functions in $\Vmesh$ or in $\Vmeshz$. 

The first result is a scaled Poincar\'e inequality in $\Vmesh$, which will be crucial in the sequel. We recall that ${\cal P}$ denotes the set of proper nodes in $\mesh$.

\begin{proposition}[scaled Poincar\'e inequality in $\Vmesh$]\label{prop:scaledPoincare}
There exists a constant $C_\Lambda>0$ depending on $\Lambda$ but independent of $\mesh$, such that 
\begin{equation}\label{eq:B0}
\sum_{E \in \mesh}  h_E^{-2} \Vert v \Vert_{0,E}^2 \leq C_\Lambda  | v |_{1,\Omega}^2 \qquad \forall v \in \Vmesh \text{ such that } v(\bm{x})=0 \ \forall \bm{x} \in {\cal P}.
\end{equation}
\end{proposition}
\noindent Due to the technical nature of the proof, we postpone it to Sect. \ref{sec:poincare}.

\medskip 
Next, we go back to the space $\Vmeshz$  introduced in \eqref{eq:def-VT0}. We note that functions in this space  are uniquely determined by their values at the proper nodes of $\mesh$. Indeed, a function $v \in  \Vmeshz$ is affine in each element of $\mesh$, hence, it is uniquely determined by its values at the three vertices of the element:  if the vertex $\bm{x}$ is a hanging node, with ${\cal B}(\bm{x})=\{\bm{x}', \bm{x}''\}$, then $v(\bm{x})=\frac12 \big(v(\bm{x}')+v(\bm{x}'')\big)$. 

In particular, $\Vmeshz$ is span by the Lagrange basis 
\begin{equation}\label{eq:Lagrange}
\forall \bm{x} \in {\cal P}: \qquad \psi_{\bm{x}} \in \Vmeshz \quad \text{satisfies} \quad 
\psi_{\bm{x}}(\bm{z})= \begin{cases}
1 & \text{if }  \bm{z}=\bm{x}\,, \\
0 & \text{if }  \bm{z} \in {\cal P}\setminus \{ \bm{x} \}
\end{cases}
\end{equation} 
(see Fig. \ref{fig:VT0} for an example of such a basis function, which looks different from the standard pyramidal basis functions on conforming meshes). 
\begin{figure}[t!]
\begin{center}
\begin{overpic}[scale=0.3]{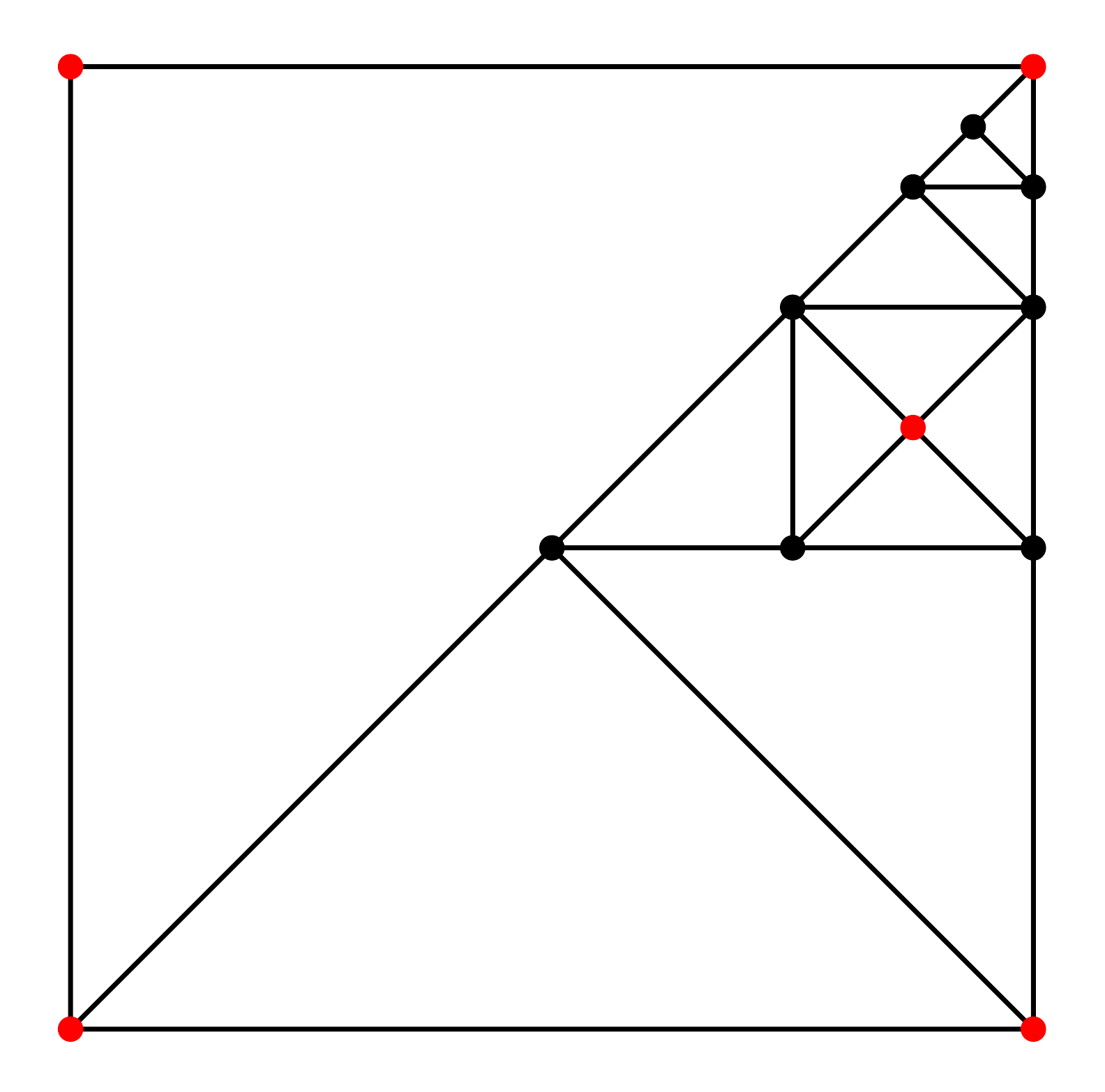}
\put(93, 95){$\bm{x}$}
\put(78, 53){$\bm{x}^*$}
\end{overpic}
\quad
\includegraphics[scale=0.3]{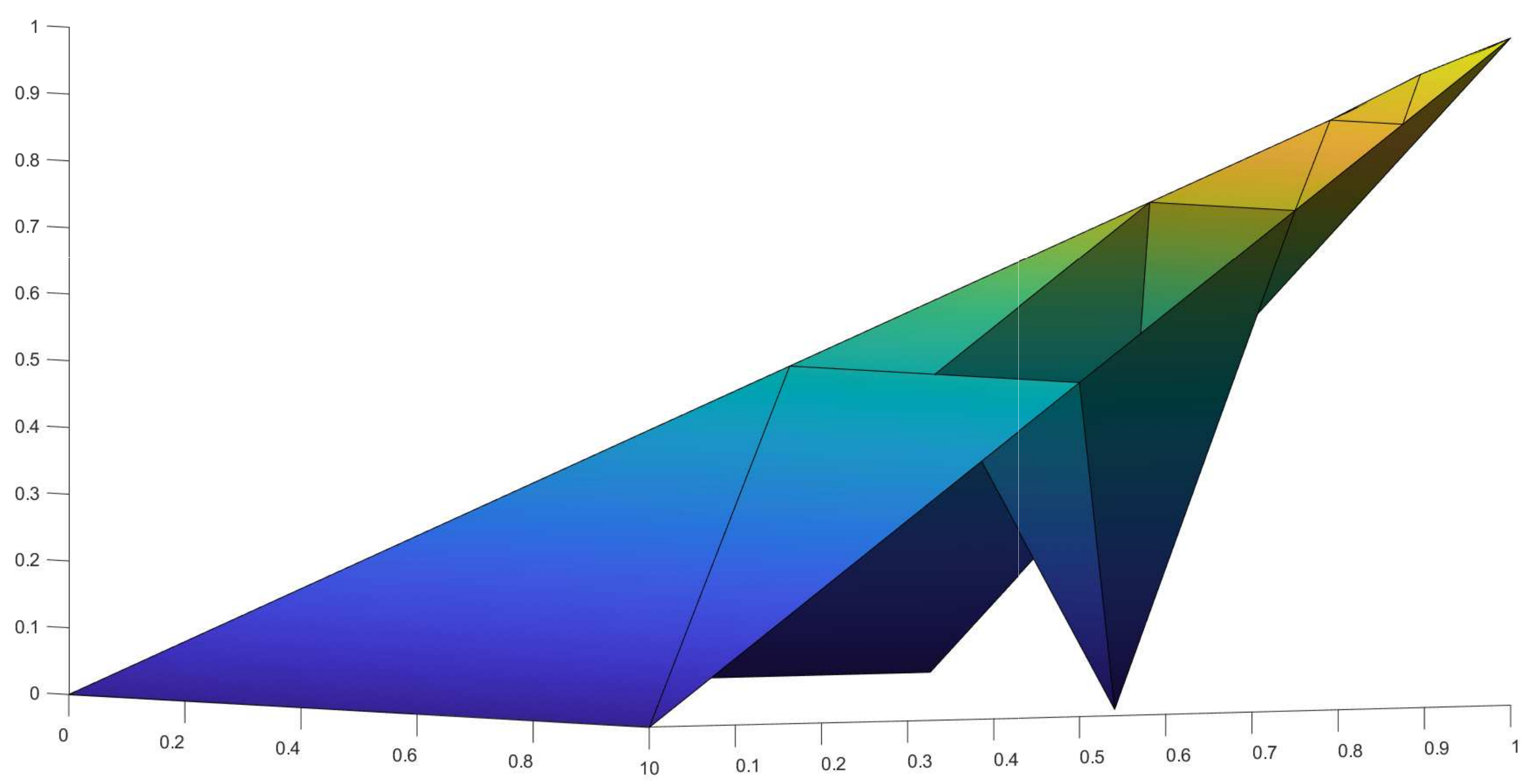}
\end{center}
\caption{{Left: detail of a mesh $\mesh$, in which red nodes $\bm{x}$ and $\bm{x^*}$ are proper nodes. 
      Right: basis function $\psi_{\bm{x}}\in\V_\mesh^0$; notice that $\psi_{\bm{x}}(\bm{x}^*) = 0$ and the basis function
      $\psi_{\bm{x}^*}\in\V_\mesh^0$ is the usual hat function supported in the square centered at $\bm{x}^*$ (not depicted).}}
\label{fig:VT0}
\end{figure}
Thus, it is natural to introduce the operator 
$$
\Imeshz : \Vmesh \to \Vmeshz
$$
defined as the Lagrange interpolation operator at the nodes in ${\cal P}$. 
The following result will be crucial in the forthcoming analysis.
\begin{proposition}[comparison between interpolation operators]\label{prop:compareInterp}
Let $\mesh$ be $\Lambda$-admissible. Then,
there exists a constant $C_I >0$, depending on $\Lambda$ but independent of $\mesh$, such that
\begin{equation}\label{eq:bound-interp}
| v-\Imeshz v|_{1,\Omega}  \, \leq C_I \, | v-\Imesh v|_{1,\mesh} \qquad \forall v \in \Vmesh \,.
\end{equation} 
\end{proposition}
Note that the result is non-trivial, since the `non-conforming' $\Imesh$ may operate at a finer scale than the `conforming' $\Imeshz$, although not too fine due to the condition of $\Lambda$-admissibility. The proof of the result is postponed to Sect. \ref{sec:bound-Interp}.

We will also need some Cl\'ement quasi-interpolation operators. Precisely, let us denote by $\widetilde{\mathcal I}_\mesh^0: \mathbb{V} \to \Vmeshz$ the classical  Cl\'ement  operator on the finite-element space $\Vmeshz$; that is, the value at each internal (proper) node is the average of the target function on the support of the associated basis function.
Similarly, let  $\widetilde{\mathcal I}_\mesh: \mathbb{V} \to \Vmesh$ be the Cl\'ement  operator on the virtual-element space $\Vmesh$, as defined in \cite{cileni}.

\begin{lemma}[Clement interpolation estimate]\label{lm:quasi-inetrp}
The following inequality holds
\begin{equation}
\sum_{E\in\mesh} h_E^{-2} \| v- \widetilde{\mathcal I}_\mesh^0 v\|_{0,E}^2 \lesssim \vert v\vert_{1,\Omega}^2 \qquad \forall v \in \mathbb{V}\,,
\end{equation}
where the hidden constant depends on the maximal index $\Lambda$ but not on $\mesh$.
\end{lemma}
\begin{proof}
Let $v_\mesh=\widetilde{\mathcal I}_\mesh v \in \Vmesh$. Since 
$ v- \widetilde{\mathcal I}_\mesh^0 v=(v-v_\mesh)+(v_\mesh-\widetilde{\mathcal I}_\mesh^0 v_\mesh) + (\widetilde{\mathcal I}_\mesh^0 v_\mesh - \widetilde{\mathcal I}_\mesh^0 v )$
and $\widetilde{\mathcal I}_\mesh^0$ is locally stable in $L^2$, we deduce
\begin{eqnarray}
\sum_{E\in\mesh} \!\! h_E^{-2} \| v- \widetilde{\mathcal I}_\mesh^0 v\|_{0,E}^2
&\lesssim& 
\sum_{E\in\mesh} \!\! h_E^{-2} \| v- v_\mesh\|_{0,E}^2 +  h_E^{-2} \| v_\mesh - \widetilde{\mathcal I}_\mesh^0 v_\mesh\|_{0,E}^2 
\lesssim \vert v\vert_{1,\Omega}^2 + \! \sum_{E\in\mesh}  \!\! h_E^{-2} \| v_\mesh - \widetilde{\mathcal I}_\mesh^0 v_\mesh\|_{0,E}^2.\nonumber
\end{eqnarray}
Thus, we need to prove  
$$  \sum_{E\in\mesh}  h_E^{-2} \| v_\mesh - \widetilde{\mathcal I}_\mesh^0 v_\mesh\|_{0,E}^2 \lesssim \vert v_\mesh\vert_{1,\Omega}^2.$$
To show this bound, write
$ v_\mesh - \widetilde{\mathcal I}_\mesh^0 v_\mesh = 
(v_\mesh - {\mathcal I}_\mesh^0 v_\mesh) + \widetilde{\mathcal I}_\mesh^0 
({\mathcal I}_\mesh^0 v_\mesh - v_\mesh)$
because $\widetilde{\mathcal I}_\mesh^0$ is invariant in $\Vmeshz$, i.e. ${\mathcal I}_\mesh^0 v_\mesh = \widetilde{\mathcal I}_\mesh^0({\mathcal I}_\mesh^0 v_\mesh)$.
We finally use again the stability of $\widetilde {\mathcal I}_\mesh^0$ in $L^2$ together with Proposition \ref{prop:scaledPoincare}  to obtain 
$$
\sum_{E\in\mesh}  h_E^{-2} \| v_\mesh - \widetilde{\mathcal I}_\mesh^0 v_\mesh\|_{0,E}^2 \lesssim
\sum_{E\in\mesh}  h_E^{-2} \| v_\mesh - {\mathcal I}_\mesh^0 v_\mesh\|_{0,E}^2 \lesssim C_\Lambda \vert v_\mesh\vert_{1,\Omega}^2.
$$
This concludes the proof.
\end{proof}

\section{A posteriori error analysis}\label{sec:aposteriori}

Since we are interested in building adaptive discretizations, we rely on a posteriori error control. Following \cite{Cangiani-etal:2017}, we first introduce a residual-type a posteriori estimator. To this end, recalling that $\data = (A,c,f)$ denotes the set of piecewise constant data, for any $v \in \Vmesh$ and any element $E$ let us define the internal residual over $E$  
\begin{equation}\label{eq:estim1}
\rmesh(E;v,\data) :=   f_E \,    - \, c_E \, \PiE v\,.
\end{equation}  
Similarly, for any two elements $E_1, E_2 \in \mesh$ sharing an edge $e \in \edge_{E_1}\cap \edge_{E_2}$, let us define the jump residual over $e$
\begin{equation}\label{eq:estim2}
\jmesh(e;v,\data) := \jump{A \nabla \, \Pi^\nabla_{\mesh} v}_e = (A_{E_1} \nabla \, \Pi^\nabla_{E_1} v_{|E_1}) \cdot \boldsymbol{n}_1  +  (A_{E_2} \nabla \, \Pi^\nabla_{E_2} v_{|E_2}) \cdot \boldsymbol{n}_2 \,, 
\end{equation}  
where $\boldsymbol{n}_i$ denotes the unit normal vector to $e$ pointing outward with respect to $E_i$; set $\jmesh(e;v) =0$ if $e \subset \partial\Omega$. Then, taking into account Remark \ref{rem:bound-global-index}, we define the local residual estimator associated with $E$
\begin{equation}\label{eq:estim3}
\etamesh^2(E;v,\data) := h_E^2 \Vert \rmesh(E;v,\data) \Vert_{0,E}^2 \ + \ \tfrac12 \sum_{e \in \edgeE} h_E \Vert \jmesh(e;v,\data) \Vert_{0,e}^2 \;,
\end{equation}
as well as the global residual estimator
\begin{equation}\label{eq:estim4}
\etamesh^2(v,\data) := \sum_{E \in \mesh} \etamesh^2(E;v,\data) \,.
\end{equation}

An upper bound of the energy error is provided by the following result.
The proof follows \cite[Theorem 13]{Cangiani-etal:2017}, with the remarkable technical
difference that the stabilization term $\Smesh(\umesh, \umesh)$ 
is not scaled by the constant $\gamma$ in \eqref{eq:apost1}.

\begin{proposition}[upper bound] \label{prop:aposteriori} 
There exists a constant $C_\text{apost} >0$  depending on $\Lambda$ and $\data$ but independent of $u$, $\mesh$, $\umesh$ and $\gamma$, such that
\begin{equation}\label{eq:apost1}
\vert u - \umesh \vert_{1, \Omega}^2  \leq C_\text{apost} \left( \etamesh^2(\umesh,\data) 
+ \Smesh(\umesh, \umesh) \right) \,.
\end{equation} 
\end{proposition}
\begin{proof}
We let $v\in H^1_0(\Omega)$ and proceed as in \cite[Theorem 13]{Cangiani-etal:2017} to write
\[
  \cB(u-\umesh,v) = 
  \big((f,v-v_\mesh)_\Omega-\cB(u_\mesh,v-v_\mesh) \big) + \cB(u-\umesh,v_\mesh) =: I + II,
\]  
except that we choose $v_\mesh=\widetilde{\mathcal I}_\mesh^0 v \in \Vmeshz$, where 
$\widetilde{\mathcal I}_\mesh^0$ is  the Cl\' ement quasi-interpolation operator on $\Vmeshz$.
This choice has a remarkable impact on \eqref{eq:apost1} compared with \cite[Theorem 13]{Cangiani-etal:2017}.

We start by estimating the term $I$: one has
\begin{align*}
I & = \sum_{E\in\mesh} \left\{(f,v-v_\mesh)_E
-(A_E\nabla \Pi^{\nabla}_E u_\mesh,\nabla(v-v_\mesh))_E- c_E(\Pi^{\nabla}_E u_\mesh,v-v_\mesh)_E \right\} \\
& +  \sum_{E\in\mesh} \left\{(A_E\nabla (\Pi^{\nabla}_E u_\mesh -u_\mesh),\nabla(v-v_\mesh))_E+ c_E((\Pi^{\nabla}_E u_\mesh -u_\mesh),v-v_\mesh)_E \right\} =: I_1+ I_2.
\end{align*}
Integrating by parts, employing Lemma \ref{lm:quasi-inetrp} and proceeding as in \cite{Cangiani-etal:2017}, we get
\begin{align}
\vert I_1\vert &\lesssim \eta_\mesh(u_\mesh,\data) \vert v \vert_{1,\Omega} \label{I_1}\\
\vert I_2\vert &\lesssim \left(
\sum_{E\in\mesh} \| \nabla(\Pi^{\nabla}_E u_\mesh-u_\mesh)\|^2_{0,E} + h_E\| \Pi^{\nabla}_E u_\mesh-u_\mesh\|^2_{0,E} 
\right)^{1/2}  \vert v \vert_{1,\Omega} \lesssim \Smesh (u_\mesh,u_\mesh)^{1/2}\vert v \vert_{1,\Omega}\label{I_3}\,.
\end{align}

We now deal with term $II$. We first apply Lemma \ref{L:Galerkin-orthogonality} to obtain
\[
\cB(u-\umesh, v_\mesh) = \sum_{E\in \mesh} c_E \int_E (\Pi^\nabla_E u_\mesh - u_\mesh) v_\mesh,
\]
because $v_\mesh\in\V_\mesh^0$; this is the key difference with \cite[Theorem 13]{Cangiani-etal:2017}.
We next recall the definition \eqref{eq:def-PinablaE} of $\Pi^\nabla_E$, and corresponding scaled Poincar\'e inequality
$\|\Pi^\nabla_E u_\mesh - u_\mesh\|_{0,E} \lesssim h_E \, |\Pi^\nabla_E u_\mesh - u_\mesh|_{1,E}$ for all $E\in\mesh$, to
arrive at
\[
\cB(u-\umesh, v_\mesh) \lesssim   h_E\, \Smesh(u_\mesh,u_\mesh)^{1/2} \vert v_\mesh\vert_{1,\Omega}.
\]

Finally, taking $v=u-\umesh \in H^1_0(\Omega)$, employing the coercivity of $\cB(\cdot,\cdot)$ and combining
the above estimates yield the assertion.
\end{proof}

We state the following result, which is proven in \cite{Cangiani-etal:2017} for the choice \eqref{vem:choice:2} but it holds with the same proof for any other admissible choice of $\VE$.

\begin{proposition}[local lower bound]\label{local-lower-bound}
There holds
\begin{equation}
\eta_\mesh^2(E;u_\mesh,\data) \lesssim \sum_{E'\in \omega_E } \left( 
| u-u_\mesh|^2_{1,E'} + S_{E'}(u_\mesh,u_\mesh)
\right)
\end{equation}
where $\omega_E:=\{E^\prime: \vert \partial E \cap \partial E^\prime\vert \not=0 \}$. The hidden constant is independent of $\gamma, h,u$ and $u_\mesh$.
\end{proposition}

\begin{corollary}[global lower bound]\label{global-lower-bound}
There exists a constant $c_\text{apost} >0$, depending on $\Lambda$ but independent of $u$, $\mesh$, $\umesh$ and $\gamma$, such that
\begin{equation}\label{eq:global-lower-bound}
c_\text{apost} \, \eta_\mesh^2(u_\mesh,\data) \leq  \vert u-u_\mesh\vert^2_{1, \Omega} + S_\mesh(u_\mesh,u_\mesh) \,.
\end{equation}
\end{corollary}

We now state one of the two main results contained in this paper. Due to the technical nature of the proof, we postpone it to Sect. \ref{sec:bound-ST}.

\begin{proposition}[bound of the stabilization term by the residual] \label{prop:bound-ST} There exists a constant $C_B >0$ depending on $\Lambda$ but independent of $\mesh$, $\umesh$ and $\gamma$ such that
\begin{equation}\label{eq:bound-ST}
\gamma^2 \Smesh(\umesh, \umesh) \leq C_B \, \etamesh^2(\umesh, \data) \,.
\end{equation}
\end{proposition}

Combining Proposition \ref{prop:aposteriori}, Corollary \ref{global-lower-bound} and Proposition \ref{prop:bound-ST}, 
we get the following stabilization-free (global) upper and lower bounds.

\begin{corollary}[stabilization-free a posteriori error estimates]\label{Corollary:stab-free}
Assume that the parameter $\gamma$ is chosen to satisfy $\gamma^2 > \displaystyle{\frac{C_B}{c_\text{apost}}}$. Then it holds
\begin{equation}\label{apost:stab-free}
C_L  \etamesh^2(\umesh,\data)    \leq    \vert u - \umesh \vert_{1, \Omega}^2  
\leq C_U   \etamesh^2(\umesh,\data)  \, ,
\end{equation} 
with $C_L=c_\text{apost} - C_B \gamma^{-2}$ and 
$C_U= C_\text{apost} \, \left(1+ C_B \gamma^{-2} \right)$.
\end{corollary}
\smallskip

The sharpness of \eqref{apost:stab-free} will be investigated also from the numerical standpoint in Section \ref{sec:numerics-1}.

\section{Adaptive VEM with contraction property}\label{sec:AVEM-pcwconstant}
In Section \ref{sub:galerkin}, we introduce an Adaptive Virtual Element Method (AVEM), called {\tt GALERKIN}, for approximating \eqref{eq:pde-var} to a given tolerance under Assumption \ref{ass:constant-coeff} (piecewise constant data). We investigate the effect of local mesh refinements on our error estimator in Section \ref{S:change-error-estimator}, and prove a contraction property of {\tt GALERKIN} in Section \ref{S:convergence}. The design and analysis of an AVEM able to handle also variable data is postponed to
\cite{BCNVV:22}.

\subsection{The module {\tt GALERKIN}}
\label{sub:galerkin}
%
Given a $\Lambda$-admissible input mesh $\mesh_0$, piecewise constant input data $\data$ on $\mesh_0$ and a tolerance $\varepsilon>0$, the call
\begin{equation}\label{module:_GAL}
 [\mesh,u_{\mesh}]={\tt GALERKIN}(\mesh_0,\data,\varepsilon)
\end{equation}
produces  a $\Lambda$-admissible bisection refinement $\mesh$ of
$\mesh_0$ and the Galerkin approximation $u_{\mesh} \in \Vmesh$ to the solution ${u}$ of problem \eqref{eq:pde} with piecewise constant data $\data$ on $\mesh_0$, such that
\begin{equation}\label{aux:AVEM:estimate}
\vvvert u -u_{\mesh} \vvvert \leq C_G \, \varepsilon \,,
\end{equation}
with $C_G=\sqrt{c^\cB C_U}$, where $c^\cB$ is defined in \eqref{norm:equiv} and 
$C_U$ is the upper bound constant in Corollary \ref{Corollary:stab-free}.
This is obtained by iterating the classical paradigm 
\begin{equation}\label{eq:paradigm}
\texttt{SOLVE} \,\,
\longrightarrow \,\,
\texttt{ESTIMATE} \,\,
\longrightarrow \,\,
\texttt{MARK} \,\,
\longrightarrow \,\,
\texttt{REFINE}
\end{equation}
thereby producing a sequence of $\Lambda$-admissible meshes $\{\mesh_k\}_{k\geq 0}$, and associated Galerkin solutions $u_k\in\V_{\mesh_k}$ to the problem \eqref{eq:pde} with data $\mathcal{D}$. The iteration stops as soon as $\eta_{\mesh_k} (u_k, \data) \leq \varepsilon$, which is possible thanks to the convergence result stated in Theorem \ref{T:contraction} below.

The modules in \eqref{eq:paradigm} are defined as follows: given piecewise constant data $\data$ on $\mesh_0$,

\medskip
\begin{enumerate}[$\bullet$]
\item $[u_\mesh]=\texttt{SOLVE}(\mesh, \data)$ produces the Galerkin solution on the mesh $\mesh$ for data $\data$;
\item $[\{\eta_\mesh(\, \cdot \, ;u_\mesh,\data)\}]=\texttt{ESTIMATE}(\mesh, u_\mesh)$ computes the local residual estimators \eqref{eq:estim3} on the mesh $\mesh$, which depend on the Galerkin solution $u_\mesh$ and data $\data$;
\item $[\mathcal{M}] = \texttt{MARK}(\mesh, \{\eta_\mesh(\, \cdot \, ;u_\mesh,\data)\}, \theta)$  implements  the D{\"o}rfler criterion \cite{dorfler}, namely for a given parameter $\theta \in (0,1)$ an almost minimal set $\mathcal{M} \subset \mesh$ is found such that
\begin{equation}
\label{eq:dorfler}
\theta \eta_{\mesh}^2 (u_\mesh, \data) \leq 
\sum_{E \in \mathcal{M}} \eta_{\mesh}^2(E; 
u_\mesh, \data) \,;
\end{equation} 
\item $[\meshs]=\texttt{REFINE}(\mesh, \mathcal{M}, \Lambda)$ produces a $\Lambda$-admissible refinement $\meshs$ of $\mesh$, obtained by newest-vertex bisection of all the elements in $\mathcal{M}$ and, possibly, some other elements of $\mesh$.
\end{enumerate}

\medskip
In the procedure \texttt{REFINE}, non-admissible hanging nodes,  i.e., hanging nodes with global index larger than $\Lambda$, might be created while refining elements in $\mathcal{M}$.
Thus, in order to obtain a  $\Lambda$-admissible partition $\meshs$, \texttt{REFINE} possibly refines other elements in $\mesh$ (completion). A practical completion procedure, called \texttt{MAKE\_ADMISSIBLE}, is described in Section \ref{sec:make-admissible}, while the complexity analysis of such a procedure is discussed in the forthcoming paper \cite{BCNVV:22}.

The following crucial property of {\tt GALERKIN} guarantees its convergence in a finite number of iterations proportional to $|\log \varepsilon|$. We postpone its proof to Section \ref{S:convergence}. 

\begin{theorem}[contraction property of {\tt GALERKIN}]\label{T:contraction}
Let $\mathcal{M}\subset \mesh$ be the set of marked elements relative to the Galerkin solution $\umesh\in\Vmesh$. If $\meshs$ is the refinement of $\mesh$ obtained by applying {\tt REFINE}, then 
for $\gamma$ sufficiently large, there exist constants $\alpha\in (0,1)$ and $\beta>0$ such that one has
\begin{eqnarray}
\vvvert u - \umeshs \vvvert^2+\beta \etameshs^2(\umeshs,\data) \leq \alpha \left( \vvvert u - \umesh \vvvert^2+\beta \etamesh^2(\umesh,\data) \right).
\end{eqnarray}
\end{theorem}

\subsection{Error estimator under local mesh refinements}\label{S:change-error-estimator}
In this section we prove three crucial properties of the error estimator that will be employed in studying the convergence of {\tt GALERKIN}.

Let $\meshs$ be a refinement of $\mesh$ produced by {\tt REFINE} by bisection. Consider an element $E \in \mesh$ which has been split into two elements $E_1, E_2 \in \meshs$. If $v \in \Vmesh$, then $v$ is known on $\partial E$, hence in particular at the new vertex of $E_1, E_2$ created by bisection. Thus, $v$ is known at all nodes (vertices and possibly hanging nodes) sitting on $\partial E_1$ and $\partial E_2$, since the new edge $e=E_1\cap E_2$ does not contain internal nodes. This uniquely identifies a function in $\V_{E_1}$ and a function in $\V_{E_2}$, which are continuous across $e$.
In this manner, we associate to any $v \in \Vmesh$ a unique function $v_* \in \Vmeshs$, that coincides with $v$ on the skeleton $\edge$ (but possibly not on $\edges$). We will actually write $v$ for $v_*$ whenever no confusion is possible.

\subsubsection*{a. Comparison of residuals under refinement}
Let $\meshs$ be a refinement of $\mesh$ as above, and let $E \in \mesh$ be an element that has been split into two elements $E_1, E_2 \in \meshs$; observe that $h_{E_i} = \frac1{\sqrt{2}} h_E$, $i=1,2$. Given $v \in \V_\mesh$, we aim at comparing the local residual estimator $\etamesh(E; v,\data)$ defined in \eqref{eq:estim3} to the local estimator $\etameshs(E;v,\data)$ defined by
\begin{equation}\label{eq:residual-comp1}
\etameshs^2(E;v,\data) := \sum_{i=1}^2 \etameshs^2(E_i;v,\data)= \sum_{i=1}^2 h_{E_i}^2 \Vert  f_{E_i} \,    - \, c_{E_i} \, \PiEi v \Vert_{0,E_i}^2 + \tfrac12 {\sum_{i=1}^2 \sum_{e \in {\cal E}_{E_i}} h_{E_i} \Vert j_{\meshs}(e;v) \Vert_{0,e}^2} \;
\end{equation}
where it is important to observe that, as $\data$ does not change under refinement, we have 
$  f_{E_i}=f_E\vert_{E_i}$, $c_{E_i}=c_E\vert_{E_i}$, $A_{E_i}=A_E\vert_{E_i}$.

\begin{lemma}[local estimator reduction]\label{prop:v-v comparison}
There exist constants $\mu \in (0,1)$ and ${c_{er,1}}>0$ independent of $\mesh$ such that for any element $E \in \mesh$ which is split into two children $E_1,E_2 \in \meshs$, one has
\begin{equation}\label{eq:residual-comp0}
\etameshs(E;v,\data) \leq \mu \ \etamesh(E; v,\data) + {c_{er,1}} \, S^{1/2}_{\mesh(E)}(v,v)  \qquad \forall v \in \Vmesh \,,
\end{equation}
where $S_{\mesh(E)}(v,v) := \sum_{E' \in \mesh(E)} S_{E'}(v,v)$ with $\mesh(E) := \{ E' \in \mesh : {\cal E}_{E} \cap {\cal E}_{E'}  \not= \emptyset \}$.
\end{lemma}
\begin{proof}
As $f_E$ and $c_E$ do not change under refinement, we simplify the notation and write
$ r_E=f-c\PiE v$, $r_{E_i}=f-c\PiEi v$ (i=1,2)
with $f=f_E$ and $c=c_E$. We have 
$
r_{E_i}=f-c\PiEi v= r_E + c(\PiE v -\PiEi v)
$,
whence
\begin{equation}\label{eq:add}
\begin{aligned}
\sum_{i=1}^2 h_{E_i}^2 \|r_{E_i}\|_{0,E_i}^2
&\leq \sum_{i=1}^2 h_{E_i}^2 \left( 
(1+\epsilon)\| r_E\|_{0,E_i}^2 + \left(1+\frac{1}{\epsilon}\right)
\|c(\PiE v -\PiEi v)\|_{0,E_i}^2
\right)\\
&\leq  \frac{1+\epsilon}{2}h_E^2 \|r_E\|_{0,E}^2 + h_E^2 \left(1+\frac{1}{\epsilon}\right)\frac{{c_E^2}}{2}
\sum_{i=1}^2 \|\PiE v -\PiEi v\|_{0,E_i}^2.
\end{aligned}
\end{equation}
In addition, we see that
$
\displaystyle{\sum_{i=1}^2 \|\PiE v -\PiEi v\|_{0,E_i}^2 \leq
2  \|v -\PiE v\|_{0,E}^2 + 2 
\sum_{i=1}^2 \|v -\PiEi v\|_{0,E_i}^2}.
$
Applying the Poincar\'e inequality and the minimality of $\PiE$, we get 
\begin{align*}
\|v - \PiE v \|_{0,E}^2 &\lesssim \vert \PiE v -v\vert_{1,E}^2 \leq  \vert \IE v -v\vert_{1,E}^2 \,,
\\
\sum_{i=1}^2 \|v -\PiEi v\|_{0,E_i}^2 &\lesssim 
\sum_{i=1}^2 \vert v -\PiEi v\vert_{1,E_i}^2
\lesssim  \vert v -\PiE v\vert_{1,E}^2 \leq 
 \vert v -\IE v\vert_{1,E}^2.
\end{align*}
Therefore, choosing $\epsilon=\frac 1 2$, setting $\mu := \frac{1+\epsilon}{2} = \frac 3 4$ in
  \eqref{eq:add}, and employing \eqref{eq:stab-norm} we infer that
\begin{eqnarray}\label{eq:residual-comp2}
\sum_{i=1}^2 h_{E_i}^2 \|r_{E_i}\|_{0,E_i}^2 \leq 
\mu\, h_E^2 \|r_E\|_{0,E}^2 + C h_E^2 \vert v - \IE v \vert_{1,E}^2
\leq  \mu\, h_E^2 \|r_E\|_{0,E}^2 + C h_E^2 S_E(v,v).
\end{eqnarray}

Concerning the jump terms, we first observe that writing 
$
 j_{\meshs}(e;v) =  j_{\mesh}(e;v) + \big(j_{\meshs}(e;v) - j_{\mesh}(e;v) \big)
$,
one has for any $\epsilon > 0$
\begin{equation}\label{eq:residual-comp3}
\sum_{i=1}^2 \sum_{e \in {\cal E}_{E_i}} h_{E_i} \Vert j_{\meshs}(e;v) \Vert_{0,e}^2   \leq (1+\epsilon)  T_1  + \left(1+\tfrac1\epsilon\right) T_2 \,,  
\end{equation}
with 
$T_1 := \displaystyle{\sum_{i=1}^2 \sum_{e \in {\cal E}_{E_i}} h_{E_i} \Vert j_{\mesh}(e;v) \Vert_{0,e}^2} $, \ 
$T_2 := \displaystyle{\sum_{i=1}^2 \sum_{e \in {\cal E}_{E_i}} h_{E_i} \Vert j_{\meshs}(e;v) - j_{\mesh}(e;v) \Vert_{0,e}^2} $.

Considering $T_1$, notice that $ j_{\mesh}(e;v)= 0$ on the new edge created by the bisection of $E$; hence, 
\begin{equation}\label{eq:residual-comp4}
T_1 \leq \frac1{\sqrt{2}} \sum_{e \in {\cal E}_{E}} h_{E} \Vert j_{\mesh}(e;v) \Vert_{0,e}^2 \,.
\end{equation}
In order to bound the second term $T_2$, define $\meshs(E_i) := \{ E' \in \meshs : {\cal E}_{E_i} \cap {\cal E}_{E'} 
\not= \emptyset \}$; for an edge $e \in {\cal E}_{E_i}$, let $E_{i,e} \in \meshs(E_i)$ be such that $e = \partial E_i \cap \partial E_{i,e}$. Then,
\begin{equation*}
\begin{aligned}
\Vert j_{\meshs}(e;v) - j_{\mesh}(e;v) \Vert_{0,e} &= \Vert \jump{A\nabla (\Pi^\nabla_\meshs - \Pi^\nabla_\mesh)v} \Vert_{0,e} 
\\
&\leq \Vert A_E \nabla (\Pi^\nabla_{E_i} - \Pi^\nabla_E)v
\Vert_{0,e} + \Vert A_{\widehat{E}_{i,e}} \nabla (\Pi^\nabla_{E_{i,e}} - \Pi^\nabla_{\hat{E}_{i,e}})v \Vert_{0,e} \,,
\end{aligned}
\end{equation*}
where, in general, $\widehat{F} \in \mesh$ denotes the parent element of $F \in \meshs$. Hence, using the trace inequality  and the equivalence $h_{E_i} \simeq h_{E_{i,e}}$, we easily get
\begin{equation*}
\begin{aligned}
T_2 &\lesssim  \sum_{i=1}^2 \sum_{E' \in \meshs(E_i)} \Vert \nabla (\Pi^\nabla_{E'} - \Pi^\nabla_{\widehat{E}'})v \Vert_{0,E'}^2  
\lesssim  \sum_{i=1}^2 \sum_{E' \in \meshs(E_i)} \big( \Vert \nabla (v- \Pi^\nabla_{E'}v) \Vert_{0,E'}^2 + \Vert \nabla (v - \Pi^\nabla_{\widehat{E}'}v) \Vert_{0,E'}^2 \big) \,.
\end{aligned}
\end{equation*}
The minimality property of the orthogonal projections $\Pi^\nabla_{E'}$ and $\Pi^\nabla_{\widehat{E}'}$ yields
$$
\Vert \nabla (v- \Pi^\nabla_{E'}v) \Vert_{0,E'} \leq \Vert \nabla (v- {\cal I}_{\widehat{E}'}v) \Vert_{0,E'} \leq \Vert \nabla (v- {\cal I}_{\widehat{E}'}v) \Vert_{0,\widehat{E}'}
$$
and
$$
\Vert \nabla (v - \Pi^\nabla_{\widehat{E}'}v) \Vert_{0,E'} \leq \Vert \nabla (v - \Pi^\nabla_{\widehat{E}'}v) \Vert_{0,\widehat{E}'} \leq \Vert \nabla (v- {\cal I}_{\widehat{E}'}v) \Vert_{0,\widehat{E}'} \,.
$$
This gives
\begin{equation}\label{eq:residual-comp5}
T_2 \leq C \sum_{E' \in \mesh(E)} \Vert \nabla (v- {\cal I}_{{E}'}v) \Vert_{0,{E}'}^2 \leq C \sum_{E' \in \mesh(E)} S_{E'}(v,v) \,,
\end{equation}
thanks to \eqref{eq:stab-sE} and \eqref{eq:def-SE}. Using \eqref{eq:residual-comp2}, \eqref{eq:residual-comp3} with a sufficiently small $\epsilon$, \eqref{eq:residual-comp4} and \eqref{eq:residual-comp5}, we arrive at the desired result.
\end{proof}

\subsubsection*{b. Lipschitz continuity of the residual estimator}
Since the following result can be proven by standard arguments, we only sketch its proof.

\begin{lemma}[Lipschitz continuity of error estimator]\label{prop:v-w comparison}
There exists a constant ${ c_{er,2}}>0$ independent of $\mesh$ such that for any element $E \in \mesh$, one has
\begin{equation}\label{eq:residual-comp6}
\big|\etamesh(E;v,\data) - \etamesh(E;w,\data)\big| \leq {c_{er,2}} \, |v-w|_{1,\mesh(E)} \qquad \forall v,w \in  { H^1(\Omega)}\,,      
\end{equation}
where $|v-w|_{1,\mesh(E)}^2 :=\sum_{E' \in \mesh(E)} | v-w |_{1,E'}^2$.
\end{lemma}
\begin{proof}
By standard FE arguments, using the fact that $\etamesh(E;v,\data) $ and $\etamesh(E;w,\data)$ are written in terms of $\PiE v$ and $\PiE w$, which are polynomials, and employing inverse estimates,  we easily obtain 
\begin{eqnarray}
\big|\etamesh(E;v,\data) - \etamesh(E;w,\data)\big| 
\leq C \, |\Pi^\nabla_{\mesh(E)}(v-w)|_{1,\mesh(E)} 
\leq C \, |v-w|_{1,\mesh(E)} \,,
\end{eqnarray}
where the last inequality uses the definition of $\PiE$.
\end{proof}

\subsubsection*{c. Reduction property for the local residual estimators}
Concatenating Lemma \ref{prop:v-v comparison} with Lemma \ref{prop:v-w comparison} we arrive at the following reduction property for the local residual estimators.
\begin{proposition}[estimator reduction property on refined elements]\label{prop: reduction}
There exist constants ${ \mu} \in (0,1)$, ${c_{er,1}}>0$ and ${ c_{er,2}}>0$ independent of $\mesh$ such that for any $v \in \Vmesh$ and $w \in \Vmeshs$, and any element $E \in \mesh$ which is split into two children $E_1,E_2 \in \meshs$, one has 
\begin{equation}\label{eq:residual-comp00}
\etameshs(E;w,\data) \leq { \mu} \ \etamesh(E; v,\data) + { c_{er,1}} \, S^{1/2}_{\mesh(E)}(v,v) + { c_{er,2}} \, |v-w|_{1,\mesh(E)} \,.
\end{equation}
\end{proposition}
\begin{proof}
Employing Lemma \ref{prop:v-w comparison} with $\mesh$ replaced by $\meshs$, we get 
$$
\etameshs(E;w,\data) \leq \etameshs(E;v,\data) +{ c_{er,2}} \, |v-w|_{1,\meshs(E)} \,.
$$
Finally, noting that $\meshs(E)$ is a refinement of  $\mesh(E)$, we conclude using Lemma \ref{prop:v-v comparison}. 
\end{proof}

\subsection{Convergence of AVEM}\label{S:convergence}
In this section we aim at proving Theorem \ref{T:contraction} (contraction property of {\tt REFINE}). To do so, since we first need to quantify the estimator and error reduction under refinement, we divide our argument into three steps.

\subsubsection*{a. Estimator reduction}
We first study the effect of mesh refinement on the estimator.

\begin{proposition}[estimator reduction]\label{prop: reduction:2}
Let $\mathcal{M}\subset \mesh$ be the  set of marked elements in {\tt{MARK}}, relative to the Galerkin solution $\umesh \in\Vmesh$, and let $\meshs$ be the refinement produced by {\tt REFINE}. Then, there exist constants $\rho<1$ and $C_{er,1}, C_{er,2}>0$ independent of $\mesh$ such that for all $w \in \Vmeshs$ one has 
\begin{equation}\label{estimator_reduction}
\etameshs^2(\meshs;w,\data) \leq \rho \ \etamesh^2(\mesh; \umesh,\data) + C_{er,1} \, S_{\mesh}(\umesh,\umesh) + C_{er,2} \, |\umesh-w|^2_{1,\Omega} \,.
\end{equation}
\end{proposition}
\begin{proof}
For simplicity of notation, let us set $v=\umesh$ and let $\mathcal{R}=\mathcal{R}_{\mesh\to\meshs}$ be the set of all elements of $\mesh$ that are refined by {\tt{REFINE}} to obtain $\meshs$. If $E\in\mathcal{R}$, then Proposition \ref{prop: reduction} (estimator reduction property on refined elements) yields for all $0<\delta\le1$
\begin{equation*}
\etameshs^2(E;w,\data) \leq (1+\delta){\mu^2} \ \etamesh^2(E; v,\data) + { 2\left(1+\frac{1}{\delta}\right) \left(c_{er,1}^2 \, S_{\mesh(E)}(v,v) + c_{er,2}^2 \, |v-w|^2_{1,\mesh(E)} \right)\,.}
\end{equation*}
If $E\not\in \mathcal{R}$, then Lemma \ref{prop:v-w comparison} (Lipschitz continuity of estimator) implies 
\begin{equation*}
  \etameshs^2(E;w,\data) \leq (1+\delta) \ \etamesh^2(E; v,\data) + \left(1+\frac{1}{\delta}\right) {c_{er,2}^2} \, |v-w|^2_{1,\mesh(E)}.
\end{equation*}
Hence, adding the two inequalities, { there exist positive constants $C_{er,1}=C_{er,1}(\delta)$ and $C_{er,2}=C_{er,2}(\delta)$} such that
\begin{equation*}
\etameshs^2(\meshs;w,\data) \leq (1+\delta) \ \etamesh^2(\mesh; v,\data) - 
(1+\delta)(1-{\mu^2}) \ \etamesh^2(\mathcal{R}; v,\data)+ C_{er,1} \, S_{\mesh}(v,v)+
C_{er,2} \, |v-w|^2_{1,\Omega}.
\end{equation*}
Since $\etamesh^2(\mathcal{R}; v,\data) \geq \etamesh^2(\mathcal{M}; v,\data)$ due to the property $\mathcal{M} \subseteq \mathcal{R}$, making use of \eqref{eq:dorfler} yields
\begin{equation}
 \etamesh^2(\mesh; v,\data) -
(1-{\mu^2})\etamesh^2(\mathcal{R}; v,\data)\leq \big(1-\theta(1-{\mu^2})\big)\etamesh^2(\mesh; v,\data).
\end{equation}
Choosing $\delta>0$ sufficiently small so that 
$$ 
(1+\delta)(1-\theta(1-{\mu^2}))=1-\frac{\theta(1-{\mu^2})}{2},
$$
and setting $\rho := 1-\frac{\theta(1-{\mu^2})}{2} < 1$ concludes the proof.
\end{proof}

\subsubsection*{b. Comparison of errors under refinement}
Let $\umesh \in \Vmesh$ be again the solution of Problem \eqref{def-Galerkin}, and let $\umeshs \in \Vmeshs$ be the solution of the analogous problem on the refined mesh $\meshs$. We aim at comparing $\vvvert u - \umesh \vvvert$ with $\vvvert u - \umeshs \vvvert$.
\begin{lemma}[lack of orthogonality]
There exists a constant $C_D>0$ independent of $\mesh$ and $\meshs$ such that for all $\epsilon>0$
\begin{equation}\label{lack:orthog}
\begin{aligned}
  \big\vert \mathcal{B}(u-\umeshs,\umeshs-\umesh) \big\vert & \leq \epsilon \vert u-\umeshs\vert_{1,\Omega}^2 + \epsilon \vert \umesh-\umeshs\vert^2_{1,\Omega}
  \\
  & \quad +C_D\left(1+\frac{1}{\epsilon}\right) \big(S_{\meshs}(\umeshs,\umeshs)+S_{\mesh}(\umesh,\umesh) \big) \,.
\end{aligned}
\end{equation}
\end{lemma}
\begin{proof}
Write $\mathcal{B}(u-\umeshs,\umeshs-\umesh) = I + II +III$ with
\begin{equation*}
  I := \mathcal{B}(u-\umeshs,\umeshs-I^0_{\meshs} \umeshs),
  \quad
  II := - \mathcal{B}(u-\umeshs,\umesh-I^0_{\mesh} \umesh),
  \quad
  III :=  \mathcal{B}(u-\umeshs,I^0_{\meshs}\umeshs-I^0_{\mesh} \umesh).
\end{equation*}
We recall the crucial estimate, stemming from Proposition \ref{prop:compareInterp} and eq. \eqref{eq:stab-norm},
\begin{equation}
\vert v - \mathcal{I}^0_\mesh v \vert_{1,\Omega} \lesssim
\vert v - \Imesh v \vert_{1,\mesh}\lesssim S_\mesh(v,v) \qquad \forall v\in \Vmesh.
\end{equation}
Invoking this estimate twice, and employing Young's inequality with $\epsilon>0$, we obtain 
\begin{eqnarray}
&& \vert I \vert \lesssim \vert u -\umeshs\vert_{1,\Omega} \ S_\meshs^{1/2}(\umeshs,\umeshs)
{\leq
\epsilon \vert u -\umeshs\vert^2_{1,\Omega} + \frac{C}{\epsilon}  S_\meshs(\umeshs,\umeshs)} \nonumber\\
&& \vert II \vert \lesssim \vert u-\umeshs\vert_{1,\Omega} \ S_\mesh^{1/2}(\umesh,\umesh) 
\leq \epsilon \vert u -\umeshs\vert^2_{1,\Omega} + \frac{C}{\epsilon}  S_\mesh(\umesh,\umesh).
\end{eqnarray}
For the term $III$, we observe that $I^0_{\meshs}\umeshs-I^0_{\mesh} \umesh\in \Vmeshs^0$ in view of 
$\Vmesh^0\subset \Vmeshs^0$,  which  allows us to apply Lemma \ref{L:Galerkin-orthogonality} (Galerkin quasi-orthogonality) with $\mesh$ replaced by  $\meshs$. Thus, $III$ is zero in the enhanced case \eqref{vem:choice:2} due to \eqref{eq:PGO}. In the other cases,
 bound \eqref{aux:PGO} yields 
\[
\vert III\vert \lesssim S_\meshs(\umeshs,\umeshs)^{1/2} \vert I^0_\meshs \umeshs-
I^0_\mesh \umesh \vert_{1,\Omega}. 
\]
Furthermore, 
\begin{eqnarray}
 \vert I^0_\meshs \umeshs - I^0_\mesh \umesh\vert_{1,\Omega} &\leq&
\vert I^0_\meshs \umeshs - \umeshs\vert_{1,\Omega}+
\vert I^0_\mesh \umesh - \umesh \vert_{1,\Omega}
+
\vert \umeshs -  \umesh\vert_{1,\Omega}\nonumber\\
&\lesssim& S_\meshs(\umeshs,\umeshs)^{1/2}
+S_\mesh(\umesh,\umesh)^{1/2} + \vert \umeshs-\umesh\vert_{1,\Omega},\nonumber
 \end{eqnarray}
 whence using again Young's inequality
 \[
 \vert III \vert \leq \epsilon \vert \umeshs-\umesh\vert^2_{1,\Omega} + {C(1+\frac{1}{\epsilon})} (S_\meshs(\umeshs,\umeshs)+S_\mesh(\umesh,\umesh)). 
 \]
 This completes the proof.
 \end{proof}

\begin{proposition}[comparison of energy errors under refinement]\label{prop:compar-errors}
For any $\delta >0$ there exists a constant $C_E>0$ independent of $\mesh$ such that 
\begin{equation}\label{eq:compar-errors}
  \vvvert u - \umeshs \vvvert^2  \leq  (1+\delta) \vvvert u - \umesh \vvvert^2 {- \vvvert \umeshs - \umesh \vvvert^2}
  + { C_E}\left(1+\frac1\delta \right)  \big( \Smesh(\umesh, \umesh) + \Smeshs(\umeshs, \umeshs) \big).
\end{equation}
\end{proposition}
\begin{proof}
We first observe that
\begin{equation*}\label{aux:err-reduction}
\mathcal{B}(u-\umeshs,u-\umeshs)=
\mathcal{B}(u-\umesh,u-\umesh)
-\mathcal{B}(\umesh-\umeshs,\umesh-\umeshs)
+2\mathcal{B}(u-\umeshs,\umesh-\umeshs) \,.
\end{equation*}
and that \eqref{lack:orthog} gives an estimate for the last term. This, in conjunction with \eqref{norm:equiv}, yields for all $\epsilon>0$
\begin{align*}
\vvvert u -\umeshs\vvvert^2&= \vvvert u-\umesh\vvvert^2 - \vvvert \umeshs-\umesh \vvvert^2+
2 \mathcal{B}(u-\umeshs,\umesh-\umeshs)
\\
&\leq  \vvvert u-\umesh\vvvert^2 + K \epsilon\vvvert u-\umeshs \vvvert^2
-(1-K\epsilon)\vvvert \umesh-\umeshs \vvvert^2
\\
& \quad +2C_D\left(1+\frac{1}{\epsilon}\right)S_\meshs(\umeshs,\umeshs)+2C_D\left(1+\frac{1}{\epsilon}\right)S_\mesh(\umesh,\umesh) \,,
\end{align*}
with $K=2 c_\cB^{-1}$, whence
\begin{align*}
\vvvert u -\umeshs\vvvert^2 \leq
{\frac{1}{1-K\epsilon}}\vvvert u-\umesh\vvvert^2
-\vvvert \umesh-\umeshs\vvvert^2
+ \ \frac{2C_D}{1-K\epsilon}\Big(1+\frac{1}{\epsilon}\Big)\big(S_\meshs(\umeshs,\umeshs)+
S_\mesh(\umesh,\umesh)\big).
\end{align*}
For any $\delta>0$, let us define $\epsilon>0$ so that
$$ 
\frac{1}{1-K\epsilon} =1+\delta
$$
namely $\epsilon=\frac{\delta}{K(1+\delta)}$. Inserting this into the previous estimate implies \eqref{eq:compar-errors} for a suitable choice of the constant $C_E$.
\end{proof}

Let us now prove a simple consequence of the above result that exploits Proposition \ref{prop:bound-ST} (bound of the stabilization term by the residual).
\begin{corollary}[quasi-orthogonality of energy errors without stabilization]\label{coroll:quasi-orth}
Let $\gamma$ be the stabilization parameter in \eqref{eq:def-BT}.
Given any $\delta \in (0,\frac14)$, there exists $\gamma_\delta>0$ such that for any $\gamma \geq \gamma_\delta$ it holds
\begin{equation}\label{eq:compar-errors:no-stab}
\begin{split}
\vvvert u - \umeshs \vvvert^2  & \leq  (1+4\delta) \vvvert u - \umesh \vvvert^2 -  \vvvert \umeshs - \umesh \vvvert^2. 
\end{split}
\end{equation}
\end{corollary}
\begin{proof}
{ Let $e=\vvvert u-\umesh\vvvert$, $e_*=\vvvert u-\umeshs\vvvert$, $S=S_\mesh(\umesh,\umesh)$, $S_*=S_\meshs(\umeshs,\umeshs)$
and $E=\vvvert \umesh-\umeshs\vvvert$.
We combine \eqref{eq:bound-ST} with \eqref{apost:stab-free} and \eqref{norm:equiv} to deduce 
$$
S\leq \frac{C_B}{\gamma^2 C_L c_\cB} \, e^2 \qquad \text{and} \qquad  S_*\leq \frac{C_B}{\gamma^2 C_L c_\cB} \, e_*^2. 
$$
Employing these  inequalities in conjunction with \eqref{eq:compar-errors}, we get 
\begin{equation*}
e_*^2   \leq  (1+\delta) e^2 - E^2
+ C_E\left(1+\frac1\delta \right) \frac{C_B}{\gamma^2 C_L c_\cB}  ( e^2+e_*^2) \,,
\end{equation*}
which can be rewritten as 
\begin{eqnarray}\label{eq:e-delta-E}
\left(1-\frac{D}{\gamma^2}\right)e_*^2 \leq \left(1+\delta +\frac{D}{\gamma^2}\right)e^2 - E^2
\end{eqnarray}
with $D=\left(1+\frac1\delta \right) \frac{C_E C_B}{C_L c_\cB}$. 
Let us choose $\gamma$ such that 
\begin{equation}\label{eq:bound-on-gamma}
\frac{D}{\gamma^2} \leq \delta\,, \qquad \text{i.e.}, \qquad \gamma^2 \geq  \gamma_\delta^2 := \frac{C_E C_B}{C_L c_\cB} \frac{1+\delta}{\delta^2}\,.
\end{equation}
Then, \eqref{eq:e-delta-E} yields $(1-\delta)e_*^2 \leq (1+2 \delta)e^2 - E^2$,  which in turn implies
$$
e_*^2 \leq \frac{1+2 \delta}{1-\delta}e^2 - E^2 \,.
$$
We conclude the proof by observing that $\frac{1+2 \delta}{1-\delta} \leq 1+4\delta$ if $\delta \leq \frac14$.
}
\end{proof}

\subsubsection*{c. Proof of Theorem \ref{T:contraction} (contraction property of {\tt GALERKIN})}
%
To simplify notation, set $e^2= \vvvert u - \umesh \vvvert^2$, $e^2_{*}= \vvvert u - \umeshs \vvvert^2$, $E^2=\vvvert \umesh-\umeshs\vvvert$, $\eta^2=\etamesh^2(\umesh,\data)$, $\eta_{*}^2=\etameshs^2(\umeshs,\data)$ and $S=S_\mesh (\umesh,\umesh)$.
By employing \eqref{eq:compar-errors:no-stab} together with \eqref{estimator_reduction} and \eqref{norm:equiv}, we obtain 
\begin{eqnarray}
e_{*}^2+\beta \eta_{*}^2&\leq& 
(1+4\delta)e^2 
+ \Big[{\frac{\beta C_{er,2}}{c_\cB}-1}\Big]E^2
+\beta \rho \eta^2 + \beta C_{er,1} S \,. \nonumber
\end{eqnarray}
Choose $\beta={\frac{c_\cB}{C_{er,2}}}$ and recall that 
$ S\leq \frac{C_B}{\gamma^2}\eta^2 $ from \eqref{eq:bound-ST}. This implies
\begin{equation*}
e_{*}^2+\beta \eta_{*}^2 \leq \left(1+4 \delta\right)e^2 +\beta \Big(\rho+\frac{C_{er,1} C_B}{\gamma^2} \Big)\eta^2.
\end{equation*}
The coefficient  of $\eta^2$ satisfies
$$ 
\beta \Big(\rho+\frac{C_{er,1}C_B}{\gamma^2} \Big) \leq \beta \frac{1+\rho}{2}
$$
provided 
$$ \frac{C_{ er,1}C_B}{\gamma^2} \leq \frac{1-\rho}{2}.$$
Recalling the condition \eqref{eq:bound-on-gamma} on $\gamma$, which stems from the proof of Corollary \ref{coroll:quasi-orth} (quasi-orthogonality of energy errors without stabilization), we thus impose
$$
\gamma^2 \geq \max\left(\frac{2C_{er,1}C_B}{1-\rho},\frac{{ C_E} C_B}{C_L c_\cB} \frac{1+\delta}{\delta}\right).
$$ 
Therefore, we get
\begin{equation*}
e_{*}^2 + \beta \eta_{*}^2 \leq (1+4\delta)e^2 + \beta \frac{1+\rho}{2}\eta^2 = (1+4\delta)e^2 -\beta \frac{1-\rho}{4}\eta^2+
\beta \frac{3+\rho}{4}\eta^2 \,.
\end{equation*}
Rewriting the a posteriori error bound \eqref{apost:stab-free} as
$e^2 \leq {c^\cB}C_{apost}\left(1+\frac{C_B}{\gamma^2}\right)\eta^2$,
we obtain
\begin{eqnarray*}
e_{*}^2 + \beta \eta_{*}^2 &\leq & \left( (1+4\delta) - \frac{\beta(1-\rho)}{4 {c^\cB} C_{apost}(1+\frac{C_B}{\gamma^2})}\right)e^2
+\beta \frac{3+\rho}{4}\eta^2.
\end{eqnarray*}
We finally choose $\delta$. { Let us assume that $\gamma\geq 1$ and} let us pick $\delta$ satisfying 
$$ 
\delta\leq \frac{\beta(1-\rho)}{20{c^\cB}C_{apost}(1+C_B)}\leq \frac{\beta(1-\rho)}{20 { c^\cB} C_{apost}(1+\frac{C_B}{\gamma^2})} \,,
$$
which implies
\begin{equation}
e_{*}^2 + \beta \eta_{*}^2\leq (1-\delta)e^2 + \frac{3+\rho}{4}\beta \eta^2 \leq \alpha(e^2 + \beta \eta^2)
\end{equation}
provided 
$$
\alpha=\max (1-\delta,\frac{3+\rho}{4})<1.
$$
We eventually realize that the final choice of parameters is
\begin{equation*}
  \beta= \frac{c_\cB}{C_{er,2}} \,,
  \quad
  \delta= \min \left(\frac 1 4, \frac{\beta(1-\rho)}{20 c^\cB C_{apost}(1+C_B)} \right) \, ,
  \quad
 \gamma^2 \geq \max \left(1,\frac{2C_{er,1}C_B}{1-\rho},\frac{{C_E} C_B}{C_L c^\cB} \frac{1+\delta}{\delta}\right) \, ,
\end{equation*}
which is admissible. This concludes the proof of Theorem \ref{T:contraction}.

\section{Scaled Poincar\'e inequality in $\Vmesh$: proof of Proposition \ref{prop:scaledPoincare}}\label{sec:poincare}

We first introduce some useful definitions. Let $\mathbb{T}$ denote the infinite binary tree obtained by newest-vertex bisection from the initial partition $\mesh_0$. If $T \in \mathbb{T}$ is not a root, denote its parent by $A(T)$, and let ${\cal A}(T)$ the chain of its ancestors, i.e.,
$$
{\cal A}(T) = \{A_1(T)=A(T), A_j(T)=A(A_{j-1}(T)) \text{ for } j\geq 2 \text{ until the root is reached} \}.  
$$
Given an integer $m \geq 1$, let ${\cal A}_m(T)$ be the subchain containing the first $m$ ancestors of $T$. 

Given any $T \in \mesh$, with vertices $\bm{x}_1$, $\bm{x}_2$, $\bm{x}_3$, define the cumulative index of $T$ to be
$$
\lambda(T) := \sum_{i=1}^3 \lambda(\bm{x}_i) \,,
$$
where $\lambda(\bm{x}_i)$ is the global index of the node $\bm{x}_i$.

Let $v \in \Vmesh$ satisfy $v(\bm{x})=0$ for all $ \bm{x} \in {\cal P}$. We divide the proof into several steps.

\noindent {\it Step 1. Local bounds of norms.} 
Let $E \in \mesh$ be fixed.  If  one of its vertices is a proper node, we immediately have
\begin{equation}\label{eq:B1}
h_E^{-2} \Vert v \Vert_{0,E}^2 \lesssim  | v |_{1,E}^2\,.
\end{equation}
So, from now on, we assume that none of the vertices of $E$ is a proper node. Since $v$ need not vanish in $E$, we use the inequality
\begin{equation}\label{eq:B2}
h_E^{-2} \Vert v \Vert_{0,E}^2 \lesssim  \left( |v(\bm{x}_0)|^2+| v |_{1,E}^2 \right)\,,
\end{equation}
where $\bm{x}_0$ is any point in $E$. Let us choose $\bm{x}_0$ as the newest vertex of $E$. In the two previous inequalities, the hidden constants only depend on $\text{dim} \, \VE$, which by \eqref{eq:cond-VE} and Remark \ref{rem:bound-global-index} can be bounded by $3\cdot 2^\Lambda$.

\noindent {\it Step 2. Path to a proper node.} Denote by $i, j, k$ the global indices of the vertices of $E$, with $i$ being the global index of $\bm{x}_0$; by assumption, they are all $>0$.
Consider the parent $T=A(E)$ of $E$, and let $\ell \geq 0$ be the global index of the vertex of $T$ not belonging to $E$. We claim that
\begin{equation}\label{eq:B3}
\ell < i \,.
\end{equation}
To prove this, observe that $\bm{x}_0$ is the midpoint of an edge $e$ of $T$, whose endpoints have global indices $\ell$ and (say) $k$ (see Fig. \ref{fig:parent-element}). 
\begin{figure}[t!]
\begin{center}
\begin{overpic}[scale=0.15]{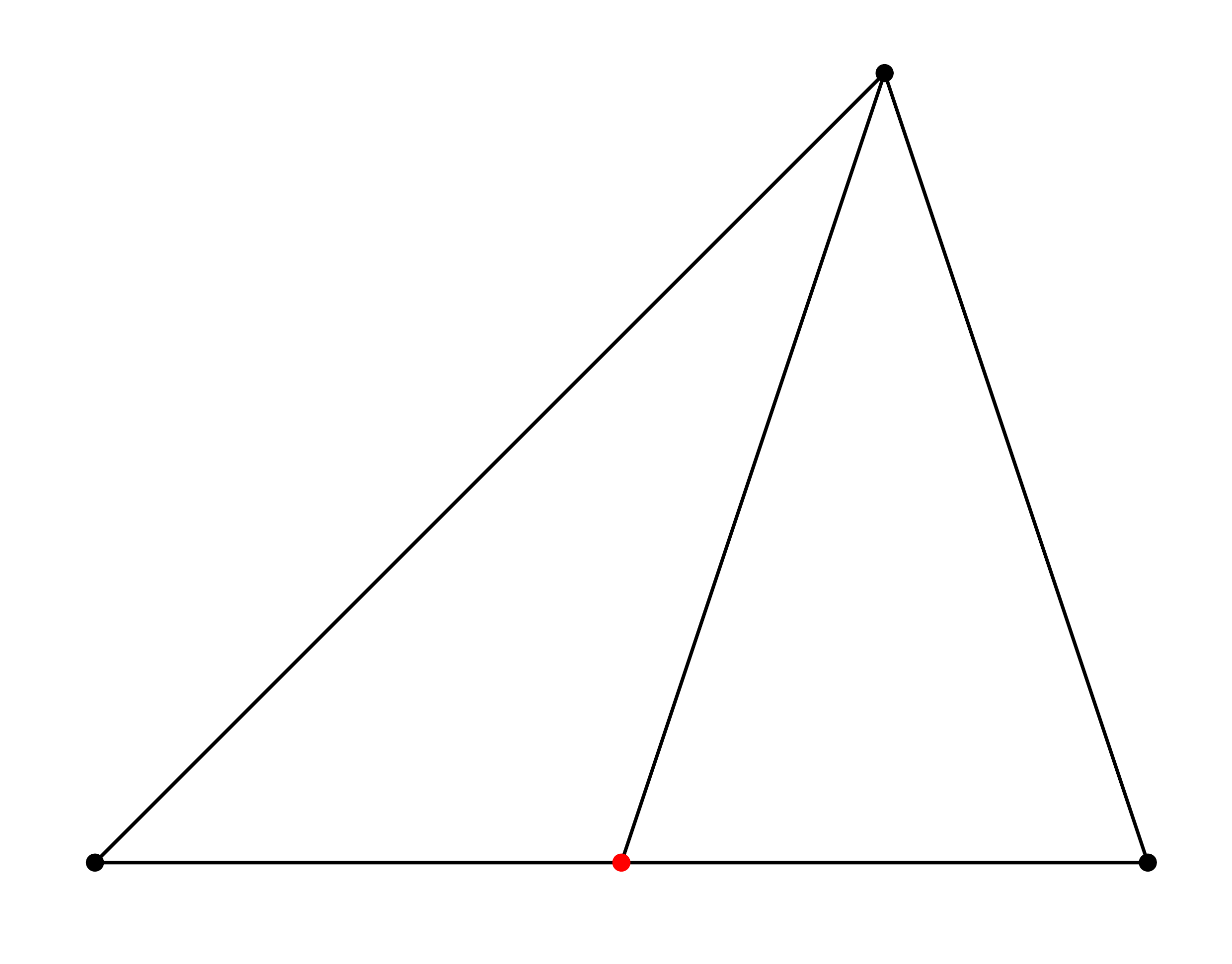}
\put(40, 28){\Large{${E}$}}
\put(40,  11){$\bm{x}_0$}
\put(60, 40){\Large{${T}$}}
\put(50,  1){$i$}
\put(75, 75){$j$}
\put(3,   1){$k$}
\put(95, 1){$\ell$}
\end{overpic}
\end{center}
\caption{{Sample of element $E \in \mesh$, where $T=A(E)$ is the parent of $E$, $i, j, k$ are the global indices of the vertices of $E$ ($i$ being the global index of $\bm{x}_0$), and $\ell$ is global index of the remaining vertex of $T$.}}
\label{fig:parent-element}
\end{figure}
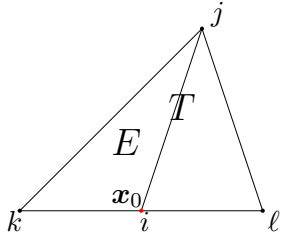
By definition of global index of a hanging node, it holds $i= \max(k, \ell)+1$: if $k \leq \ell$, then $i=\ell+1$ whence $\ell =i-1< i$, whereas if $\ell < k$, then $i=k+1$ whence $\ell < i-1<i$.
Inequality \eqref{eq:B3} implies that the cumulative index decreases in going from $E$ to $T$, i.e.,
$\lambda(E) > \lambda(T)$.

By repeating the argument above with $E$ replaced by $T$, and then arguing recursively,  we realize that when we move along the chain ${\cal A}(E)$ of ancestors of $E$, the cumulative index strictly decreases by at least 1 unit each time, until it becomes $< 3$, indicating the presence of a proper node. In this way, after observing that $\lambda(E) \leq 3\Lambda$ by Definition \ref{def:Lambda-partitions}, we obtain the existence of a subchain of ancestors
\begin{equation}\label{eq:B7}
{\cal A}_M(E) = \{T_0=E, T_1, \dots, T_M\}  
\end{equation}
 with the following properties:
\begin{enumerate}
\item $T_M$ is the first element in the chain which has a proper node,  say $\bm{x}^P$, as a vertex;
\item $M < 3\Lambda$;
\item for $m>0$, each $T_m$ has an edge $g_m$ whose midpoint is the newest vertex of $T_{m-1}$;
 the path $g_1 \rightarrow g_2 \rightarrow \cdots \rightarrow g_M$  connects the node $\bm{x}_0 \in E$ (the midpoint of $g_1$) to the proper node $\bm{x}^P \in T_M$ (the endpoint of $g_M$)
(see Fig. \ref{fig:path}).
\end{enumerate}

\vspace{0.5cm}
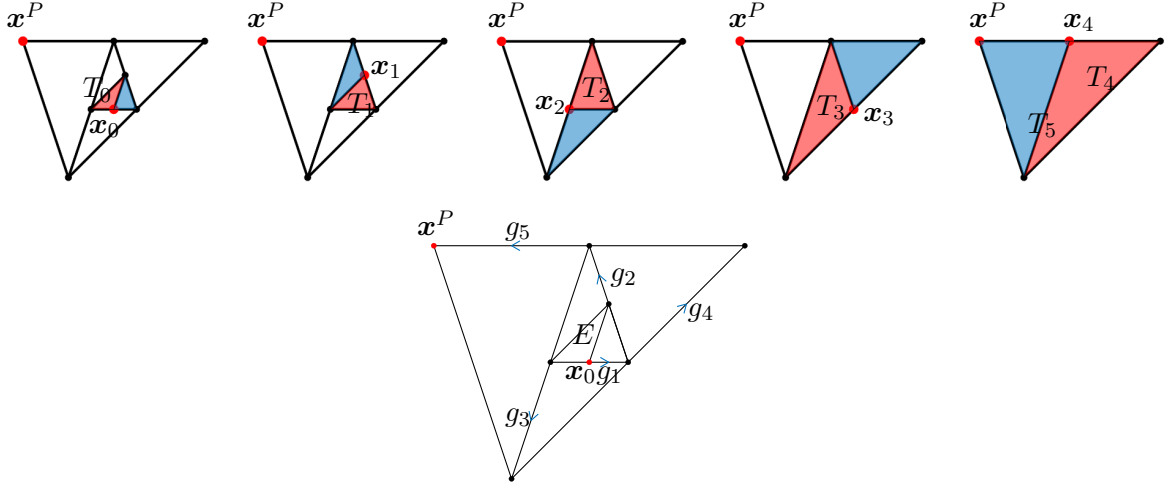
\begin{figure}[t!]
\begin{center}
\begin{overpic}[scale=0.48]{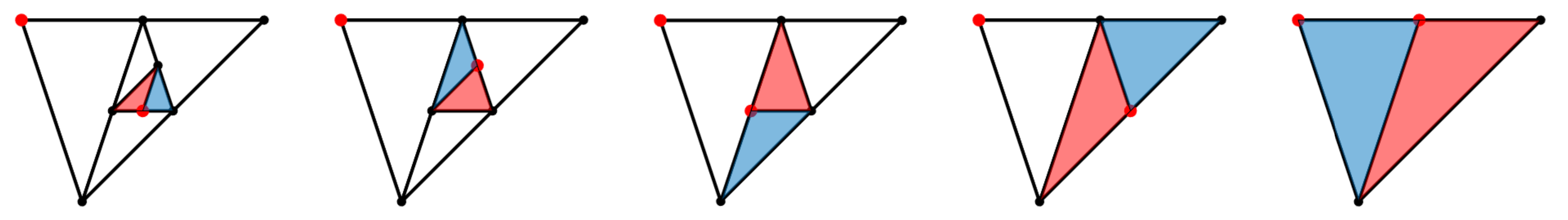}
\put( 0,  14){$\bm{x}^P$}
\put(20.5,  14){$\bm{x}^P$}
\put(41,  14){$\bm{x}^P$}
\put(61.5,  14){$\bm{x}^P$}
\put(82,  14){$\bm{x}^P$}
\put( 7,  5){$\bm{x}_0$}
\put(31, 10){$\bm{x}_1$}
\put(45,  7){$\bm{x}_2$}
\put(73,  6){$\bm{x}_3$}
\put(90,  14){$\bm{x}_4$}
\put(92,  9){$T_4$}
\put(87,  5){$T_5$}
\put(69,  6.5){$T_3$}
\put(49,  8){$T_2$}
\put(29,  6.7){$T_1$}
\put(6.5,  8){$T_0$}
\end{overpic}
\\
\vspace{0.5cm}
\begin{overpic}[scale=0.20]{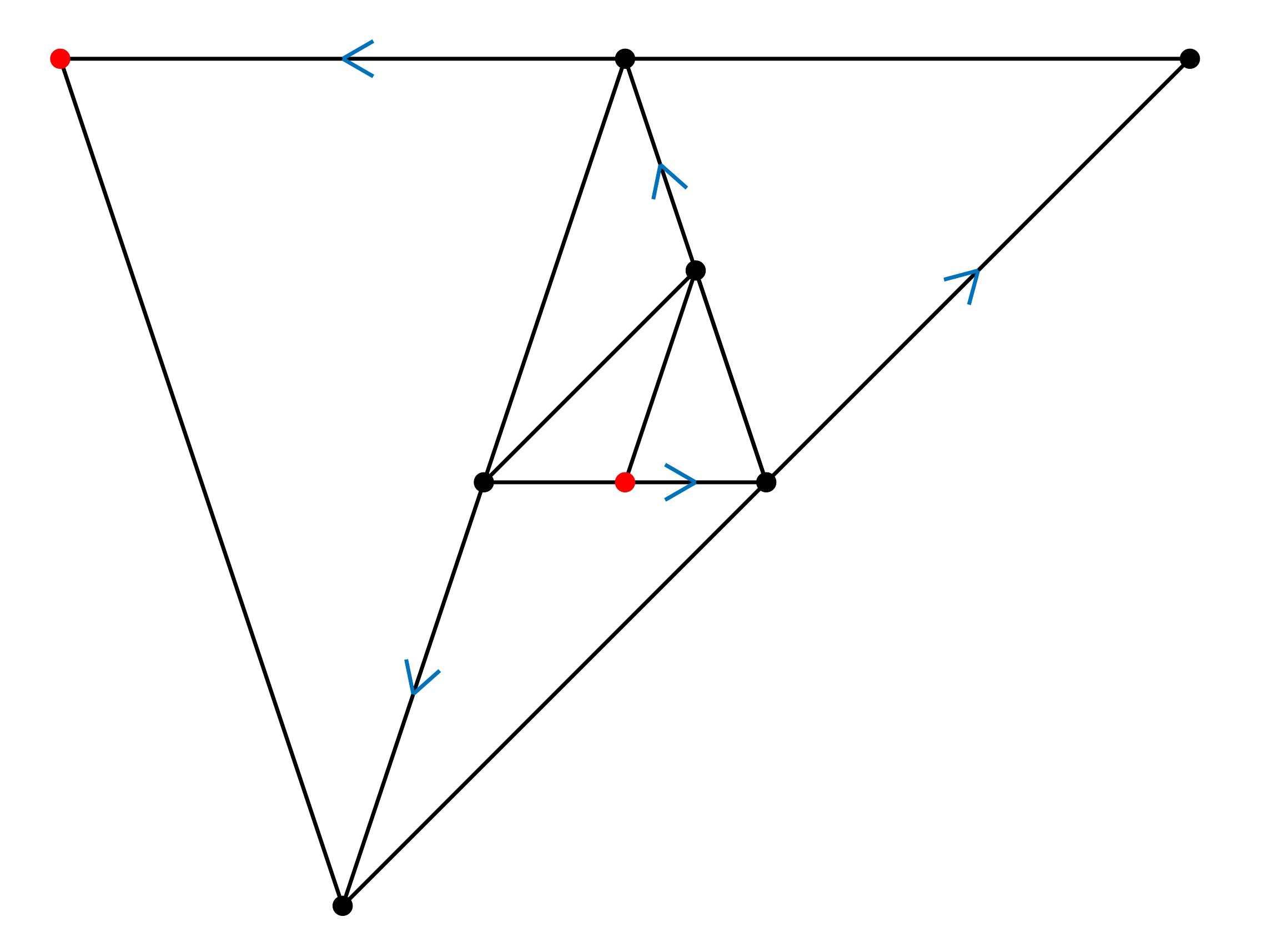}
\put(44, 42){$E$}
\put(42, 32){$\bm{x}_0$}
\put(51, 32){$g_1$}
\put(55, 60){$g_2$}
\put(25, 20){$g_3$}
\put(77, 50){$g_4$}
\put(25, 73){$g_5$}
\put(0, 73){$\bm{x}^P$}
\end{overpic}
\end{center}
\caption{{The chain of ancestors of $E=T_0$} leading to $T_5$ having a proper node as a vertex (above). 
The path $g_1 \rightarrow g_2 \rightarrow \cdots \rightarrow g_5$  connecting $\bm{x}_0 \in E$ to the proper node $\bm{x}^P \in T_5$ (below)}
\label{fig:path}
\end{figure}

\noindent {\it Step 3. Properties of edges with hanging nodes.} Consider an edge $g$ shared by two triangles $T, T' \in \mathbb{T}$, with $T' \in \mesh$; suppose that the midpoint $\hat{\bm{x}}$ of $g$ is a hanging node for $T'$, created by a refinement of $T$ to produce elements in $\mesh$. 
Then, $g$ cannot contain proper nodes except possibly the endpoints, since their presence would be possible only by a refinement of $T'$, which is ruled out by the assumption on $\hat{\bm{x}}$.

Consequently, the edge $g$ is partitioned by the hanging nodes into a number of edges $\tilde{e}$ of elements $\tilde{E} \in \mesh$ contained in $T$; 
recalling Remark \ref{rem:bound-global-index}, the number of such edges is bounded by $2^\Lambda$.

\noindent {\it Step 4. Sequence of elements along the path.} We apply the conclusions of {\it Step 3.} to each edge $g_m$ of the path defined in {\it Step 2}. We obtain the existence of a sequence of edges $e_n$ ($1 \leq n \leq N_E$ for some integer $N_E$) and corresponding elements $E_n \in \mesh$, such that  (see Fig. \ref{fig:subtriangles})
\begin{enumerate}
\item $e_n \subset \partial E_n$; 
\item $e_n \subseteq g_m$ for some $m$, and correspondingly $E_n \subseteq T_m$, with
$$
|E_n| \simeq |e_n|^2 \geq 2^{-2\Lambda} |g_m|^2 \simeq 2^{-2\Lambda} |T_m| \,,
$$
where the hidden constants only depend on the shape of the initial triangulation but not on $\Lambda$;
\item the number $N_E$ of such elements is bounded by $M2^\Lambda < 3\Lambda 2^\Lambda$;
\item writing $e_n = [\bm{x}_{n-1},\bm{x}_n]$, then $\bm{x}_0$ is the newest vertex of $E$, whereas $\bm{x}_N$ is the proper node $\bm{x}^P$.
\end{enumerate}
\begin{figure}[b!]
\begin{center}
\begin{overpic}[scale=0.35]{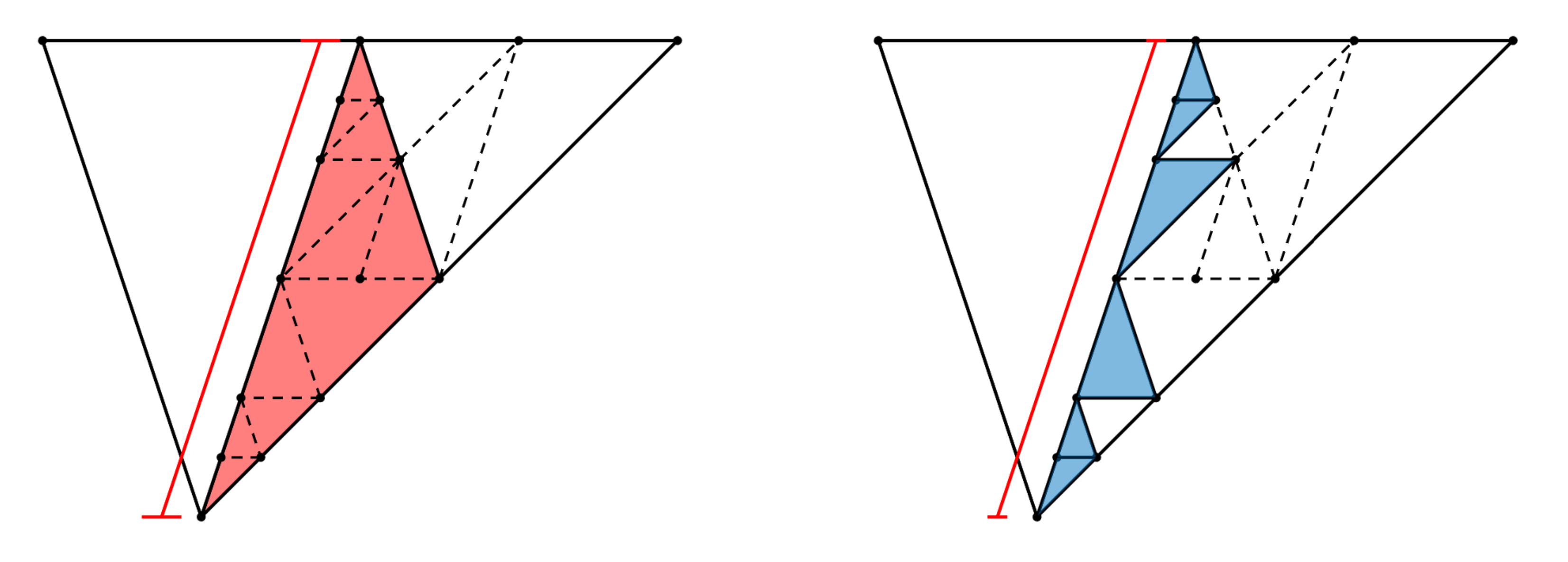}
\put(19, 14){\Large{$T_m$}}
\put(11, 18){$g_m$}
\put(65, 18){$g_m$}
\put(70, 22){$e_n$}
\put(73, 23){$E_n$}
\end{overpic}
\end{center}
\caption{{The edge $g_m$ of the ancestor $T_m$ of $E$ (left). The partition of $g_m$ into edges $e_n$ of elements $E_n \in \mesh$ (right)}}
\label{fig:subtriangles}
\end{figure}
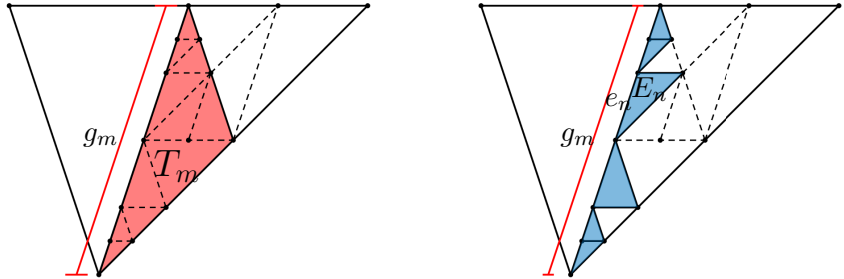

\noindent {\it Step 5. Bound of $|v(\bm{x}_0)|$.} Let us write
$v(\bm{x}_0) = v(\bm{x}_0) - v(\bm{x}^P) = \sum_{n=1}^{N_E} \left( v(\bm{x}_{n-1}) - v(\bm{x}_n) \right) = \sum_{n=1}^{N_E} \nabla v_{| E_n} \cdot (\bm{x}_{n-1} - \bm{x}_n )$.  Then,
$$
|v(\bm{x}_0) |^2 \leq N_E \sum_{n=1}^{N_E} \|\nabla v_{| E_n}\|^2 \Vert \bm{x}_{n-1} - \bm{x}_n \Vert^2 \lesssim N_E \sum_{n=1}^{N_E} \|\nabla v_{| E_n}\|^2 |E_n| = N_E \sum_{n=1}^{N_E} |v|_{1,E_n}^2\,.
$$
Inserting this into \eqref{eq:B2} yields
\begin{equation}\label{eq:B4}
h_E^{-2} \Vert v \Vert_{0,E}^2 \lesssim  \left( | v |_{1,E}^2 + N_E \sum_{n=1}^{N_E} |v|_{1,E_n}^2 \right)\,.
\end{equation}
Taking into account also \eqref{eq:B1}, we end up with the bound
\begin{equation}\label{eq:B5}
\sum_{E \in \mesh} h_E^{-2} \Vert v \Vert_{0,E}^2 \lesssim  \left( \sum_{E \in \mesh} | v |_{1,E}^2 + 3\Lambda 2^\Lambda \sum_{E \in \mesh^h}  \sum_{n=1}^{N_E} |v|_{1,E_n}^2 \right)\,,
\end{equation}
where $\mesh^h$ denotes the set of elements in $\mesh$ whose vertices are all hanging nodes. 
Thus, we will arrive at the desired result \eqref{eq:B0} if we show that an element $E_n \in \mesh$ may occur in the double summation on the right-hand side a number of times bounded by some constant depending only on $\Lambda$.

\medskip

\noindent {\it Step 6. Combinatorial count.} Let $E_\star = E_n$ be such an element, which is contained in some triangle $T_\star = T_m$ according to {\it Step 4.}, where $T_\star$ belongs to a subchain of ancestors  ${\cal A}_M(E)$ defined in \eqref{eq:B7}, for some $E \in \mesh^h$.

Notice that $T_\star \in {\cal A}(E_\star)$ and its measure satisfies $|T_\star| \lesssim 2^{2\Lambda} |E_\star| $ by {\it Step 4}. Thus, the number $K$ of admissible $T_\star$'s satisfies
$$
2^K |E_\star| = |T_\star| \lesssim 2^{2\Lambda} |E_\star|\,,
$$ 
whence $K \leq 2\Lambda +c$ for some $c$ independent of $\Lambda$.

On the other hand, a triangle $T_\star$ may belong to a subchain of ancestors  ${\cal A}_M(E)$, for at most $2^{3\Lambda}$ descendants $E$, since we have seen in {\it Step 2.} that $M < 3\Lambda$.

We conclude that $E_\star$ may occur in the double summation on the right-hand side of \eqref{eq:B5} at most $(2\Lambda +c)2^{3\Lambda}$ times. This concludes the proof of Proposition 
\ref{prop:scaledPoincare}.

\section{Interpolation errors: proof of Proposition \ref{prop:compareInterp}}\label{sec:bound-Interp}

This section is devoted to the proof of Proposition \ref{prop:compareInterp}, which is crucial for the proof of Proposition \ref{prop:bound-ST} in the next section. 

Note that by the triangle inequality
$$
| v-\Imeshz v|_{1,\Omega} = | v-\Imeshz v|_{1,\mesh} \leq | v-\Imesh v|_{1,\mesh} + | \Imesh v-\Imeshz v|_{1,\mesh}  \,,
$$
it is enough to prove the bound
\begin{equation}\label{eq:pre-bound-interp}
| \Imesh v-\Imeshz v|_{1,\mesh}  \lesssim | v-\Imesh v|_{1,\mesh} \,.
\end{equation}
To this end, we need several preparatory results that allow us to express both semi-norms as  sums of hierarchical details. Let us start by considering the right-hand side.

Let $E$ be any element in $\mesh$. Define
\begin{eqnarray*}
&& {\cal N}_E = \{\bm{x} : \bm{x} \text{ is a node of } \mesh \text{ sitting on }\partial E\}, \\
&& {\cal V}_E = \{\bm{x} : \bm{x} \text{ is a vertex of } E\}, \\
&& {\cal H}_E = {\cal N}_E \setminus {\cal V}_E = \{\bm{x} : \bm{x} \text{ is a hanging node for }E\}.
\end{eqnarray*}

To each function $v \in \VE$ we associate a vector $d(v) = \{ d(v,\bm{z}) \}_{\bm{z} \in {\cal N}_E}$ 
that collects the following values, so called {\it hierarchical details} of $v$: 
\begin{equation}\label{eq:def-d}
d(v,\bm{z}) = \begin{cases} v(\bm{z}) & \text{if } \bm{z} \in {\cal V}_E, \\
v(\bm{z})- \frac12(v(\bm{z}')+v(\bm{z}'') ) & \text{if } \bm{z} \in {\cal H}_E \,,
\end{cases}
\end{equation}
where for $\bm{z} \in {\cal H}_E$ we denote by $ \bm{z}', \bm{z}'' \in  {\cal B}(\bm{z})$ the endpoints of the edge halved to create $\bm{z}$.
While the collection $\{ v(\bm{z}) \}_{\bm{z} \in {\cal N}_E}$ represents the coefficients expressing $v \in \VE $ in terms of the (local) dual basis associated to the degrees of freedom, the collection $\{ d(v,\bm{z}) \}_{\bm{z} \in {\cal N}_E}$ represents the coefficients with respect to a hierarchical-type basis. 

The following lemma introduces a relation between the $H^1$ semi-norm of a function $v \in \VE$ and the Euclidean norm of $d(v)$.

\begin{lemma}[local interpolation error vs hierarchical details]
Let $\mesh$ be $\Lambda$-admissible. For all $E \in \mesh$ the relation
\begin{equation}\label{eq:norm-equiv-Yser4}
\sum_{\bm{x} \in {\cal H}_E} d^2(v,\bm{x})  \lesssim  | v-{\cal I}_E v 
 |_{1,E}^2 \lesssim \sum_{\bm{x} \in {\cal H}_E} d^2(v,\bm{x}),
\qquad \forall v \in \V_E\,
\end{equation}
holds with hidden constants only depending on $\Lambda$.
\end{lemma}
\begin{proof}
We recall that, according to \cite{BLR:2017,brenner-sung:2018}, the stability property \eqref{eq:stab-sE} holds true for the particular choice \eqref{eq:stab-dofidofi} of stabilization
(which need not be the stabilization used in our Galerkin scheme). Hence, we obtain
\begin{equation}\label{pippo-1}
| v-{\cal I}_E v |_{1,E}^2 
\simeq \sum_{\bm{x} \in {\cal H}_E} | v(\bm{x})-{\cal I}_E(\bm{x})|^2 
\quad \forall v \in \V_E \, ,
\end{equation}
where the symbol $\simeq$ denotes an equivalence up to uniform constants.
Since $d(v,\bm{x}) = d(v-{\cal I}_E v,\bm{x})$ for $\bm{x} \in {\cal H}_E$, equation \eqref{pippo-1} yields that
\eqref{eq:norm-equiv-Yser4} is equivalent to
$$
\sum_{\bm{x} \in {\cal H}_E} d^2(v-{\cal I}_E,\bm{x})  \simeq  
\sum_{\bm{x} \in {\cal H}_E} | v(\bm{x})-{\cal I}_E(\bm{x})|^2
\qquad \forall v \in \V_E \,,
$$
which in turn corresponds to
\begin{equation}\label{pippo-2}
\sum_{\bm{x} \in {\cal H}_E} d^2(w,\bm{x})  \simeq  
\sum_{\bm{x} \in {\cal H}_E} | w(\bm{x})|^2
\qquad \forall w \in \VV_E \,,
\end{equation}
where $\VV_E = \{ v \in \VE \ : \ v(\bm{x})=0 \ \forall \bm{x} \in {\cal V}_E \} $.
The two quantities appearing in \eqref{pippo-2} are norms on the finite dimensional space $\VV_E$. Furthermore, since both depend only on $\{ w(\bm{x}) \}_{\bm{x} \in {\cal H}_E}$, the values of $w$ at the hanging nodes, such norms do not depend on the position (that is, the coordinates) of the nodes of $E$. On the other hand, an inspection of \eqref{eq:def-d} reveals that the norm at the left hand side depends on the particular node ``pattern'' on $E$, i.e. on which sequential edge subdivision led to the appearance of hanging nodes on $\partial E$.
Nevertheless, since $\mesh$ is $\Lambda$-admissible,
not only the number of hanging nodes is uniformly bounded, but also the number of possible patterns is finite. As a consequence, property \eqref{pippo-2} easily follows from the equivalence of norms in finite dimensional spaces, with hidden constants only depending on $\Lambda$. 
\end{proof}

In order to get the global equivalence, we observe that the set ${\cal H}=\bigcup_{E \in \mesh} {\cal H}_E$ of all hanging nodes of $\mesh$ is a disjoint union, i.e., $\bm{x} \in {\cal H}$ if and only if  there exists a unique $E \in \mesh$ such that $\bm{x}\in {\cal H}_E$. Combining this with \eqref{eq:norm-equiv-Yser4}, and recalling that $(v-\Imesh v)_{|E}= v_{|E} - {\cal I}_E v_{|E}$ for any $E \in \mesh$, we obtain the following result.
\begin{corollary}[global interpolation error vs hierarchical details]\label{lem:norm-equiv-1}
The following semi-norm equivalence holds: 
\begin{equation}\label{eq:norm-equiv-Yser5} 
c_D \sum_{\bm{x} \in {\cal H}} d^2(v,\bm{x})  \leq |v- \Imesh v |_{1,\mesh}^2 \leq  C_D \sum_{\bm{x} \in {\cal H}} d^2(v,\bm{x}),
\qquad \forall v \in \Vmesh\,,
\end{equation}
where the constants $C_D \geq c_D >0$ depend on $\Lambda$ but are independent of the triangulation $\mesh$. \endproof
\end{corollary}

Let us now focus on the left-hand side of \eqref{eq:pre-bound-interp}.
Since $\Imesh v-\Imeshz v$ is affine on each element of $\mesh$, one has
$$
| \Imesh v-\Imeshz v|_{1,\mesh}^2 = \sum_{E \in \mesh} | \Imesh v-\Imeshz v|_{1,E}^2 \simeq 
\sum_{E \in \mesh} \sum_{\bm{x}\in {\cal V}_E} ({\cal I}_E v-\Imeshz v)^2(\bm{x}) \,.
$$
Note that $({\cal I}_E v)(\bm{x})=v(\bm{x})$ if $\bm{x} \in {\cal V}_E$. Furthermore, $(\Imeshz v)(\bm{x}) =v(\bm{x})$ if $\bm{x}$ is a proper node of $\mesh$. Thus, for any $\bm{x} \in {\cal H}$, let us define the detail
$$
\delta(v,\bm{x}) = v(\bm{x})- (\Imeshz v)(\bm{x}),
$$
so that
\begin{equation}\label{eq:norm-equiv-Izero}
| \Imesh v-\Imeshz v|_{1,\mesh}^2 \simeq \sum_{\bm{x} \in {\cal H}} \delta^2(v,\bm{z})\,.
\end{equation}
Recalling \eqref{eq:norm-equiv-Yser5}, the desired result \eqref{eq:pre-bound-interp} follows if we prove the bound
\begin{equation}\label{eq:bound-interp-2}
\sum_{\bm{x} \in {\cal H}} \delta^2(v,\bm{x}) \lesssim \sum_{\bm{x} \in {\cal H}} d^2(v,\bm{x})  \qquad \forall v \in \Vmesh\,.
\end{equation}
From now on, to ease the notation, we assume $v$ fixed and we write $d(\bm{x}):=d(v,\bm{x})$, $\delta(\bm{x}):=\delta(v,\bm{x})$, and $v^*:=\Imeshz v$. 
Setting
$$
\bm{\delta} = \big(\delta(\bm{x}) \big)_{\bm{x} \in {\cal H}} \,, \qquad \bm{d} = \big(d(\bm{x}) \big)_{\bm{x} \in {\cal H}} \,,
$$
the desired inequality \eqref{eq:bound-interp-2} is equivalent to
\begin{equation}\label{eq:rel-delta-d-1}
\Vert \bm{\delta} \Vert_{l^2({\cal H})} \lesssim  \Vert \bm{d} \Vert_{l^2({\cal H})} \,.
\end{equation}

We can relate $\bm{\delta}$ to $\bm{d}$ as follows: let $\bm{x} \in {\cal H}$, and let $\bm{x}', \bm{x}'' \in {\cal B}(\bm{x})$; since $v^*$ is linear on the segment $[\bm{x}', \bm{x}'']$, one has
\begin{eqnarray}\label{eq:rel-delta-d-2}
\delta(\bm{x}) &=& v(\bm{x})-v^*(\bm{x})  =  v(\bm{x})-\tfrac12(v^*(\bm{x}')+v^*(\bm{x}''))  \nonumber \\ 
&= &v(\bm{x})-\tfrac12(v(\bm{x}')+v(\bm{x}'')) + \tfrac12(v(\bm{x}')-v^*(\bm{x}'))  +  \tfrac12(v(\bm{x}'')-v^*(\bm{x}''))  \nonumber \\
&=& d(\bm{x}) + \tfrac12 \delta(\bm{x}')+ \tfrac12 \delta(\bm{x}'').
\end{eqnarray}
Thus, we have $\bm{\delta} = \bm{W} \bm{d}$
for a suitable matrix of weights $\bm{W} : l^2({\cal H}) \to l^2({\cal H})$, and  \eqref{eq:rel-delta-d-1} holds true for any $\bm{d}$ if and only if  
$$
\Vert \bm{W} \Vert_2 \lesssim 1 \,.
$$

To establish this bound, it is convenient to organize the hanging nodes in a block-wise manner according to the values of the global index $\lambda \in [1, \Lambda_{\mesh}]$. Let
$$
{\cal H} = \bigcup_{1 \leq \lambda \leq \Lambda_{\mesh}} {\cal H}_\lambda  \qquad \text{with \ } {\cal H}_\lambda = \{\bm{x} \in {\cal H} : \lambda(\bm{x}) = \lambda \}\,,
$$
and let $\bm{\delta} = \big(\bm{\delta}_\lambda \big)_{1 \leq \lambda \leq \Lambda_{\mesh} }$, $\bm{d} = \big(\bm{d}_\lambda \big)_{1 \leq \lambda \leq \Lambda_{\mesh} }$
be the corresponding decompositions on the vectors $\bm{\delta}$ and $\bm{d}$; then, the matrix $\bm{W}$, considered as a block matrix, can be factorized as
\begin{equation}\label{eq:factorW}
\bm{W} =\bm{W}_{\Lambda_\mesh} \bm{W}_{\Lambda_\mesh-1} \cdots \bm{W}_{\lambda} \cdots \bm{W}_{2} \bm{W}_{1} \,,
\end{equation}
where the lower-triangular matrix $\bm{W}_{\lambda}$ realizes the transformation \eqref{eq:rel-delta-d-2} for the hanging nodes of index $\lambda$, leaving unchanged the other ones. In particular, $\bm{W}_{1} = \bm{I}$ since $\delta(\bm{x}') = \delta(\bm{x}'') =0$ when $\bm{x}'$ and $\bm{x}''$ are proper nodes; on the other hand,  any other $\bm{W}_{\lambda}$ differs from the identity matrix only in the rows corresponding to the block $\lambda$: each such row contains at most two non-zero entries, equal to $\frac12$, in the off-diagonal positions, and 1 on the diagonal (see Fig. \ref{fig:pp}, left).
In order to estimate the norm of $\bm{W}_{\lambda}$, we use H\"older's inequality $\Vert \bm{W}_{\lambda} \Vert_2 \leq \left( \Vert \bm{W}_{\lambda} \Vert_1 \Vert \bm{W}_{\lambda} \Vert_\infty \right)^{1/2}$. Easily one has
$$
\Vert \bm{W}_{\lambda} \Vert_\infty \leq \frac12 + \frac12 + 1 =2\,, \qquad \Vert \bm{W}_{\lambda} \Vert_1 \leq 5 \frac12 +1= \frac72 \,,
$$
the latter inequality stemming from the fact that a hanging node of index $< \lambda$ may appear on the right-hand side of  \eqref{eq:rel-delta-d-2} at most 5 times (since at most 5 edges meet at a node, see Fig. \ref{fig:pp}, right).
\begin{figure}[!htb] 
\begin{center}
        \begin{tikzpicture}[scale=0.5]
\draw[very thick] (0,0) rectangle (8,8);
\draw [line width=0.01mm](0,8) -- (8,0);
\path [draw, thick](0,3) -- (8,3) -- (8,2) -- (0,2) -- cycle;
\draw[very thick] (5,2) rectangle (6,3);
\draw[line width=0.01mm] (0,2.5) -- (8,2.5);
\fill (2,2.5) circle[radius=2pt];
\node at (1.7,2.6) {$\frac{1}{2}$};
\fill (3.5,2.5) circle[radius=2pt];
\node at (3.2,2.6) {$\frac{1}{2}$};
\node at (6,6) {$0$};
\node[draw opacity=1] at (2,6) {$1$};
\node[draw opacity=1] at (7,1) {$1$};
\node[draw opacity=1] at (5.5,2.5) {$1$};

\draw [very thick, decorate,
    decoration = {brace,mirror}] (0,-0.5) --  (5,-0.5)
    node[pos=0.5,below=10pt,black]{{hanging~nodes~of~index~$<\lambda$}};
    
    \draw [very thick, decorate,
    decoration = {calligraphic brace,mirror}] (8.2,2) --  (8.2,3)
node[pos=0.6,right=10pt,black]   {{hanging~nodes
}}
node[pos=0.1,right=10pt,black]   {{
~of~index~$\lambda$}};;
\end{tikzpicture}
\qquad \quad
\begin{overpic}[scale=0.25]{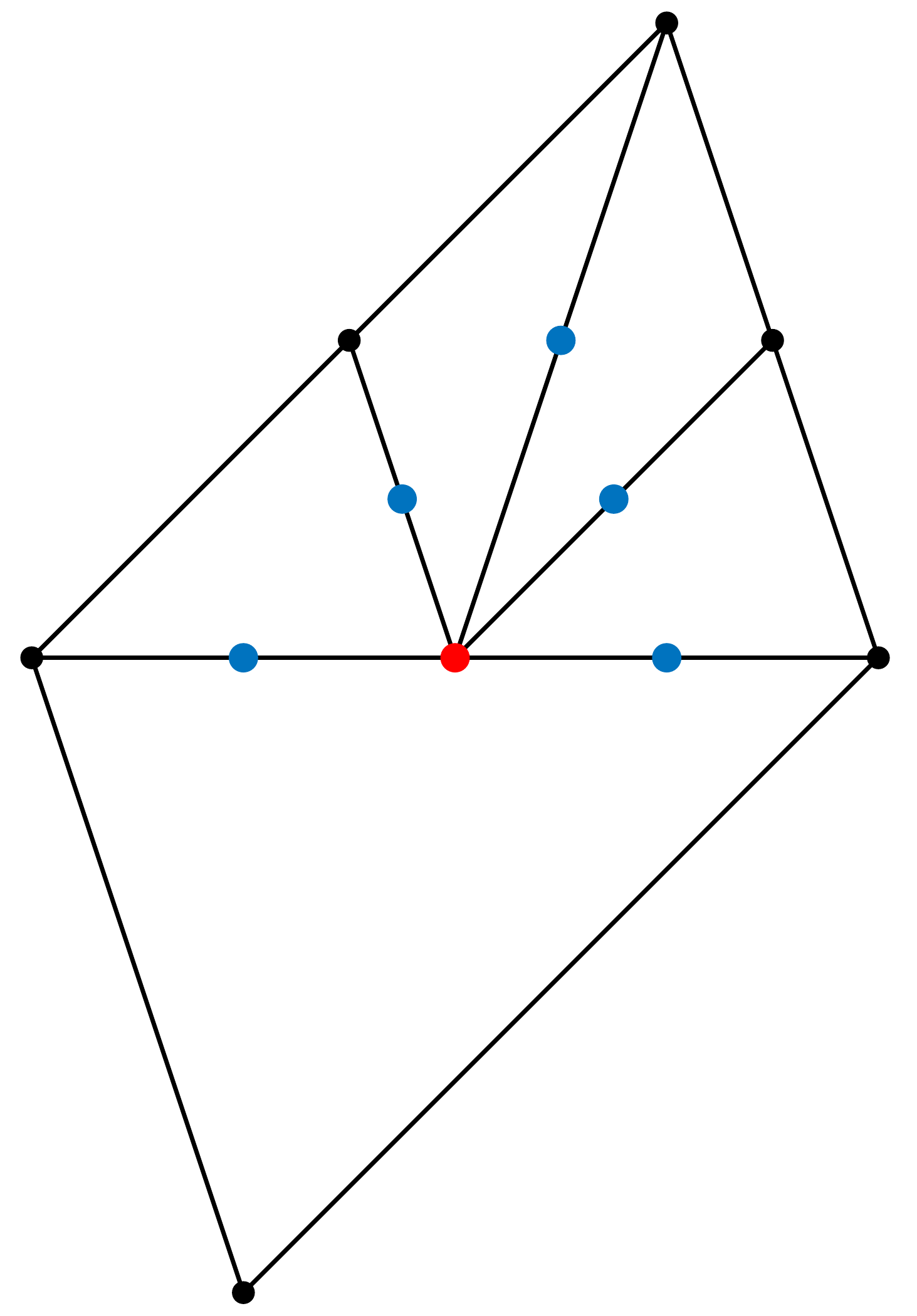}
\put(31,43){$\bm{x}'$}
\put(15,44){$\bm{x_1}$}
\put(22,60){$\bm{x_2}$}
\put(34,76){$\bm{x_3}$}
\put(50,60){$\bm{x_4}$}
\put(50,44){$\bm{x_5}$}
\end{overpic}
\end{center}        
\caption{Left: structure of one factor $\bm{W}_{\lambda}$ in the factorization \eqref{eq:factorW}. 
Right: given an hanging node $\bm{x}'$ with $\lambda(\bm{x}')< \lambda$, at most 5 hanging nodes $\bm{x_i}$ for $i=1, \dots, 5$ of index $\lambda$ may be such that
$\bm{x}' \in B(\bm{x_i})$.}
\label{fig:pp}
\end{figure}
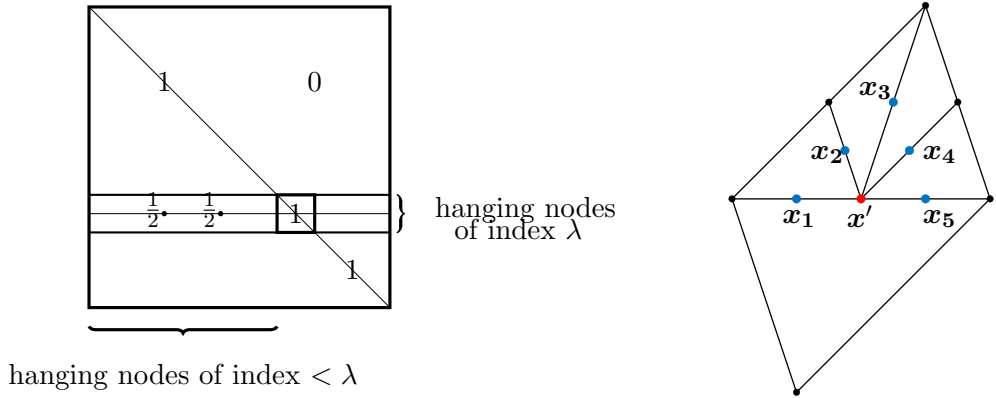
Hence, $\Vert \bm{W}_{\lambda} \Vert_2 \leq 7^{1/2}$, which yields
$$
\Vert \bm{W} \Vert_2 \leq \prod_{2 \leq \lambda \leq \Lambda_\mesh} \Vert \bm{W}_{\lambda} \Vert_2  \leq 7^{(\Lambda_\mesh-1)/2} \leq 7^{(\Lambda-1)/2}
$$
as desired. Concatenating \eqref{eq:norm-equiv-Izero}, \eqref{eq:bound-interp-2}  and \eqref{eq:norm-equiv-Yser5}, we conclude the proof of Proposition \ref{prop:compareInterp}.

\section{Bounding the stabilization term by the residual: proof of Proposition \ref{prop:bound-ST}}\label{sec:bound-ST}

By \eqref{eq:propB1} and the definition \eqref{eq:def-BT}, for all $w \in \Vmeshz$ we obtain
\begin{equation*}
\gamma  \Smesh(\umesh, \umesh) = \gamma \Smesh(\umesh, \umesh-w)= \Bmesh(\umesh, \umesh-w) - \amesh(\umesh, \umesh-w)  - \mmesh(\umesh, \umesh-w)
\,.
\end{equation*}
By \eqref{def-Galerkin}, \eqref{eq:estim1}, the scaled Poincar\'e inequality and the continuity of $\Pimesh$ with respect to the $H^1$ broken seminorm we get
\begin{equation*}
\begin{aligned}
\Bmesh(\umesh, \umesh-w) &- \mmesh(\umesh, \umesh-w) = 
(f - c\Pimesh \umesh , \Pimesh (\umesh - w) )_\Omega\\
& \leq \sum_{E \in \mesh} h_E \Vert \rmesh(E;\umesh,\data) \Vert_{0,E} \, h_E^{-1} 
\left(\Vert \umesh - w \Vert_{0,E} + 
h_E \vert \umesh - w \vert_{1, E} \right)\,.
\end{aligned}
\end{equation*}
On the other hand, by \eqref{eq;def-aT}, \eqref{eq:def-PinablaE}  and \eqref{eq:estim1}
\begin{equation*}
\begin{split}
\amesh(\umesh, \umesh-w) &= \sum_{E \in \mesh} (A_E \nabla \PiE \umesh, \nabla \PiE (\umesh -w))_E = \sum_{E \in \mesh} (A_E \nabla \PiE \umesh, \nabla (\umesh -w))_E \\
& = \sum_{E \in \mesh} ( (A_E \nabla \PiE \umesh)\cdot \boldsymbol{n}, \umesh - w)_{\partial E}  = \sum_{e \in \edge} ( \jmesh(e;\umesh,\data), \umesh-w)_e  \\
& \leq  \tfrac12 \sum_{E \in \mesh} \sum_{e \in \edgeE} h_E^{1/2} \Vert  \jmesh(e;\umesh,\data) \Vert_{0,e} \ h_E^{-1/2} \Vert \umesh - w \Vert_{0,e} \,.
\end{split}
\end{equation*}
Recalling \eqref{eq:estim3} and \eqref{eq:estim4}, we thus obtain for any $\delta >0$
\begin{equation}\label{eq:bound-ST1}
\gamma \Smesh(\umesh, \umesh) \leq \frac1{2\delta} \, \etamesh^2(\umesh, \data) + \frac{\delta}{2} \, \Phi_\mesh(\umesh-w) \qquad \forall w \in \Vmeshz \,,
\end{equation}
with 
$$
\begin{aligned}
\Phi_\mesh(\umesh-w) &= \sum_{E \in \mesh} \left( h_E^{-2} \Vert \umesh - w \Vert_{0,E}^2 + \vert \umesh - w \vert_{1,E}^2 + \sum_{e \in \edgeE} h_E^{-1} \Vert \umesh - w \Vert_{0,e}^2 \right) 
\\
&\lesssim 
\sum_{E \in \mesh} \left(h_E^{-2} \Vert \umesh - w \Vert_{0,E}^2 + |\umesh - w|^2_{1,E}\right)  \,.
\end{aligned}
$$ 
At this point, we choose $w = \Imeshz \umesh$ in \eqref{eq:bound-ST1} and apply \eqref{eq:B0} to $\umesh-\Imeshz \umesh$, getting 
$\Phi_\mesh(\umesh-\Imeshz \umesh)  \, \lesssim \, | \umesh-\Imeshz \umesh|_{1,\Omega}^2 $.
Then recalling \eqref{eq:bound-interp} and \eqref{eq:stab-norm},  we derive the existence of a constant $C_\Phi>1$ independent of $\mesh$ and $\umesh$, such that
$ \Phi_\mesh(\umesh-\Imeshz \umesh) \leq C_\Phi \, \Smesh(\umesh, \umesh) $. 
We obtain the desired result by choosing $\delta = \frac{\gamma}{C_\Phi}$ in \eqref{eq:bound-ST1}, and setting $C_B := C_\Phi$.

\section{Variable data}\label{sec:extensions}

We briefly consider the extension of Propositions \ref{prop:aposteriori} and \ref{local-lower-bound} to the case of variable data $\data=(A,c,f)$, while we postpone to the forthcoming paper \cite{BCNVV:22} the development of an AVEM for variable data. To this end, we denote by $\appdata=(\widehat{A},\widehat{c},\widehat{f})$ a piecewise constant approximation to $\data$. The discrete virtual problem is obtained from \eqref{def-Galerkin} by taking $A_E=\widehat{A}_E, c_E=\widehat{c}_E$ and $f_E=\widehat{f}_E$ in \eqref{eq;def-aT} and \eqref{discr-rhs}, respectively.
Similarly, we define $\etamesh^2(\umesh,\appdata)$ from \eqref{eq:estim1}-\eqref{eq:estim2}. The following result generalizes Proposition \ref{prop:aposteriori} (upper bound) to variable data.
\begin{proposition}[global upper bound]\label{prp:variable}
There exists a constant $\widehat{C}_\text{apost} >0$  depending on $\Lambda$ and $\data$, but independent of $u$, $\mesh$, $\umesh$ and $\gamma$, such that
\begin{equation}\label{eq:apost2}
\vert u - \umesh \vert_{1, \Omega}^2  \leq \widehat{C}_\text{apost} \left( \etamesh^2(\umesh,\appdata) 
+ \Smesh(\umesh, \umesh) + \dataoscmesh^2(\umesh,\appdata)  \right) 
\end{equation} 
where 
\begin{eqnarray}
&&\dataoscmesh^2(\umesh,\appdata)=\sum_{E\in\mesh} \dataoscmesh^2(E;\umesh,\appdata)\ , \nonumber\\
&&\dataoscmesh^2(E;\umesh,\appdata)=h_E^2\|f-\widehat{f}_E\|_{0,E}^2+ \| (A-\widehat{A}_E)\nabla\Pi^\nabla_E \umesh\|_{0,E}^2 +  \| (c-\widehat{c}_E)\Pi^\nabla_E \umesh\|_{0,E}^2.\nonumber
\end{eqnarray}
\end{proposition} 
\begin{proof}
We proceed as in the proof of Proposition \ref{prop:aposteriori}. We set $v_\mesh=\widetilde{\mathcal I}_\mesh^0 v$ and we get 
\begin{eqnarray}
\cB(u - \umesh,v)&=&\sum_{E\in\mesh} \left\{
(\widehat{f},v-\vmesh)_E-(\widehat{A}_E \nabla 
\Pi^\nabla_E \umesh,\nabla(v-\vmesh))_E -
(\widehat{c}_E 
\Pi^\nabla_E \umesh,v-\vmesh)_E
\right\}\nonumber\\
&& +\sum_{E\in\mesh} \left\{
((\widehat{A}_E-A)\nabla \Pi^\nabla_E \umesh,\nabla v)_E + ((\widehat{c}_E-c) \Pi^\nabla_E \umesh,v)_E
\right\}\nonumber\\
&& +\sum_{E\in\mesh} \left\{
(A\nabla (\Pi^\nabla_E-I) \umesh,\nabla v)_E + (c (\Pi^\nabla_E-I) \umesh,v)_E
\right\}\nonumber\\
&&+~\gamma  \Smesh(\umesh, \vmesh)  + (f-\widehat{
f},v)\nonumber
\end{eqnarray}
from which the thesis easily follows using standard arguments.
\end{proof}
\begin{remark}
In the a posteriori bound  \eqref{eq:apost2} we highlight the presence of the oscillation term 
$\dataoscmesh(\umesh,\appdata)$ measuring the impact of data approximation on the error. 
\end{remark}
The following is a generalization  of Proposition \ref{local-lower-bound}  to variable coefficients.
\begin{proposition}[local lower bound]\label{local-lower-bound-2}
There holds
\begin{equation}
\eta_\mesh^2(E;u_\mesh,\appdata) \lesssim \sum_{E'\in \omega_E } \left( 
\| u-u_\mesh\|^2_{1,E'} + S_{E'}(u_\mesh,u_\mesh)
 + \dataoscmesh^2(E;\umesh,\appdata)\right)
\end{equation}
where $\omega_E:=\{E^\prime: \vert \partial E \cap \partial E^\prime\vert \not=0 \}$. The hidden constant is independent of $\gamma, h,u$ and $u_\mesh$.
\end{proposition}
\begin{proof}
Using standard arguments of a posteriori analysis the thesis follows as in \cite{Cangiani-etal:2017}.
\end{proof}

With these results at hand, we can extend the validity of Corollary \eqref{Corollary:stab-free} to the variable-coefficient case, as follows.

\begin{corollary}[stabilization-free a posteriori error estimates]\label{Corollary:stab-free-2}
There exist constants $\widehat{C}_\text{apost} \geq \widehat{c}_\text{apost} >0$  depending on $\Lambda$ and $\data$, but independent of $u$, $\mesh$, $\umesh$ and $\gamma$, such that if  $\gamma$ is chosen to satisfy $\gamma^2 > \displaystyle{\frac{C_B}{\widehat{c}_\text{apost}}}$, it holds
\begin{equation}
\vert u - \umesh \vert_{1, \Omega}^2  \leq \widehat{C}_\text{apost} \, \left( \, \left(1+ C_B\gamma^{-2} \right)   \etamesh^2(\umesh,\data)  + \dataoscmesh^2(\umesh,\appdata)  \right) \,.
\end{equation} 
\begin{equation}
\left(\widehat{c}_\text{apost} - C_B \gamma^{-2} \right)  \etamesh^2(\umesh,\data)    \leq    \vert u - \umesh \vert_{1, \Omega}^2  + \dataoscmesh^2(\umesh,\appdata)   \,.
\end{equation} 
\end{corollary}

\begin{remark}
So far, we have considered homogeneous Dirichlet boundary conditions. Following \cite{Sacchi-Veeser},  it is possible to extend our analysis to the case of non-homogeneous conditions, which amounts to incorporate the oscillation of the boundary data measured in $H^{1/2}(\partial\Omega)$ into the term $\dataoscmesh^2(\umesh,\appdata)$. Since this endeavor is similar to AFEMs, we omit the technical details and refer to \cite{Sacchi-Veeser}.
\end{remark}

\section{Numerical results}\label{sec:experiments}

This section contains  a discussion on the enforcement of $\Lambda$-admissibility and two sets of numerical experiments which corroborate our theoretical findings and elucidate some computational properties of AVEM.

\subsection{Enforcing $\Lambda$-admissibility}\label{sec:make-admissible}
In the VEM framework, clearly meshes generated by {\tt REFINE} need not be conforming. However, as already noted, we require that  our meshes are $\Lambda$-admissible, i.e., the global index is uniformly bounded by $\Lambda$. We now described how this is achieved within {\tt REFINE}.

Let $\texttt{BISECTION}(\mesh, E)$ be the procedure that implements the newest-vertex bisection of an element $E \in \mesh$. 
In the first part, the algorithm \texttt{REFINE} bisects all the marked elements $E \in \mathcal{M}$  employing $\texttt{BISECTION}(\mesh, E)$. This procedure clearly may in principle generate a mesh not $\Lambda$-admissible. Assume that there exists $\widehat{\bm{x}} \in \mathcal{N}$ such that $\lambda(\widehat{\bm{x}}) = \Lambda_\mesh > \Lambda$. Since the input mesh $\mesh$ of \texttt{REFINE} is $\Lambda$-admissible, necessarily  $\lambda(\widehat{\bm{x}}) = \Lambda+1$. 
Furthermore $\lambda(\widehat{\bm{x}}) > 0$, therefore $\widehat{\bm{x}}$ is an hanging node for an element (say) $\widehat{E}$ (see Fig. \ref{fig:ref}).
In order to restore the $\Lambda$-admissibility of the mesh we need to refine $\widehat{E}$. Two possible cases arise.
\texttt{Case A}:   
if the node $\widehat{\bm{x}}$ belongs to the opposite edge to the newest-vertex of $\widehat{E}$, then $\texttt{BISECTION}(\mesh, \widehat{E})$ immediately yields $\lambda(\widehat{\bm{x}}) \leq \Lambda$  (see Fig. \ref{fig:ref}).
\texttt{Case B}:   
if the node $\widehat{\bm{x}}$ belongs to the same edge of the newest-vertex of $\widehat{E}$, then in order to reduce the global index of $\widehat{\bm{x}}$ we need to refine twice the element $\widehat{E}$, namely after the first bisection, the new element (say) $\widehat{E}'$ having $\widehat{\bm{x}}$ as hanging node is further bisected (see Fig. \ref{fig:ref}). 
Notice that in the latter case the procedure can create a new node (the blue node in Fig.\ref{fig:ref}) possibly having global index greater than $\Lambda$. However in \cite{BCNVV:22} we prove that the algorithm \texttt{REFINE} here detailed is optimal in terms of degrees of freedom,  very much in the spirit of the completion algorithm for
conforming bisection meshes by Binev, Dahmen, and DeVore \cite{BDD:04}; see also 
\cite{NSV:09,NochettoVeeser:12,Stevenson:08}.

\begin{figure}[h!]
\begin{center}
\begin{overpic}[scale=0.2]{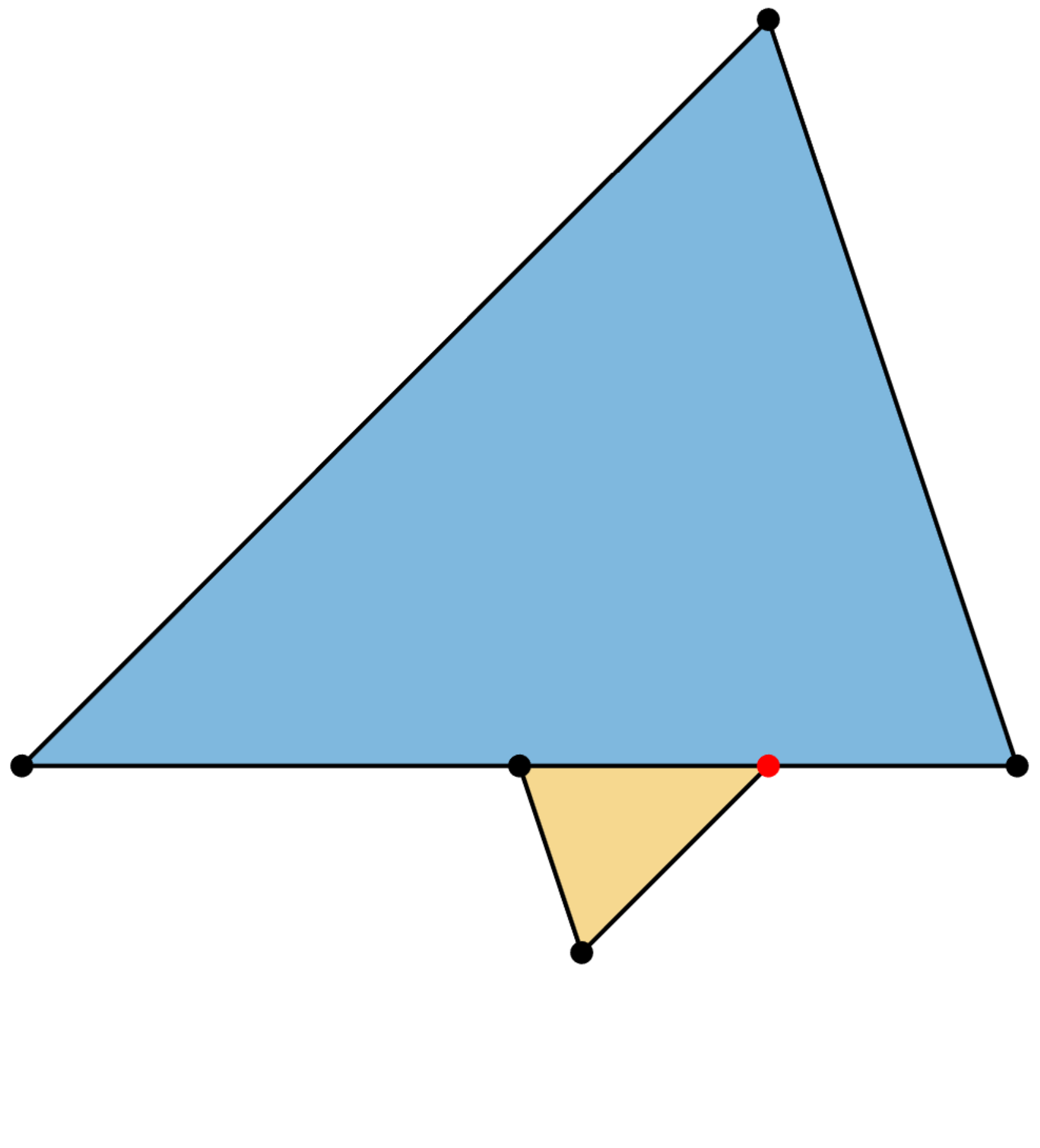}
\put(50,60){$\widehat{E}$}
\put(65,35){$\widehat{\bm{x}}$}
\end{overpic}
\qquad \quad
\begin{overpic}[scale=0.2]{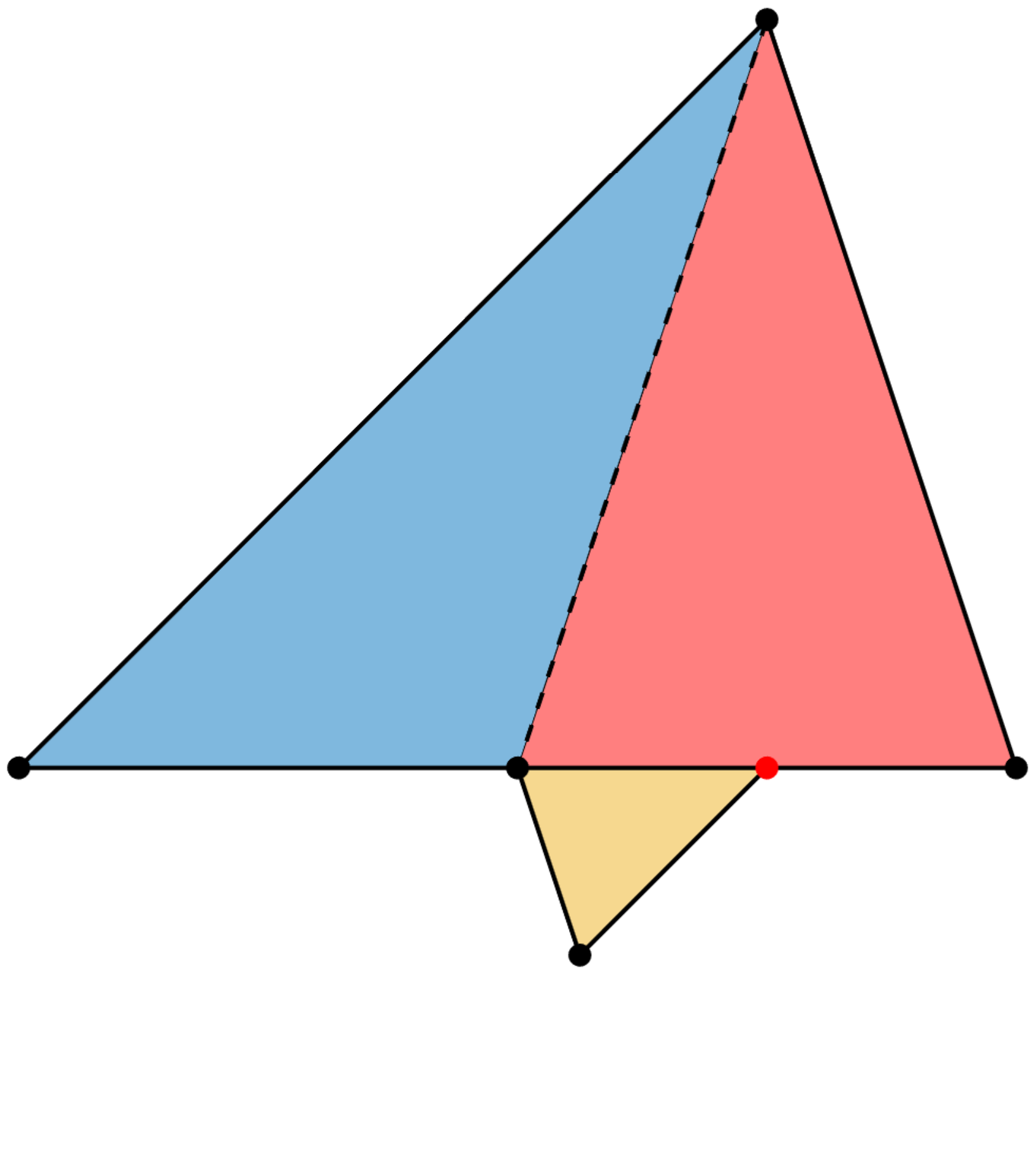}
\put(60,100){$\bm{nv}(\widehat{E})$}
\put(50,60){$\widehat{E}$}
\put(65,35){$\widehat{\bm{x}}$}
\put(40,7){\texttt{Case A}}
\end{overpic}
\qquad \quad
\begin{overpic}[scale=0.2]{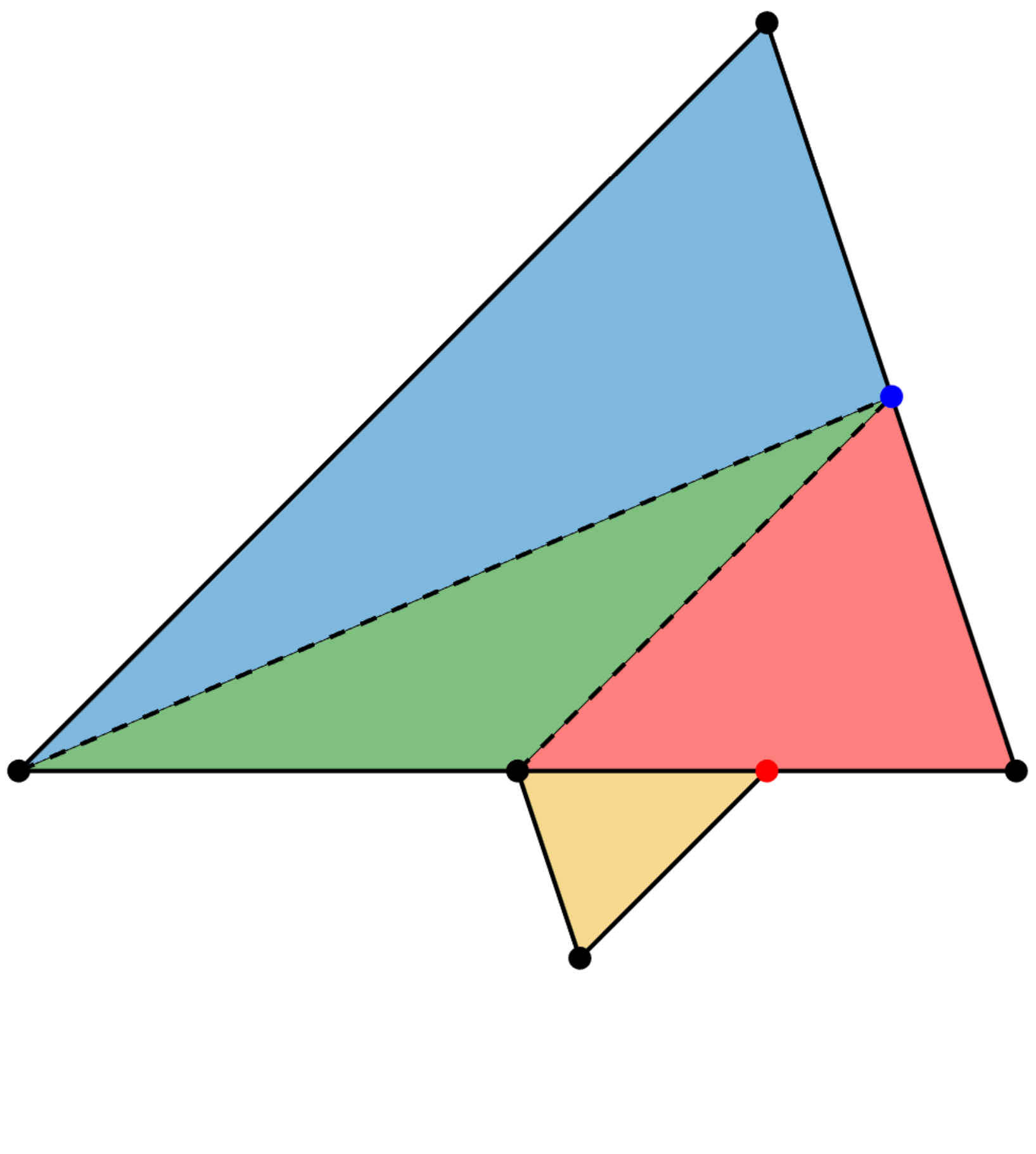}
\put(-10,25){$\bm{nv}(\widehat{E})$}
\put(50,60){$\widehat{E}$}
\put(65,35){$\widehat{\bm{x}}$}
\put(45,42){$\widehat{E}'$}
\put(40,7){\texttt{Case B}}
\end{overpic}
\end{center}
\caption{
Left: mesh element $\widehat{E}$ having $\widehat{\bm{x}}$ as hanging node, with $\widehat{\bm{x}}$ such that $\lambda(\widehat{\bm{x}})= \Lambda_\mesh > \Lambda$.
Middle: \texttt{case A} the node $\widehat{\bm{x}}$ belongs to the opposite edge to the newest-vertex of $\widehat{E}$, therefore one bisection is needed.
Right: \texttt{case B} the node $\widehat{\bm{x}}$ belongs to the same edge of the newest-vertex of $\widehat{E}$, therefore two bisections are needed.}
\label{fig:ref}
\end{figure}
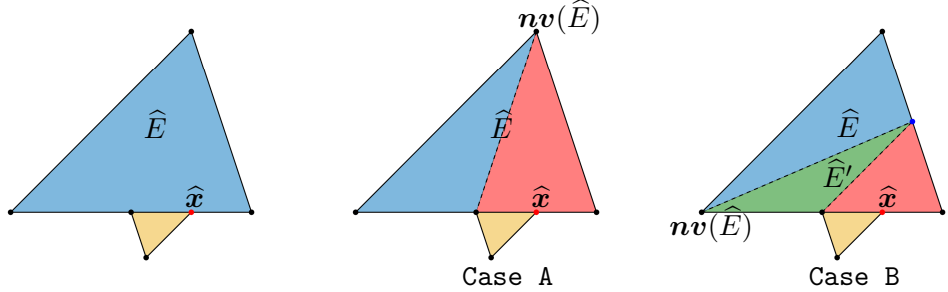

The modules $\texttt{REFINE}$ and $\texttt{MAKE\_ADMISSIBLE}$ consist of the following steps:

\medskip
\begin{algotab}
\label{eq:refine}
  \>  $[\meshs]=\texttt{REFINE}(\mesh, \mathcal{M}, \Lambda)$
  \\
  \>  $\quad  \text{for } E \in \mathcal{M}$ 
  \\
  \>  $\quad \quad  [\mesh] = \texttt{BISECTION}(\mesh, E)$ 
  \\
  \>  $\quad \text{end for }$
  \\
  \>  $\quad \text{while } \Lambda_\mesh > \Lambda$ 
  \\
  \>  $\quad \quad \widehat{E} = \text{element having as hanging node a vertex with maximum global index $\widehat{x}$}$
  \\
  \>  $\quad \quad [\mesh] = \texttt{MAKE\_ADMISSIBLE}(\mesh, \widehat{E})$ 
  \\
  \>  $\quad  \text{end while}$
  \\
  \>  $\quad  \text{return}(\mesh)$
\end{algotab}
\medskip

\medskip
\begin{algotab}
\label{eq:make_admissibile}
  \>  $[\meshs] = \texttt{MAKE\_ADMISSIBLE}(\mesh, \widehat{E})$
  \\
  \>  $\quad  \text{if \texttt{Case A}}$ 
  \\
  \>  $\quad \quad  [\mesh] = \texttt{BISECTION}(\mesh, \widehat{E})$ 
  \\
  \>  $\quad \text{else}$
  \\
  \>  $\quad \quad  [\mesh] = \texttt{BISECTION}(\mesh, \widehat{E})$ 
  \\
  \>  $\quad \quad \widehat{E}' = \text{element having as hanging node a vertex with maximum global index $\widehat{x}$}$
  \\
  \>  $\quad \quad  [\mesh] = \texttt{BISECTION}(\mesh, \widehat{E}')$ 
  \\
  \>  $\quad  \text{end if}$
  \\
  \>  $\quad  \text{return}(\mesh)$
\end{algotab}
\medskip

\subsection{Test 1: Control of stabilization by the estimator}\label{sec:numerics-1}

The aim of this first test is to confirm the theoretical predictions of Proposition \ref{prop:bound-ST} and in particular to assess the sharpness of inequality \eqref{eq:bound-ST}.
To this end we solve a Poisson problem \eqref{eq:pde} combining the VEM setting \eqref{def-Galerkin} with the adaptive algorithm {\tt GALERKIN} described in Section \ref{sub:galerkin}.
We run {\tt GALERKIN} and iterate the loop \eqref{eq:paradigm} with the following stopping criterion based on the total number of degrees of freedom $\texttt{NDoFs}$, rather than a tolerance $\varepsilon$,
\begin{equation}
\label{eq:stopping}
\texttt{NDoFs} \geq \texttt{N\_Max} \,.
\end{equation}

In the present test we consider the L-shaped domain $\Omega = (-1, 1) \times (-1, 1) \setminus [0, 1] \times [-1, 0]$ and solve the Poisson problem \eqref{eq:pde} with $A=I$, $c=0$, $f=1$ and vanishing boundary conditions. The exact solution has a singular behaviour at the re-entrant corner.
In the test we adopt the loop \eqref{eq:paradigm} with D{\"o}rfler parameter $\theta = \texttt{0.5}$, and $\texttt{N\_Max} = \texttt{2000}$, furthermore we pick $\Lambda = \texttt{10}$. 
We adopt the \texttt{dofi-dofi} stabilization \eqref{eq:stab-dofidofi}.
To assess the effectiveness of bound \eqref{eq:bound-ST} we consider the following quantity
\[
\texttt{ratio} := \frac{\gamma^2 \Smesh(u_\mesh, u_\mesh)}{\etamesh^2(\umesh, \data) } \,.
\]
In Fig. \ref{fig:test-gamma} we display the quantity \texttt{ratio} for different values of the stabilization parameter $\gamma$ obtained with the adaptive algorithm \eqref{eq:paradigm}. 
Notice that for all the proposed values of $\gamma$ the quantity \texttt{ratio} at the first iteration of the algorithm is zero since the starting mesh $\mesh = \mesh_0$ is made of triangular elements consequently $\Smesh(u_\mesh, u_\mesh) =0$.
Fig. \ref{fig:test-gamma} shows that the estimate of Proposition \ref{prop:bound-ST} is sharp and for the proposed problem the constant $C_B$ in \eqref{eq:bound-ST} is bounded from above by $\texttt{0.1}$.

\begin{figure}[!hbt]
\begin{center}
\includegraphics[scale=0.3]{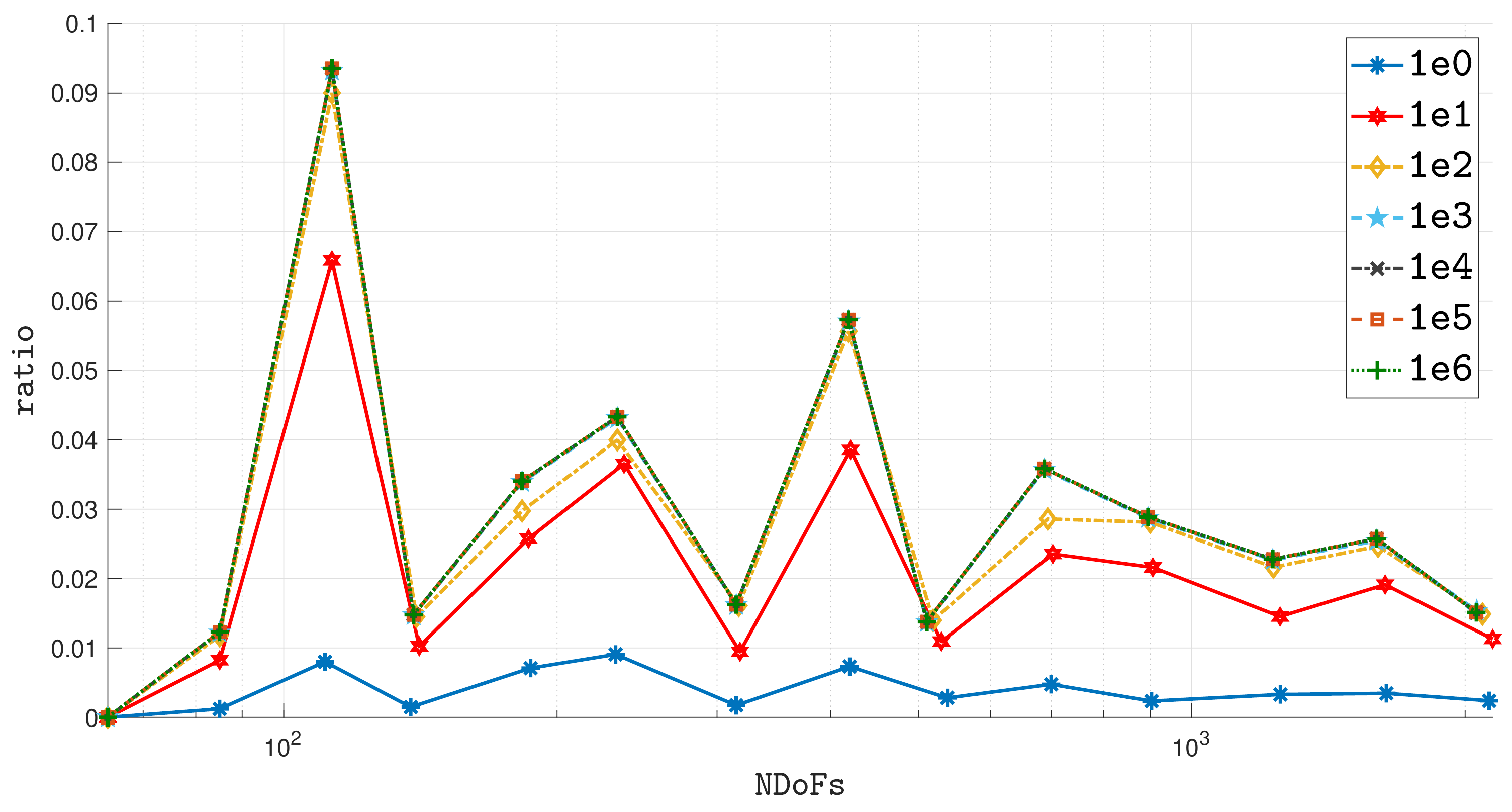}
\end{center}
\caption{Test 1. Sharpness of bound \eqref{eq:bound-ST}: \texttt{ratio} between the term $\gamma^2 \Smesh(u_\mesh, u_\mesh)$ and the term $\etamesh^2(\umesh, \data)$ obtained with the adaptive algorithm \eqref{eq:paradigm}.}
\label{fig:test-gamma}
\end{figure}

\subsection{Test 2: Kellogg's checkerboard pattern}\label{sec:numerics-2}

The purpose of this numerical experiment is twofold. First, we again confirm the theoretical results in Proposition \ref{prop:bound-ST} and Corollary \ref{Corollary:stab-free}. Second, we discuss the practical performance of {\tt GALERKIN} with a rather demanding example and compare it with the corresponding AFEM.
In order to compute the VEM error between the exact solution $u_{\rm ex}$ and
the VEM solution $u_\mesh$, we consider the computable $H^1$-like error quantity:
\[
\texttt{H\textasciicircum 1-error} := 
\frac{\left(\sum_{E \in \mesh} \|\nabla (u_{\rm ex} - \PiE u_\mesh)\|_{0, E}^2 \right)^{1/2}}
{\|\nabla u_{\rm ex}\|_{0, \Omega}} \,.
\]
Notice that \texttt{H\textasciicircum 1-error} (as well as the discrete problem \eqref{def-Galerkin}) depends only on the DoFs values of the discrete solution $u_\mesh$, hence, it is independent of the choice of the VEM space $\VE$ in \eqref{eq:def-VT}.
If the mesh $\mesh$ does not contain hanging nodes, obviously $\texttt{H\textasciicircum 1-error}$ coincides with the `true' $H^1$-relative error.
In the numerical test we use the \texttt{dofi-dofi} stabilization \eqref{eq:stab-dofidofi} with stabilization parameter $\gamma = \texttt{1}$ (cf. \eqref{eq:def-BT}), we pick $\Lambda = \texttt{10}$ (cf. Definition \ref{def:Lambda-partitions}) and  D{\"o}rfler parameter $\theta = \texttt{0.5}$ (cf. \eqref{eq:dorfler}), whereas the stopping parameter $\texttt{N\_Max}$ (cf. \eqref{eq:stopping})  will be specified later.

We consider from \cite[Example 5.3]{MorinNochettoSiebert:2000} the Poisson problem \eqref{eq:pde} with piecewise constant coefficients and vanishing load with the following data:
$\Omega = (-1, 1)^2$,  $A = a I$, with $a= \texttt{161.4476387975881}$ in the first and third quadrant and $a = \texttt{1}$ in the second and fourth quadrant, $c = 0$ and  $f = 0$.
According to Kellogg formula \cite{kellogg} the exact solution is given in polar coordinates by  
$u_{\rm ex}(r, \alpha) = r^\delta \nu(\alpha)$ where
\[
\nu(\alpha) := \left \{
\begin{array}{lll}
 \cos\left( \left(\frac{\pi}{2} - \sigma \right) \delta \right)  
  \cos\left( \left(\alpha - \frac{\pi}{2} + \rho \right) \delta \right) 
& \qquad
& \text{if $0 \leq \alpha \leq \pi/2$,}
\\
 \cos\left( \rho \delta \right)  
  \cos\left( \left(\alpha - \pi + \sigma \right) \delta \right) 
& \qquad 
& \text{if $\pi/2 \leq \alpha \leq \pi$,}
\\
 \cos\left( \sigma \delta \right)  
  \cos\left( \left(\alpha - \pi - \rho \right) \delta \right) 
& \qquad
& \text{if $\pi \leq \alpha \leq 3\pi/2$,}
\\
 \cos\left( \left(\frac{\pi}{2} - \rho \right) \delta \right)  
  \cos\left( \left(\alpha - \frac{3\pi}{2} - \sigma \right) \delta \right)
& \qquad
& \text{if $3\pi/2 \leq \alpha \leq 2\pi$,} 
\end{array}
\right .
\]
and where the numbers $\delta$, $\rho$, $\sigma$ satisfy suitable nonlinear relations. In particular we pick 
$ \delta = \texttt{0.1}$,  $\rho =  \texttt{pi/4}$ and $\sigma = \texttt{-14.92256510455152}$.
Notice that the exact solution $u_{\rm ex}$ is in the Sobolev space $H^{1+\epsilon}$ only for $\epsilon < \texttt{0.1}$ and thus is very singular at the origin.

In Fig. \ref{fig:test2-err_est_stab} (plot on the left) we display $\texttt{H\textasciicircum 1-error}$, the estimator $\etamesh(\umesh,\data)$ and the stabilization term $\Smesh(\umesh, \umesh)^{1/2}$ obtained with the adaptive algorithm \eqref{module:_GAL} with stopping parameter \texttt{N\_Max} = \texttt{25000}. The predictions of Proposition \ref{prop:bound-ST} and Corollary \ref{Corollary:stab-free} are confirmed: 
the estimator bounds from above both the energy error and the stabilization term. Furthermore, one can appreciate that, after a fairly long transient due to the highly singular structure of the solution, the error decay reaches asymptotically the 
theoretical optimal rate $\texttt{NDoFs\textasciicircum - 0.5}$ (whereas the estimator  decays with this rate along the whole refinement history).
\begin{figure}[!htb]
\begin{center}
\includegraphics[width=75mm,height=50mm]{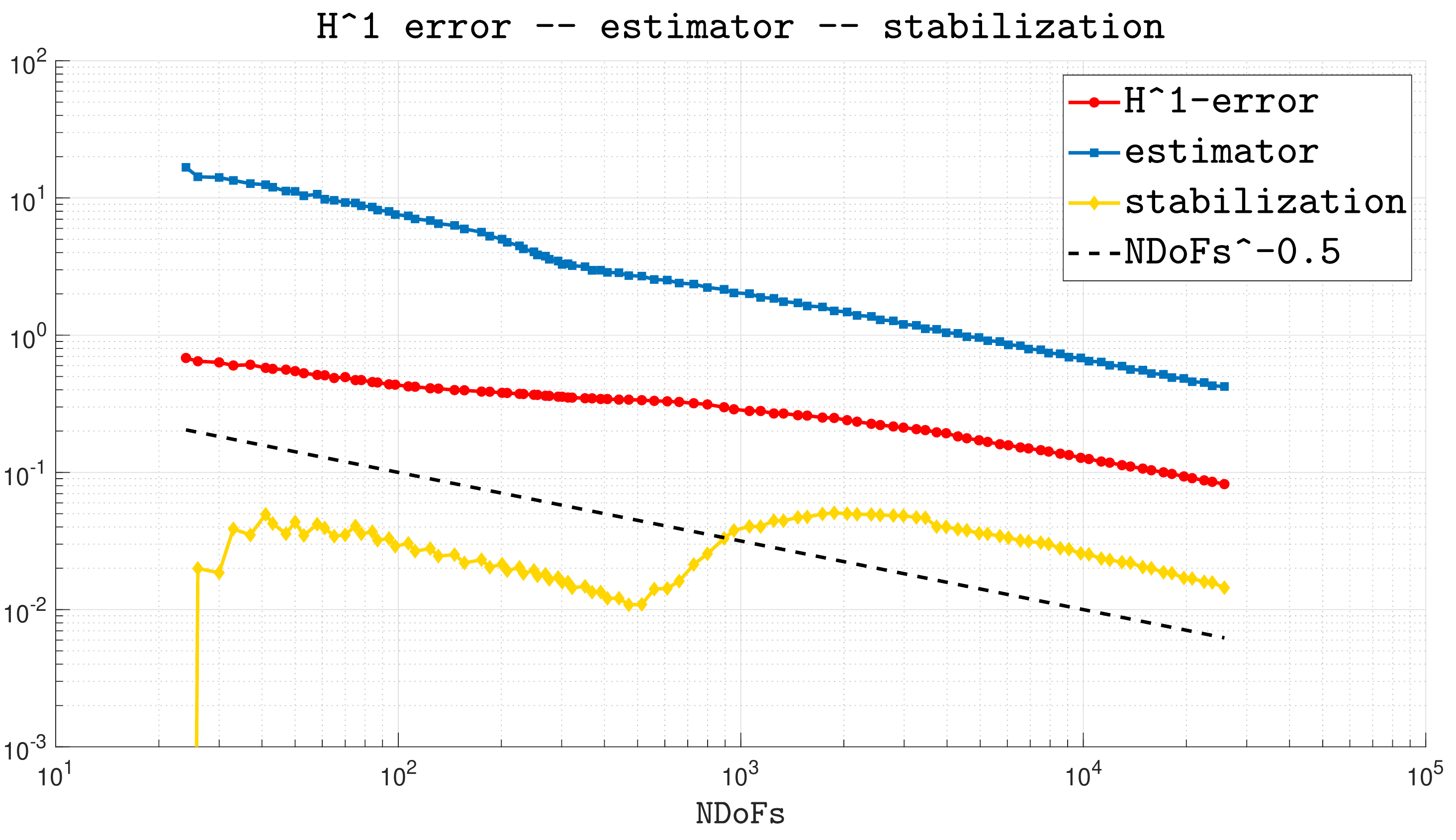}
\includegraphics[width=75mm,height=50mm]{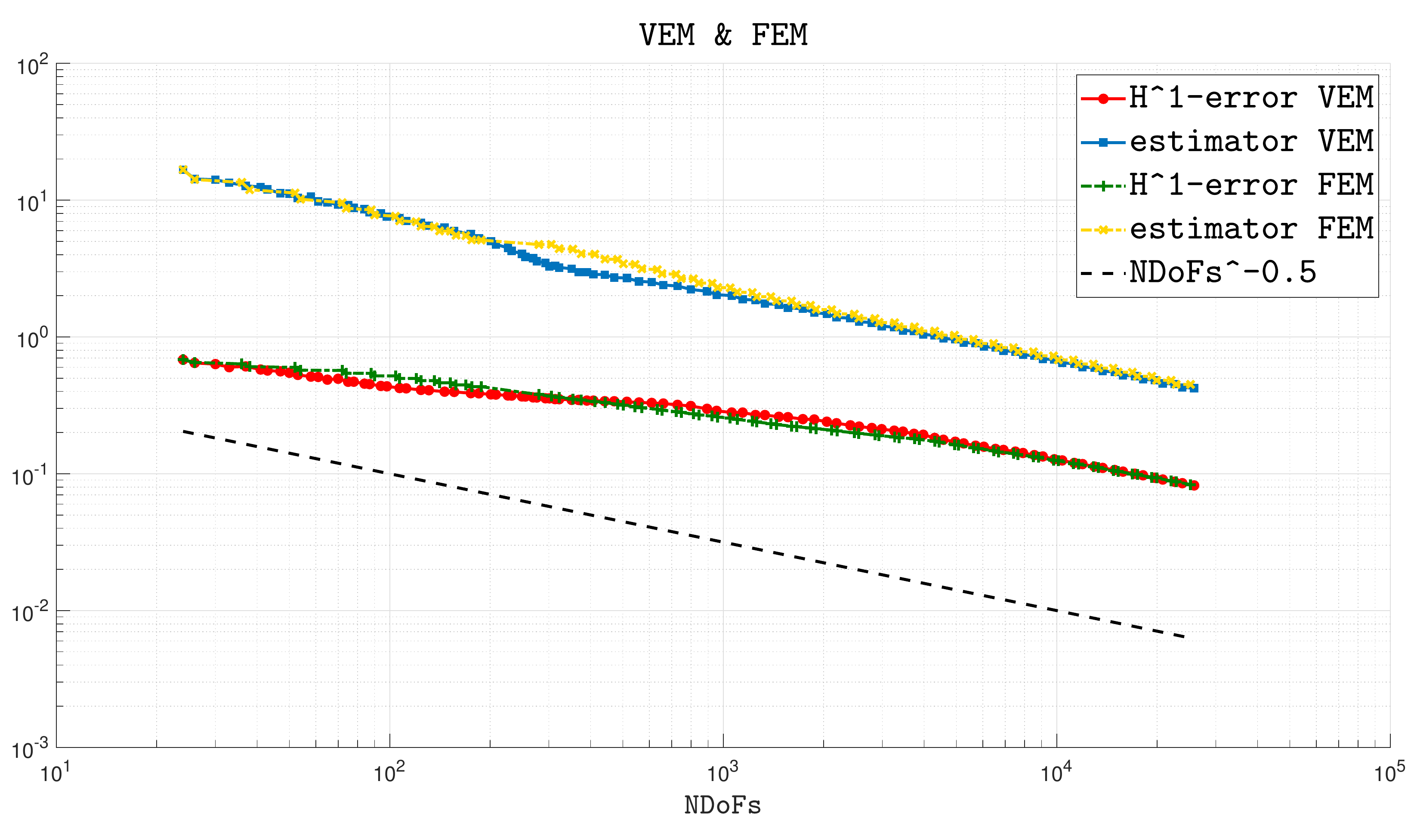}
\end{center}
\caption{{Test 2. Left: $\texttt{H\textasciicircum 1-error}$, \texttt{estimator} $\etamesh(\umesh,\data)$, \texttt{stabilization} term $\Smesh(\umesh, \umesh)^{1/2}$. Right: $\texttt{H\textasciicircum 1-error}$ and \texttt{estimator} $\etamesh(\umesh,\data)$ obtained with \texttt{VEM} and \texttt{FEM}.
In both figures the optimal decay is indicated by the dashed line with slope $\texttt{-0.5}$.}}
\label{fig:test2-err_est_stab}
\end{figure}

\noindent
To validate the practical performances of the proposed numerical scheme \eqref{def-Galerkin}, we compare the results obtained with our VEM to those obtained with a standard $\mathbb{P}_1$ FEM, implemented as a VEM with $\Lambda = \texttt{0}$ (cf. Definition \ref{def:Lambda-partitions}). The results for both methods are obtained with an ``in-house'' code, yet the FEM outcomes coincide (up to machine precision) with those obtained with the code developed in \cite{matlab}.
In Fig. \ref{fig:test2-err_est_stab} (plot on the right) we display  $\texttt{H\textasciicircum 1-error}$ and \texttt{estimator} $\etamesh(\umesh,\data)$ obtained with VEM and FEM coupled with the adaptive algorithm \eqref{eq:paradigm} and stopping parameter $\texttt{N\_Max} = \texttt{25000}$. Notice that for FEM $\texttt{H\textasciicircum 1-error}$ is the ``true'' $H^1$-relative error. 
Both methods yield very similar results in terms of behaviour of the error and the estimator. 

However, a deeper analysis shows important differences between VEM and FEM approximations in terms of the final grids denoted respectively with $\mesh_{\texttt{VEM}}$ and $\mesh_{\texttt{FEM}}$. In Fig. \ref{fig:test2 mesh global} we display the meshes  $\mesh_{\texttt{VEM}}$ and $\mesh_{\texttt{FEM}}$ obtained with stopping parameter $\texttt{N\_Max} = \texttt{5000}$. 
The number of nodes \texttt{N\_vertices} and elements \texttt{N\_elements} are 
$
\texttt{N\_vertices}(\mesh_{\texttt{VEM}}) = \texttt{5259}
$, 
$
\texttt{N\_elements}(\mesh_{\texttt{VEM}}) = \texttt{8725}
$, 
$
\texttt{N\_vertices}(\mesh_{\texttt{FEM}}) = \texttt{5070}
$, 
$
\texttt{N\_elements}(\mesh_{\texttt{FEM}}) = \texttt{10094}
$, 
i.e. the mesh $\mesh_{\texttt{FEM}}$ has \texttt{16\%} more elements than the mesh $\mesh_{\texttt{VEM}}$.
Furthermore the number of polygons in $\mesh_{\texttt{VEM}}$ (elements with more than three vertices) is 
\texttt{1653}: \texttt{1563} quadrilaterals, \texttt{86} pentagons, \texttt{2} hexagons, 
\texttt{1} heptagon and \texttt{1} nonagon.    
\begin{figure}[!htb]
\begin{center}
\includegraphics[scale=0.3]{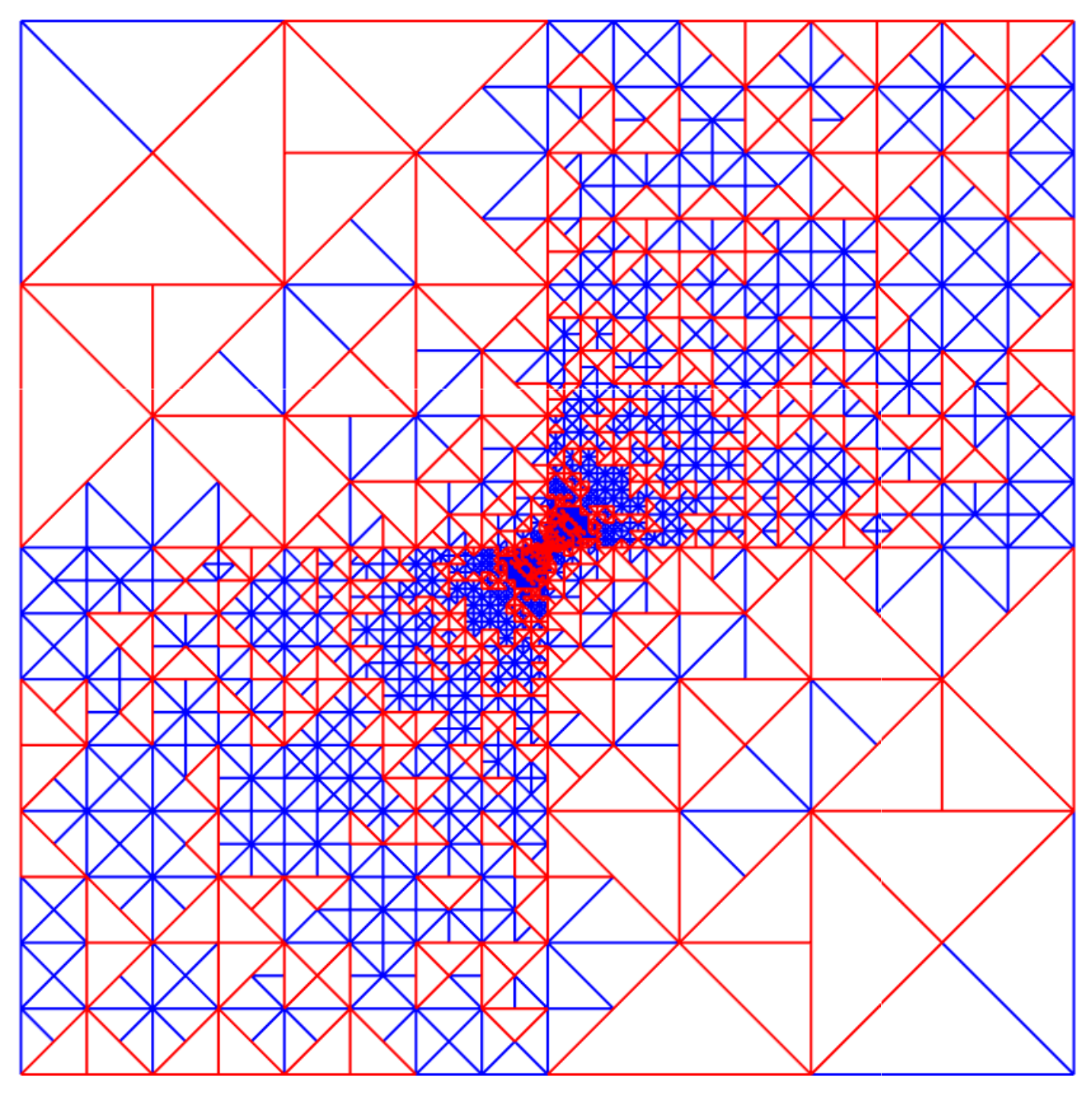}
\qquad
\includegraphics[scale=0.3]{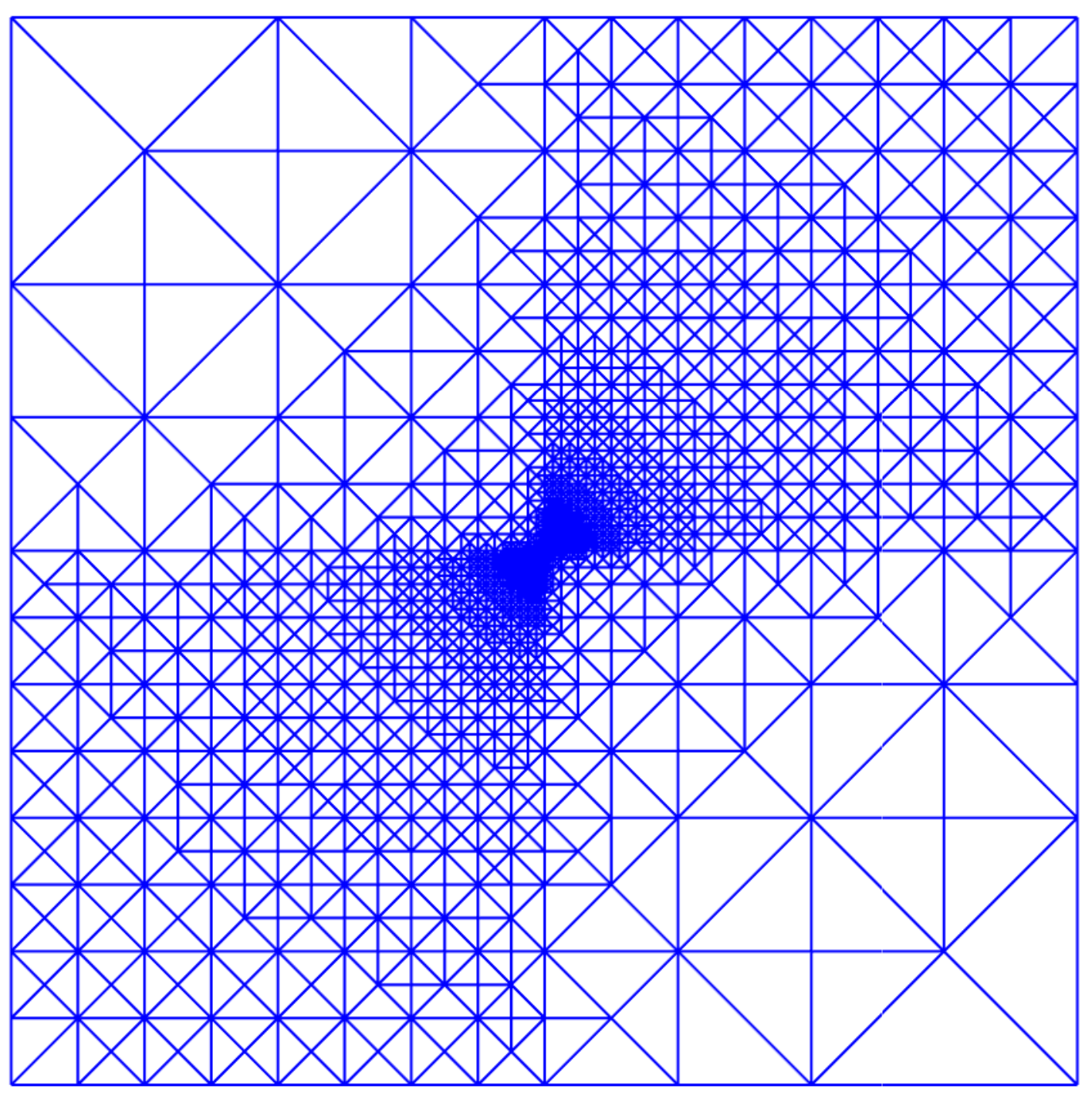}
\caption{{Test 2. Left: final grid $\mesh_{\texttt{VEM}}$ obtained with \texttt{VEM}. 
Right: final grid $\mesh_{\texttt{FEM}}$ obtained with \texttt{FEM}. Mesh elements having more than three vertices are drawn in red.}}
\label{fig:test2 mesh global}
\end{center}
\end{figure}

The grids $\mesh_{\texttt{VEM}}$ and $\mesh_{\texttt{FEM}}$ are highly graded at the origin along the bisector of the first and third quadrants. However from Fig. \ref{fig:test2 mesh 1em9} we can appreciate how grid $\mesh_{\texttt{VEM}}$ still exhibits a rather strong grading also for the zoom scaled to $10^{-9}$, thus revealing the singularity structure much better than $\mesh_{\texttt{FEM}}$.

\begin{figure}[!htb]
\begin{center}
\includegraphics[scale=0.3]{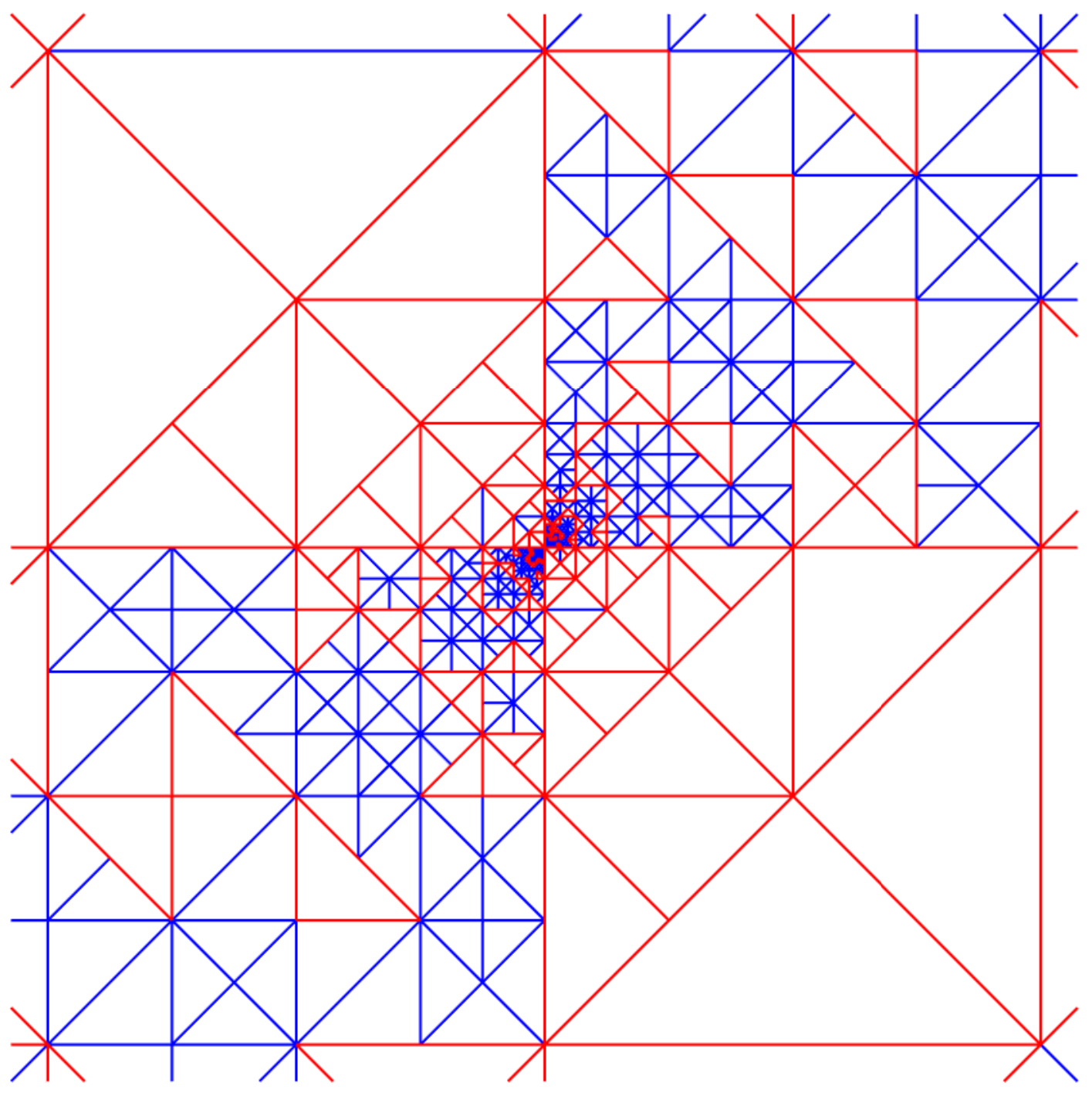}
\qquad
\includegraphics[scale=0.3]{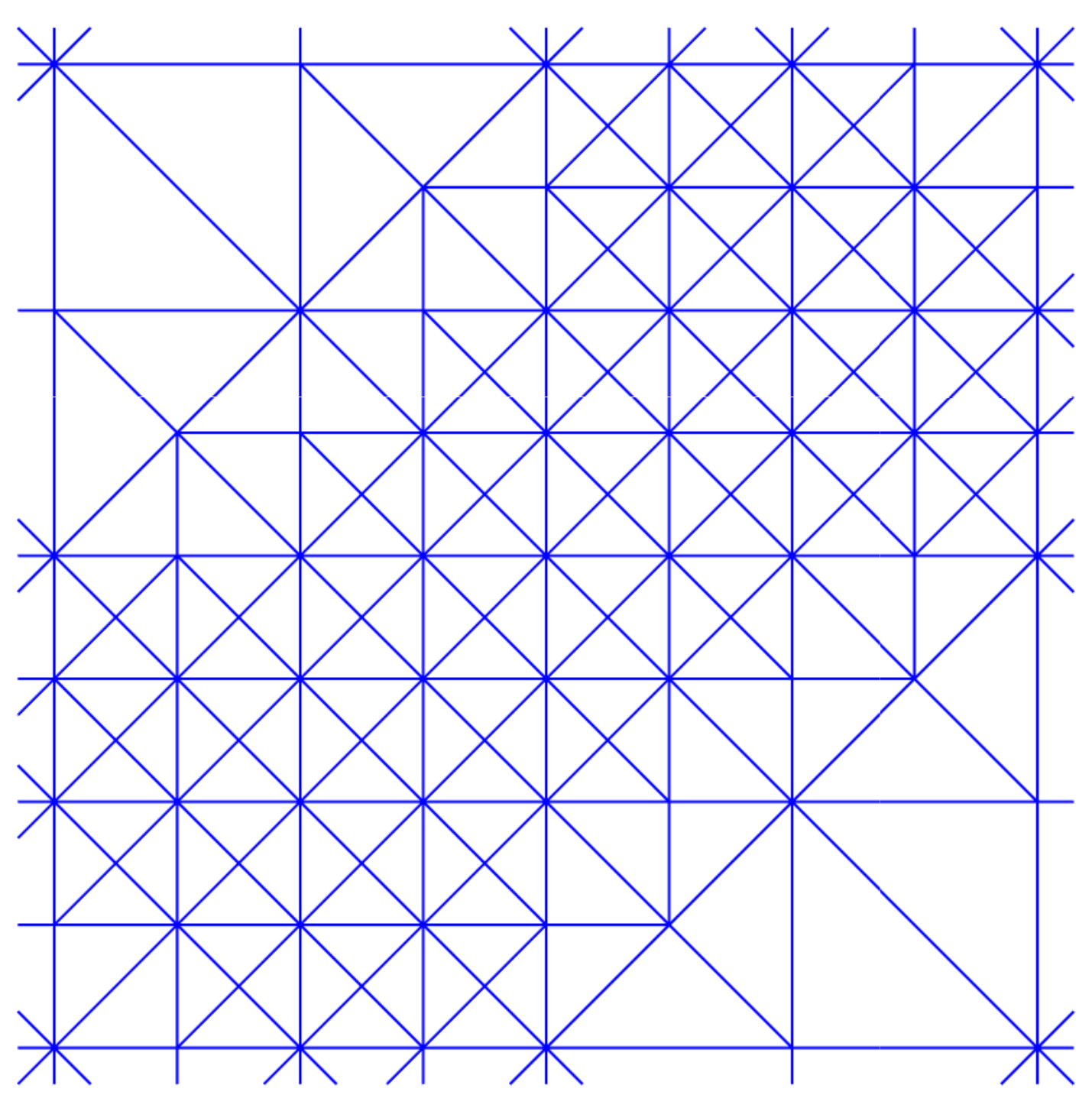}
\caption{{Test 2. Left: final grid $\mesh_{\texttt{VEM}}$.  Right: final grid $\mesh_{\texttt{FEM}}$. Zoom to $(-10^{-9}, 10^{-9})^2$. Mesh elements having more than three vertices are drawn in red.}}
\label{fig:test2 mesh 1em9}
\end{center}
\end{figure}

Finally, in Fig. \ref{fig:test 2 dettaglio} we plot the zoom to $(-10^{-10}, 10^{-10})^2$ for the grid 
$\mesh_{\texttt{VEM}}$ and the plot of the discrete solution for the finer grid. We highlight the presence of the nonagon with two nodes having global index $\lambda = $ \texttt{3}. It is worth noting that the largest global index is $\lambda=\texttt{3}$, whence the threshold $\Lambda = \texttt{10}$ is never reached by the module \texttt{REFINE}. Therefore, the condition $\lambda\leq\Lambda$ is not restrictive in practice.

\begin{figure}[!htb]
\begin{center}
\includegraphics[scale=0.31]{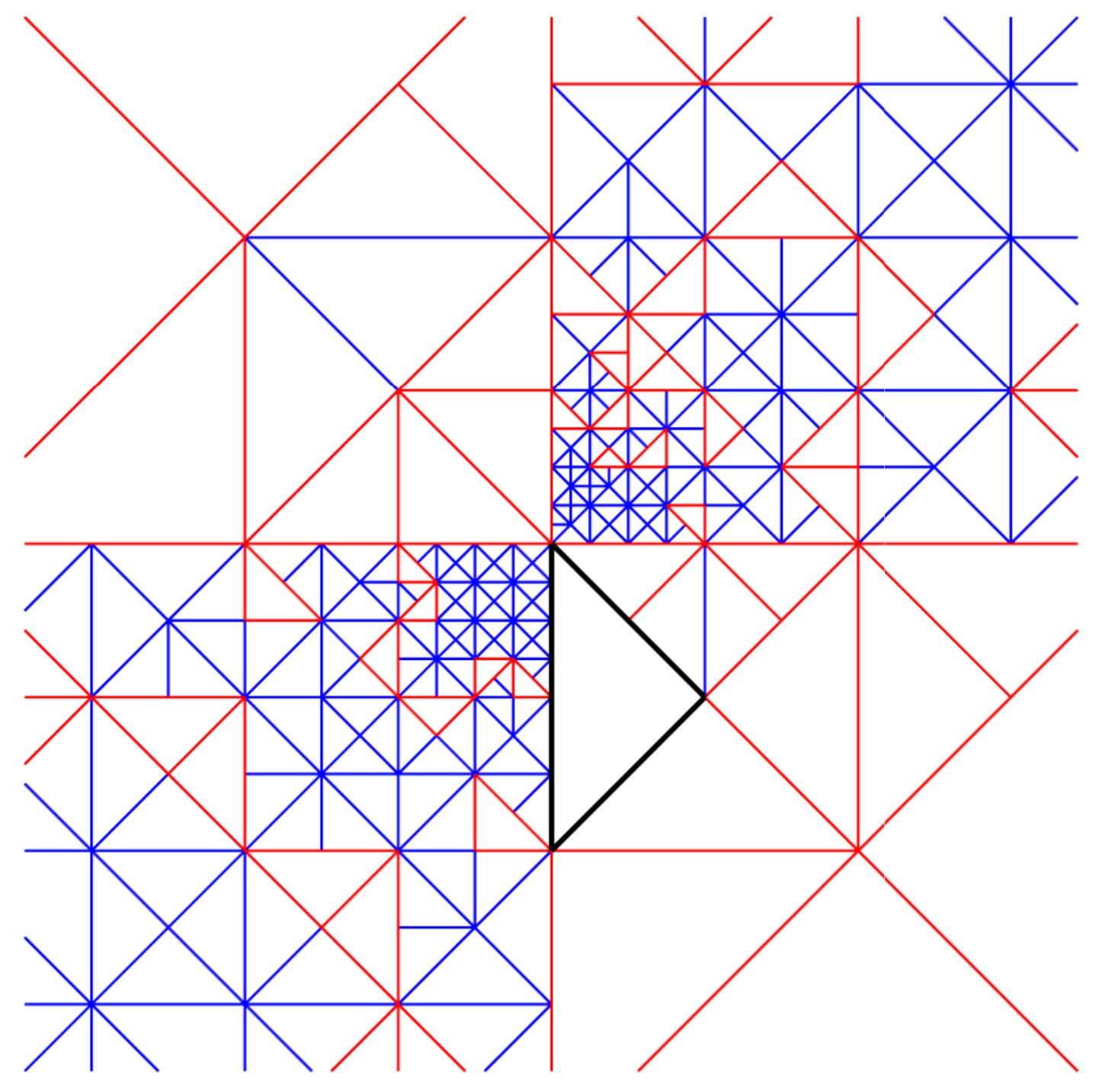}
\hspace{5mm}
\includegraphics[scale=0.28]{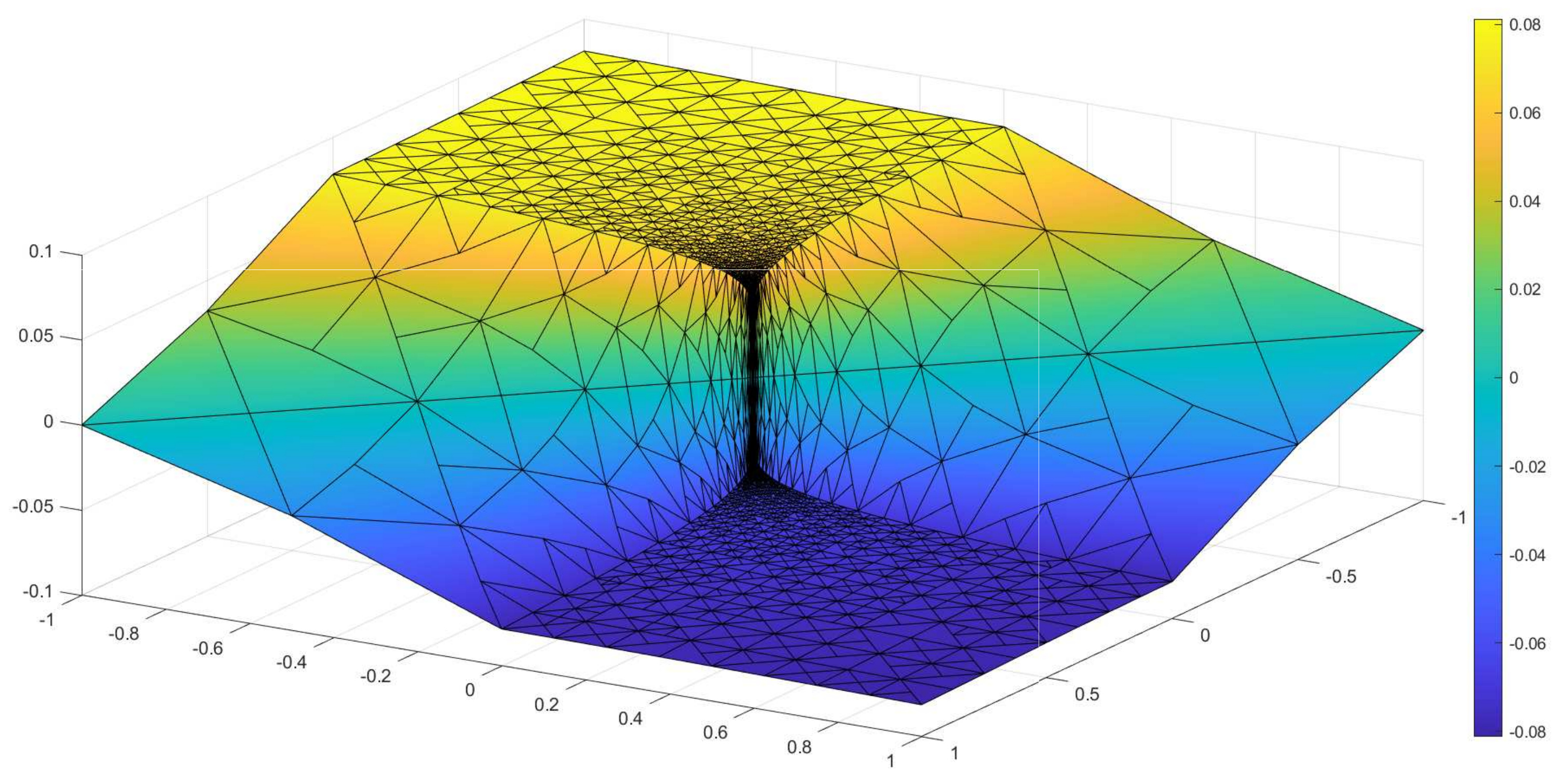}
\end{center}
\caption{{Test 2. Left: final grid $\mesh_{\texttt{VEM}}$, zoom to $(-10^{-10}, 10^{-10})^2$
(mesh elements having more than three vertices are drawn in red, and the mesh element drawn in black is a nonagon).
Right: graph of the discrete solution.
}}
\label{fig:test 2 dettaglio}
\end{figure}

\section{Conclusions}

The analysis of this paper relies crucially on the existence of a subspace $\Vmeshz \subseteq \Vmesh$ satisfying the following properties:
\begin{enumerate}[\quad$\bullet$]
\item The discrete forms satisfy the consistency property \eqref{eq:propB1} on $\Vmeshz$;

\item There exists a subset ${\cal P}$ of mesh nodes such that the collection of linear operators 
$v \rightarrow \{ v(\bm{x}) \}_{\bm{x}\in {\cal P}}$ constitutes a set of degrees of freedom for $\Vmeshz$;

\item Propositions \ref{prop:scaledPoincare} and \ref{prop:compareInterp} hold for the above choice of $\Vmeshz$ and ${\cal P}$.

\end{enumerate}

We have established these properties for meshes made of triangles, which is the most common situation in finite element methods. Yet,
there are other cases in which the above construction can be easily applied. One notable example is that of square meshes, where we assume a standard quadtree element refinement procedure that subdivides each square into four squares; in this framework the advantage of allowing hanging nodes is evident.
In such case, the space $\Vmeshz$ is chosen as
$$
\Vmeshz = \big\{ v \in \Vmesh \textrm{ such that } v|_E \in \mathbb{Q}_1(E) \ \forall E \in \mesh \big\} \, ,
$$
where $\mathbb{Q}_1$ denotes the space of bilinear polynomials, and ${\cal P}$ is the set of proper (i.e., non-hanging) nodes of the mesh. It is not difficult to adapt to the new framework the arguments given in the paper, and prove the validity of the three conditions above, thereby arriving at the same conclusions obtained for triangles. Obviously, heterogeneous meshes formed by triangles and squares could be handled as well.

The extension of the techniques presented above to general polygonal meshes (for which even a deep understanding of the refinement strategies is currently missing) seems highly non-trivial. However, we believe that the main results of this paper, namely the bound of the stabilization term by the a posteriori error estimator and the contraction property of the proposed adaptive algorithm, should hold in a wide variety of situations. We also hope that some of the ideas that we have elaborated here will turn useful to attack the challenge of providing a sound mathematical framework to adaptive virtual element methods (AVEM) in a more general setting. 

Furthermore, the results presented herein will serve as a basis in the sequel paper \cite{BCNVV:22}, in which we will design and analyze a two-step AVEM (still on triangular partitions admitting hanging nodes) able to handle variable coefficients. The primary goal of \cite{BCNVV:22} is to develop a complexity analysis.

\bigskip
\begin{center}
{\bf Acknowledgements}
\end{center}
\noindent
LBdV, CC and MV where partially supported by the Italian MIUR through the PRIN grants n. 201744KLJL (LBdV, MV) and n. 201752HKH8 (CC).
CC was also supported by the DISMA Excellence Project (CUP: E11G18000350001). RHN has been supported in part by NSF grant DMS-1908267. These supports are gratefully acknowledged. LBdV, CC, MV and GV are members of the INdAM research group GNCS.

\bibliographystyle{plain}
\bibliography{biblio}

\end{document}